\documentclass{amsart}
  \usepackage[fontsize=10pt]{fontsize}

  \usepackage[a4paper, left=1.5cm, right=1.5cm, top=1.3cm,bottom=1.4cm]{geometry} 

\usepackage{amsmath,amsthm} 
\usepackage{latexsym,amssymb}  
\usepackage{amsfonts,mathrsfs} 
\usepackage{stmaryrd} 

\usepackage{color}
\usepackage{graphicx}
\usepackage{epstopdf} 
\usepackage{indentfirst} 
\usepackage{url} 
\usepackage[linktocpage=true, hidelinks, colorlinks, linkcolor=blue, citecolor=magenta, bookmarksopen=true, urlcolor=brown]{hyperref}

\usepackage{enumitem}
 \usepackage[normalem]{ulem} 
 \usepackage{verbatim, comment} 
 \usepackage{svn, mathtools, mathscinet}  

\usepackage[all]{xy} 
\usepackage{tikz-cd, amscd}
 \usepackage{array}
 \usepackage{tabu} 
 
 \label{Theorems environ}
 \newtheorem{thm}{Theorem}[section]
\newtheorem{theorem}[thm]{Theorem}

\newtheorem{prop}[thm]{Proposition}

\newtheorem{lemma}[thm]{Lemma}

\newtheorem{cor}[thm]{Corollary}

\newtheorem{lem-def}[thm]{Lemma-Definition}

\theoremstyle{definition}
\newtheorem{defn}[thm]{Definition}

\newtheorem{rem}[thm]{Remark}
\newtheorem{remark}[thm]{Remark}

 \newtheorem{construction}[thm]{Construction}

\newtheorem{axiom}[thm]{Axiom}

\newtheorem{example}[thm]{Example}

\newtheorem{notation}[thm]{Notation}
\newtheorem{Notation}[thm]{Notation}

\newtheorem{convention}[thm]{Convention}

\numberwithin{equation}{section}
\label{••••••••••3}
\label{From Min-Wang}

\newcommand{\Zp}{{\bZ_p}}

\newcommand{\Qp}{{\bQ_p}}

\newcommand{\Gal}{{\mathrm{Gal}}}           

\newcommand{\an}{{\mathrm{an}}}             
\newcommand{\dR}{{\mathrm{dR}}}             

\label{arrows}

\newcommand \inj {\hookrightarrow }
\newcommand \into {\hookrightarrow }
\renewcommand \to {\rightarrow}

\renewcommand{\projlim}{\varprojlim}
\renewcommand{\injlim}{\varinjlim}

\label{•••••••••4}
 
 \label{matrix, groups}

\def\det{\mathrm{det}}

\label{module, homological alg}
 
\newcommand{\rep}{{\mathrm{Rep}}}

\def\cont{\mathrm{cont}}
\def\inf{{\mathrm{inf}}}

\newcommand{\gal}{{\mathrm{Gal}}}

\def\an{\mathrm{an}}

\newcommand{\Lie}{\mathrm{Lie}}

\label{!•••••••••}
\label{ALL: p-adic Hodge}

\label{LAV}

\newcommand{\pa}{{\mathrm{pa}}}
\newcommand{\la}{{\mathrm{la}}}

\newcommand{\dan}{\text{$\mbox{-}\mathrm{an}$}}

\newcommand{\dla}{\text{$\mbox{-}\mathrm{la}$}}

\newcommand{\dpa}{\text{$\mbox{-}\mathrm{pa}$}}
\renewcommand{\log}{\mathrm{log}}

\newcommand{\nablagamma}{{\nabla_\gamma}}
\newcommand{\nablatau}{{\nabla_\tau}}
\newcommand{\nnabla}{{N_\nabla}}
\newcommand{\hatgla}{{\hat{G}\dla}}

\newcommand{\hatg}{{\hat{G}}}

\newcommand{\Kpinfty}{{K_{p^\infty}}}
\newcommand{\kpinfty}{{K_{p^\infty}}}

\newcommand{\Kinfty}{{K_{\infty}}}
\newcommand{\kinfty}{{K_{\infty}}}
\newcommand{\hatkinfty}{{\widehat{K_{\infty}}}}

\label{LAV: Lie gp Lie alg}

\newcommand{\gammak}{{\Gamma_K}}

\newcommand{\gk}{{G_K}}
\newcommand{\gkinfty}{{G_{K_\infty}}}
\newcommand{\gkpinfty}{{G_{K_{p^\infty}}}}

\label{period ring: Fontaine}

\newcommand\dr{{\mathrm{dR}}}

\newcommand\rig{{\mathrm{rig}}}

\newcommand{\ainf}{{\mathbf{A}_{\mathrm{inf}}}}

\newcommand{\bcrisplus}{{\mathbf{B}^+_{\mathrm{cris}}}}

\newcommand{\bdrplus}{{\mathbf{B}^+_{\mathrm{dR}}}}
\newcommand{\bdr}{{\mathbf{B}_{\mathrm{dR}}}}

\label{period ring: Gao-Min-Wang}

\newcommand{\be}{{\mathbf{B}_e}}

\newcommand{\smat}[1]{\left( \begin{smallmatrix} #1 \end{smallmatrix} \right)}

\label{period ring: bb font}

\newcommand{\bbdrplus}{{\mathbb{B}_{\mathrm{dR}}^{+}}}

\newcommand*{\wt}[1]{\widetilde{#1}}

\label{period ring: AB ring}
 \newcommand{\A}{ {\mathbf{A}}   }

\newcommand{\wta}{   {\widetilde{{\mathbf{A}}}}  }

\newcommand{\wtb}{   {\widetilde{{\mathbf{B}}}}  }
\newcommand{\wtbrig}{   {\widetilde{{\mathbf{B}}}_\rig^\dagger}  }

\newcommand{\wtbrigL}{   {\widetilde{{\mathbf{B}}}_{\rig, L}^\dagger}  }

\newcommand{\wtB}{   {\widetilde{{\mathbf{B}}}}  }

\newcommand{\wte}{   {\widetilde{{\mathbf{E}}}}  }
\newcommand{\wtm}{{\wt{M}}}

\label{period ring: integral Hodge}

\newcommand{\ocpflat}{{\mathcal{O}_C^\flat}}

\label{peirod module: D}

\label{Cats: Rep and Higgs}

\label{••••••••1}
\label{ALL: site, sheaves}

\newcommand{\rg}{\mathrm{R}\Gamma}

\label{site: etale proetale}

\label{font: roman rm}

\label{font: mathbb}

\newcommand{\bB}{{\mathbb B}}

\newcommand{\bM}{{\mathbb M}}

\newcommand{\bQ}{{\mathbb Q}}

\newcommand{\bZ}{{\mathbb Z}}

\newcommand{\bbb}{{\mathbb{B}}}

  \newcommand{\bbd}{{\mathbb{D}}}
   \newcommand{\bbD}{{\mathbb{D}}}

 \newcommand{\Q}{{\mathbb{Q}}}

  \newcommand{\z}{{\mathbb{Z}}}
 
  \newcommand{\bbz}{{\mathbb{Z}}}

\newcommand{\zp}{{\mathbb{Z}_p}}
\newcommand{\qp}{{\mathbb{Q}_p}}
\newcommand{\fp}{{\mathbb{F}_p}}

\def\upi{\underline \pi}

\label{font: mathcal}

\newcommand{\cale}{{\mathcal{E}}}

\label{font: mathfrak}

\newcommand{\fkg}{{\mathfrak{g}}}

\newcommand{\fkt}{{\mathfrak{t}}}

 \label{font: hat}

\label{font: bar}

\newcommand{\barK}{{\overline{K}}}

\label{font: bold}

 \label{font: mathbf}
 \newcommand{\bfa}{\mathbf{A}}
\newcommand{\bfb}{\mathbf{B}}
\newcommand{\bfD}{ {\mathbf{D}}}

\newcommand{\bfd}{ {\mathbf{D}}}

\newcommand{\wtd}{\wt{\bfd}}

\newcommand{\bfB}{\mathbf{B}}

   \renewcommand{\phi}{{\varphi}}  
\renewcommand{\mod}{\mathrm{Mod}}
\newcommand{\brigkinfty}{{\mathbf{B}_{\rig, \kinfty}^\dagger}}

\newcommand{\hatgpa}{{\hat{G}\dpa}}
 
\newcommand{\wtD}{\wt{\mathbf{D}}}
\newcommand{\phitau}{{\varphi, \tau}}
 
\newcommand{\phigamma}{{\varphi, \gammak}}

\author[]{Hui Gao}   \address{Department of Mathematics and Shenzhen International Center for Mathematics, Southern University of Science and Technology, Shenzhen 518055, China}   \email{gaoh@sustech.edu.cn}

\author[]{Luming Zhao}   \address{School of Mathematical Sciences, Peking University, Beijing 100084, China}   \email{lumingzhao@math.pku.edu.cn}

\begin{document}
\title[]{Cohomology of $(\varphi, \tau)$-modules} 
  
\begin{abstract}   
\normalsize{ 
We construct   cohomology theories for $(\varphi, \tau)$-modules, and study their relation with cohomology of  $(\varphi, \Gamma)$-modules, as well  as Galois cohomology. The method is axiomatic, and can treat the \'etale case, the overconvergent case, and the rigid-overconvergent case simultaneously. We use  recent advances in  locally analytic cohomology as a key ingredient.
}
\end{abstract}

\date{\today}
\maketitle
\setcounter{tocdepth}{1}
\tableofcontents

\section{Introduction}
 \subsection{Overview}
Let $K$ be a mixed characteristic complete discrete valuation field with perfect residue field of characteristic $p$, and let $\gk=\gal(\barK/K)$ be the Galois group.
In \cite{Fon90}, Fontaine introduces the $(\phi, \Gamma)$-modules to classify $p$-adic representations of $\gk$. These modules are defined over a    concrete ``one-variable" ring, and are further equipped with a Frobenius operator $\phi$ and a $\gammak=\gal(\kpinfty/K)$-action (here $\kpinfty$ is the cyclotomic extension adjoining all $p$-power roots of unity). Thus, the data of a Galois representation can now be described using just \emph{two} matrices: one for $\phi$ and one for $\gamma$ (a topological generator of $\gammak$; for simplicity, the readers can assume $p>2$ in the introduction). This very concrete description makes it possible to carry out explicit  computations about Galois representations.

Let $V$ be a $p$-adic Galois representation of $\gk$, and let $\bM$  be the associated  $(\phi, \Gamma)$-module. In \cite{Her98, Her01}, Herr defines a three term complex
\begin{equation} \label{eqcohophigammaintro}
 C_{\phigamma}(\bM):=\quad   [
 \bM\xrightarrow{\phi-1, \gamma-1} \bM \oplus \bM \xrightarrow{1-\gamma, \phi-1} \bM].
\end{equation} 
(Here the second arrow sends $(a, b)$ to $(1-\gamma)(a)+(\varphi-1)(b)$).
Herr shows that this complex is quasi-isomorphic to $\rg(\gk, V)$; using this,  Herr  gives an explicit reproof  of Tate's theorems on cohomology of Galois representations (when $K/\qp$ is a finite extension).  
 In Herr's work, he only considers  \'etale  $(\phi, \Gamma)$-modules. In  \cite{CC98},  (rigid)-overconvergent $(\phi, \Gamma)$-modules are constructed by Cherbonnier-Colmez;   exactly the same formula  as \eqref{eqcohophigammaintro} gives rise to Herr complex for these variants of $(\phi, \Gamma)$-modules. R.~Liu \cite{LiuIMRN08} shows that these overconvergent complexes still compute Galois cohomology  (when $K/\qp$ is a finite extension); in addition, Liu also studies cohomology of  rigid-overconvergent $(\phi, \Gamma)$-modules that are not necessarily associated with Galois representations.
 These results are important tools in many aspects of  $p$-adic Hodge theory.

In this paper, we study similar cohomological questions for the $(\phi, \tau)$-modules. The $(\phi, \tau)$-modules are first introduced by Caruso \cite{Car13} (building on earlier work of Breuil \cite{Bre99b}), and recently have found many applications in \emph{integral} $p$-adic Hodge theory. Indeed, the Breuil--Kisin modules, first considered by Breuil \cite{Bre99b} then constructed in full generality by Kisin \cite{Kis06}, can be regarded as a special case of  $(\phi, \tau)$-modules. The theory of  Breuil--Kisin modules is a \emph{fundamental tool} in every aspect of integral $p$-adic Hodge theory: they are used to construct (various) moduli of Galois representations \cite{Kis08, EG23}, and inspire the  construction of several  unifying  cohomology theories \cite{BMS1, BMS2, BS22}.
Recently, overconvergent $(\phi, \tau)$-modules are constructed  in \cite{GL20, GP21}  by the first named author jointly with T.~Liu and Poyeton respectively, in analogy with the overconvergent $(\phi, \Gamma)$-modules constructed by Cherbonnier--Colmez. These overconvergent $(\phi, \tau)$-modules have found several applications, e.g., in the study of Breuil--Kisin $\gk$-modules \cite{Gao23}, and most recently in Hodge--Tate and $\bbdrplus$-prismatic crystals \cite{GMWHT, GMWdR}. 
As there will only be more applications of the $(\phi, \tau)$-modules,  it has thus become urgent to study other \emph{foundational} properties of these modules. This paper grows out with these intended applications in mind, and fills some blanks in the literature.
In addition,  our  approach even leads to some new results for the cohomology of $(\phi, \Gamma)$-modules (which has been such a widely-known theory for circa 20 years); for the curious readers, cf. Rem. \ref{rem: introend}(\ref{itemremintro2}) and the historical and comparative Remarks \ref{remhistphigamma} and \ref{remfinalourmethodphigamma}.

Before we discuss   cohomology theories, we already need to point out that  there are several different \emph{variants} of $(\phi, \tau)$-modules (similar to the $(\phi, \Gamma)$-modules). We dub the three main  genres  with the adjectives: \'etale, overconvergent and rigid-overconvergent; these are defined over the Laurent ring, the overconvergent Laurent ring, and the Robba ring respectively.
In this introduction, to avoid introducing too many notations, we shall focus on the \emph{rigid-overconvergent $(\phi, \tau)$-modules} (also called the \emph{$(\phi, \tau)$-modules over the Robba ring}). These are indeed the most useful variants in all the recent applications mentioned in the last paragraph, and have the richest cohomology theory.

 To guide the readers (who might already have some experience with cohomology of $(\phi, \Gamma)$-modules), we summarize our main discussions in this introduction.
 In \S \ref{subsec: intro phi tau mod}, we quickly recall $(\phi, \tau)$-modules.
 In \S \ref{subsec coho 1} (for any $p$), we produce a cohomology theory using $\phi$-operator and group actions; this uses the least machinery and  is not too difficult.
 In \S \ref{subsec coho 2}, assuming $p>2$, we can refine results in \S \ref{subsec coho 1} by producing a 3-term $(\phi, \tau)$-complex  resembling Herr's complex \eqref{eqcohophigammaintro}: however the terms are much more \emph{complicated} in this case, making the $(\phi, \tau)$-complex very unsatisfying. 
 Finally in \S \ref{subsec coho 3}, (for any $p$ again), we produce a 3-term complex using $\phi$-operator and \emph{Lie algebra} actions. This is the most satisfying complex, and makes use of recent advances in \emph{locally analytic cohomology};  we do warn that its cohomologies are ``bigger" than desired but nonetheless can be ``pruned".

\subsection{Category of  $(\phi, \tau)$-modules}\label{subsec: intro phi tau mod}
In this subsection, we quickly recall the category of  $(\phi, \tau)$-modules.

\begin{notation}[Fields and groups] \label{notafields} 
 Let  $p$ be a prime. Let $k$ be a perfect  field of characteristic $p$, let $W(k)$ be the ring of Witt vectors, and let $K_0 :=W(k)[1/p]$.
Let $K$ be a totally ramified finite extension of $K_0$, let $\mathcal O_K$ be the ring of integers, and let $e := [K: K_0]$.
Fix an algebraic closure $\overline {K}$ of $K$ and set $G_K:=\Gal(\overline{K}/K)$.
 Let $\mu_1 \in \barK$ be a primitive $p$-root of unity, and inductively, for each $n \geq 2$, choose $\mu_n \in \barK$ a $p$-th root of $\mu_{n-1}$. Fix $\pi \in   K$  a   uniformizer. Fix a sequence of elements $\pi _n \in \overline K$ inductively such that $\pi_0 = \pi$ and $(\pi_{n+1})^p = \pi_n$.
  Define the fields
$$K_{\infty}   = \cup _{n = 1} ^{\infty} K(\pi_n), \quad K_{p^\infty}=  \cup _{n=1}^\infty
K(\mu_{n}), \quad L =  \cup_{n = 1} ^{\infty} K(\pi_n, \mu_n).$$
Let $$G_{\kinfty}:= \gal (\overline K / K_{\infty}), \quad G_{\kpinfty}:= \gal (\overline K / K_{p^\infty}), \quad G_L: =\gal(\overline K/L).$$
Further define $\Gamma_K, \hat{G}$ as in the following diagram, where we let $\tau$ be a topological generator of $\gal(L/\kpinfty) \simeq \zp$, cf. Notation \ref{nota hatG} for more details.
\[
\begin{tikzcd}
                                       & L                                                                                             &                             \\
\kpinfty \arrow[ru, "\langle\tau\rangle", no head] &                                                                                               & \kinfty \arrow[lu, no head] \\
                                       & K \arrow[lu, "\Gamma_K", no head] \arrow[ru, no head] \arrow[uu, "\hat{G}"', no head, dashed] &
\end{tikzcd}
\]
Let $E(u)=\mathrm{Irr}(\pi, K_0) \in W(k)[u]$ be the Eisenstein polynomial for $\pi$.
\end{notation}

\begin{construction}
We now briefly recall the three variants of $(\varphi, \tau)$-modules. All rings here and in the introduction are reviewed in details in \S \ref{sec: rings and gps}.
\begin{enumerate} 
\item Recall that the $\varphi$-action  of an \emph{\'etale} $(\varphi, \tau)$-module is defined over the fraction field of $\bfa_\kinfty=(W(k)[[u]][\frac{1}{u}])^{\wedge_p}$ where $\wedge_p$ denotes the $p$-adic completion; more  concretely, it is the field
\begin{equation*}\label{eqbkinf}
\mathbf{B}_{\Kinfty}   :=\{ \sum_{i=-\infty}^{+\infty} a_i u^i : a_i \in K_0, \lim_{i \to -\infty}v_p(a_i) =+\infty, \text{ and } \inf_{i \in \mathbb{Z}}v_p(a_i) >-\infty \}.
\end{equation*}
Here $v_p$ is the valuation on $\barK$ such that $v_p(p)=1$. 
 The $\tau$-action   is defined over a bigger filed ``$\wt{\mathbf{B}}_L$" which we do not recall here, see \S \ref{sec: rings and gps}.
 Indeed, roughly speaking, a  $(\varphi, \tau)$-module is a finite free $\mathbf{B}_{\Kinfty}$-vector space $\bfD_\kinfty$  equipped with  certain commuting maps $\varphi: \bfD_\kinfty \to \bfD_\kinfty$ and $\tau: \wt{\mathbf{B}}_L\otimes_{\mathbf{B}_{\Kinfty} } \bfD_\kinfty \to \wt{\mathbf{B}}_L\otimes_{\mathbf{B}_{\Kinfty} } \bfD_\kinfty$.
By \cite[Thm. 1]{Car13}, the \'etale  $(\varphi, \tau)$-modules classify all Galois representations of $G_K$.

\item By \cite{GL20, GP21}, \'etale $(\varphi, \tau)$-modules  are \emph{overconvergent}. That is, (roughly speaking), the $\varphi$-action can be defined over the sub-field:
\begin{equation*}\label{eqbkinfd}
\mathbf{B}_{K_\infty}^\dagger: =\{ \sum_{i=-\infty}^{+\infty} a_i u^i \in \mathbf{B}_{K_\infty}, \lim_{i \to -\infty}(v_p(a_i) +i\alpha)= +\infty \text{ for some } \alpha >0  \};
\end{equation*}
 also, the $\tau$-action can  be defined over some sub-field $\wtB_L^\dagger \subset \wtB_L$.
 
 \item Let us introduce the following Robba ring (which contains $\mathbf{B}_{K_\infty}^\dagger$),
\begin{equation*}\label{eqbrig}
\begin{split}
\mathbf{B}_{\rig, K_\infty}^\dagger: =\{f(u)= \sum_{i=-\infty}^{+\infty} a_i u^i, a_i \in K_0,   f(u) \text{ converges } \\
 \text{ for all } u \in \overline{K} \text{ with } 0<v_p(u)<\rho(f) \text{ for some } \rho(f)>0\}.
\end{split}
\end{equation*}
Similar to \emph{rigid-overconvergent} $(\phi, \Gamma)$-modules (i.e., $(\phi, \Gamma)$-modules over the Robba ring)  which are not necessarily \'etale, one can also define \emph{rigid-overconvergent} $(\phi, \tau)$-modules, now with $\tau$-action defined over another ring $\wtb_{\rig, L}^\dagger$.  
\end{enumerate}
\end{construction}

\begin{theorem}[cf. \S \ref{sec:equiv cats}] (Let $p$ be any prime).
There is an equivalence between the category of  $(\phi, \Gamma)$-modules over the Robba ring and the category of $(\phi, \tau)$-modules over the Robba ring
    \[ \mod_{\varphi, \gammak}(\bB_{\rig, \kpinfty}^\dagger) \simeq \mod_{\varphi, \tau}(\bfB_{\rig, \kinfty}^\dagger,  \wtb_{\rig, L}^\dagger) \]
\end{theorem}
Here we use the blackboard font to denote rings and modules in $(\phi, \Gamma)$-module theory; the algebraic structure of $\bB_{\rig, \kpinfty}^\dagger$ is quite similar to $\bfB_{\rig, \kinfty}^\dagger$, i.e., is a Robba ring with one variable; but their Frobenius and Galois structures are completely different.
To discuss cohomology theory of these modules, we first fix some notations.

\begin{notation}
For an object
\[ (\bfD_{\rig, \kinfty}^\dagger, \wtD_{\rig, L}^\dagger) \in  \mod_{\varphi, \tau}(\bfB_{\rig, \kinfty}^\dagger,  \wtb_{\rig, L}^\dagger ). \]
Denote its corresponding object in $\mod_{\varphi, \gammak}(\bB_{\rig, \kpinfty}^\dagger)$ by $\bbD_{\rig, \kpinfty}^\dagger$. 
These notations are meant to be suggestive; e.g., for $X\in \{\kinfty, L, \kpinfty\}$,   $D_{\ast,X}^\ast$ is \emph{fixed} under $G_X$-action. In addition, we have (base change) isomorphisms
\[ \bfD_{\rig, \kinfty}^\dagger \otimes_\brigkinfty \wtbrigL \simeq \wtD_{\rig, L}^\dagger \simeq \bbD_{\rig, \kpinfty}^\dagger \otimes_{\bB_{\rig, \kpinfty}^\dagger} \wtbrigL \]
\end{notation} 

\begin{remark}
There will be many ``$D$"-modules in the main text, cf. diagrams in Notations \ref{notaetalmod} and \ref{notarigmod}. Thus in the introduction, we  have chosen to use the full (although complicated) notations to avoid further confusions. (For example, the simpler notation $\bfD_\kinfty$ will mean the \'etale $\phi$-module).
\end{remark}

\subsection{Cohomology I: $\phi+$group} \label{subsec coho 1}

A useful realization is that the Herr complex \eqref{eqcohophigammaintro} can be interpreted using (continuous, as always in this paper) group cohomology. Indeed, the Herr complex $C_{\phi, \gammak}(\bbD_{\rig, \kpinfty}^\dagger)$ can be interpreted as 
\[ \rg(\Gamma_K, \bbD_{\rig, \kpinfty}^\dagger)^{\phi=1} \]
where  $\rg(\Gamma_K, \bbD_{\rig, \kpinfty}^\dagger) =[\bbD_{\rig, \kpinfty}^\dagger\xrightarrow{\gamma-1} \bbD_{\rig, \kpinfty}^\dagger]$ (say, when $p>2$) is the (continuous) group cohomology, and $\phi=1$ denotes the  homotopy fiber of $\phi-1$.
In more concrete terms, we are simply saying that   $C_{\phi, \gammak}(\bbD_{\rig, \kpinfty}^\dagger)$ is the totalization of the double complex
\[
\begin{tikzcd}
{\bbD_{\rig, \kpinfty}^\dagger} \arrow[d, "\gamma-1"] \arrow[r, "\varphi-1"] & {\bbD_{\rig, \kpinfty}^\dagger} \arrow[d, "\gamma-1"] \\
{\bbD_{\rig, \kpinfty}^\dagger} \arrow[r, "\varphi-1"]                       & {\bbD_{\rig, \kpinfty}^\dagger}                      
\end{tikzcd}
\]
We would like to imitate this in the $(\phi, \tau)$-module case; however, a naive idea to replace ``$\gamma-1$" by ``$\tau-1$" would fail, because $\tau$-action on 
$\bfD_{\rig, \kinfty}^\dagger$ is not stable. Nonetheless, we have a $\hatg$-action on $\wtD_{\rig, L}^\dagger$; thus we can form the complex
\[ C_{\phi, \hatg}(\wtD_{\rig, L}^\dagger):= \rg(\hatg, \wtD_{\rig, L}^\dagger)^{\phi=1} \]
Let us explain this complex in more concrete terms (when $p>2$). 
The $p$-adic Lie group $\hatg$ is topologically generated by $\gamma$ and $\tau$; using a Lazard--Serre resolution, cf. Lem. \ref{lemgroupiwasawa},    $\rg(\hatg, \wtD_{\rig, L}^\dagger)$ can be explicitly expressed by a three term complex
 \begin{equation} \label{intro3termhatg}
C_{\gamma, \tau}(\wtD_{\rig, L}^\dagger):=\quad [\wtD_{\rig, L}^\dagger \xrightarrow{\gamma-1, \tau-1} \wtD_{\rig, L}^\dagger\oplus \wtD_{\rig, L}^\dagger \xrightarrow{\tau^{\chi(\gamma)}-1, \delta-\gamma}\wtD_{\rig, L}^\dagger ]  
\end{equation}
Here, $\chi$ is cyclotomic character, and $\delta = \frac{\tau^{\chi(\gamma)} - 1}{\tau - 1}$  is an element in the Iwasawa algebra of $\hatg$.
However, one needs to add in the $\phi=1$ fiber, and thus it would produce a \emph{four} term complex $C_{\phi, \gamma, \tau}(\wtD_{\rig, L}^\dagger)$.
It turns out  it is a correct complex.

\begin{theorem}[cf. Thm. \ref{thmcohononetale}] \label{thmintrogpcoho}
 (Let $p$ be any prime). 
There is a natural quasi-isomorphism
\[C_{\phi, \hatg}(\wtD_{\rig, L}^\dagger) \simeq C_\phigamma (\bbD_{\rig, \kpinfty}^\dagger) \]
If these modules are associated to $V\in \rep_\gk(\qp)$, then the complexes are further quasi-isomorphic to $\rg(\gk, V)$.
\end{theorem}

\begin{remark} We give some historical remarks; cf. \S \ref{subsec: final historical remarks} for more.
\begin{enumerate}
\item  In \cite{TR11}, Tavares Ribeiro defines a similar $C_{\phi, \hatg}$-complex for \'etale $(\phi, \tau)$-modules, and shows that it computes Galois cohomology. His proof is ``direct" (using devissage  and $\delta$-functors), and thus involves many complicated computations (which quite resembles Herr's original approach \cite{Her98}).

\item In contrast, our proof of Thm. \ref{thmintrogpcoho} is \emph{axiomatic}; it involves certain standard cohomology vanishing computations that work in the \'etale case as well, and thus recovers Tavares Ribeiro's results in a very conceptual way. 

 \item The comparison  
 \[  C_\phigamma (\bbD_{\rig, \kpinfty}^\dagger) \simeq \rg(\gk, V)\]
 was formerly known by \cite[Prop. 2.7]{LiuIMRN08}, but only when $K/\qp$ is a finite extension; Liu's method uses $\psi$-operator  and  does \emph{not} work when $K/\qp$ is an infinite extension. In contrast, our   conceptual proof (which  uses   other complexes as well) works for any $K$, and does not use $\psi$-operator.
 \end{enumerate}
\end{remark}

\begin{rem}\label{remunsatisfy}
    We find the  4-term complex $C_{\phi, \gamma, \tau}(\wtD_{\rig, L}^\dagger)$   unsatisfying in (at least) two ways:
\begin{enumerate}
    \item Firstly, \emph{a priori}, it is not  obvious to see  its cohomology is   concentrated in degree $[0, 2]$;
    \item  Secondly and   more importantly, the module $\bfD_{\rig, \kinfty}^\dagger$ itself does not appear in this complex!
\end{enumerate}
\end{rem}

\subsection{Cohomology II: $\phi+\tau$ } \label{subsec coho 2}

We first try to ``resolve"  Item (1) in Rem. \ref{remunsatisfy}.

\begin{theorem}[cf. Thm. \ref{thmcohononetale}] \label{thmintro3term}
Suppose $p>2$. 
Let $\wtD_{\rig, \kinfty}^\dagger=(\wtD_{\rig, L}^\dagger)^{\gal(L/\kinfty)}$.
Define a 3-term complex
\begin{equation}\label{eqintro3term}
    C_{\phi, \tau}(\wtD_{\rig, \kinfty}^\dagger, \wtD_{\rig, L}^\dagger):=\quad  [\wtD_{\rig, \kinfty}^\dagger \xrightarrow{\phi-1, \tau-1} \wtD_{\rig, \kinfty}^\dagger \oplus (\wtD_{\rig, L}^\dagger)^{\delta-\gamma=0} \xrightarrow{1-\tau, \phi-1} (\wtD_{\rig, L}^\dagger)^{\delta-\gamma=0}] 
\end{equation}
Then
\[C_{\phi, \tau}(\wtD_{\rig, \kinfty}^\dagger, \wtD_{\rig, L}^\dagger) \simeq C_{\phi, \gamma, \tau}(\wtD_{\rig, L}^\dagger) \simeq C_{\phi, \gamma}(\bbD_{\rig, \kpinfty}^\dagger)\]
 \end{theorem}
 \begin{proof}[Sketch of main ideas.]  
 Note that the complex $C_{\gamma, \tau}(\wtD_{\rig, L}^\dagger)$ \eqref{intro3termhatg} is the totalization of the double complex
\[ \begin{tikzcd}
\wtD_{\rig, L}^\dagger  \arrow[r, "\tau-1"] \arrow[d, "\gamma-1"'] & \wtD_{\rig, L}^\dagger  \arrow[d, "\delta-\gamma"] \\
\wtD_{\rig, L}^\dagger  \arrow[r, "\tau^{\chi(\gamma)}-1"]         & \wtD_{\rig, L}^\dagger                            
\end{tikzcd} \]
We can prove that both vertical arrows $\gamma-1$ and $\delta-\gamma$ are \emph{surjective}. Thus $C_{\gamma, \tau}(\wtD_{\rig, L}^\dagger)$ is quasi-isomorphic to the ``vertical kernel" complex, which is precisely
\[   [ \wtD_{\rig, \kinfty}^\dagger   \xrightarrow{\tau-1} (\wtD_{\rig, L}^\dagger)^{\delta-\gamma=0}] \]
One can conclude by incorporating  $\phi$-cohomology.
 \end{proof}

\begin{rem} \label{rem: intro p equal 2}
   We point our a thorny issue when $p=2$. In this case, the groups $\gal(L/\kinfty)$ and $\gal(\kpinfty/K)$ are both open subgroups of $\mathbb{Z}_2^\times$ and hence are not necessarily pro-cyclic. In theory of $(\phi, \Gamma)$-modules, one can first take a finite (and Galois) extension $K'/K$ inside $\kpinfty$ such that $\kpinfty/K'$ is pro-cyclic, then study $(\phi, \Gamma)$-modules over $K'$, and finally descend to $K$-level. This strategy unfortunately breaks for $(\phi, \tau)$-modules. Indeed, the uniformizer of $K'$ is different from $K$, and hence its Kummer tower extension is \emph{completely} different from that of $K$. Indeed, even though it might still be possible to construct certain \emph{explicit} Lazard--Serre resolution as in Lem. \ref{lemgroupiwasawa} when $p=2$, the complex is likely to be very complicated and unsuitable to construct simple $(\phi,\tau)$-complexes as in Thm. \ref{thmintro3term}.
 \end{rem}

Even though we \emph{fabricated} a 3-term complex in Thm. \ref{thmintro3term}, the module  $\bfD_{\rig, \kinfty}^\dagger$ remains at large.
With this in mind, it is tempting to construct $C_{\phi, \tau}(\bfD_{\rig, \kinfty}^\dagger, \wtD_{\rig, L}^\dagger) $ simply by replacing $\wtD_{\rig, \kinfty}^\dagger$ in Eqn. \eqref{eqintro3term} by $\bfD_{\rig, \kinfty}^\dagger$. 
Item (1) in the next Thm. \ref{thm3complexwrong}  tells us that it is indeed not a whimsical thought; Item (2) then informs us that unfortunately (and somewhat surprisingly), this (natural) idea does \emph{not} work.

\begin{thm} \label{thm3complexwrong}
Suppose $p>2$.    Let $V \in \rep_\gk(\qp)$, and consider its associated modules. (cf.  Notations \ref{notaetalmod} and \ref{notarigmod} for the many different ``$D$"-modules here; alternatively, the readers could ignore all others and  focus on $ C_{\phi, \tau}(\bfD_{\rig, \kinfty}^\dagger, \wtD_{\rig, L}^\dagger) $ in  Item (2) here.)
\begin{enumerate}
    \item \emph{(cf. Thm. \ref{thmetalcompa}, Thm. \ref{thmcohononetale})}. We have quasi-isomorphisms
\[ C_{\varphi, \tau}(\bfD_{\kinfty}, \wtd_L) \simeq C_{\varphi, \tau}(\wtD_{\kinfty}, \wtd_L) \simeq C_\phitau(\wt\bfD_\kinfty^\dagger,\wtd_L^\dagger) \simeq C_{\phi, \tau}(\wtD_{\rig, \kinfty}^\dagger, \wtD_{\rig, L}^\dagger) \simeq \rg(\gk, V) \]
in particular: when considering cohomology of ``\'etale" $(\varphi, \tau)$-modules (as in the first two complexes), it does not hurt to remove tilde (on the $\varphi$-modules).

\item \emph{(cf. Prop. \ref{propwrongphitau}).} We have  a quasi-isomorphism 
\begin{equation}\label{eqwrongintrocomp1}
   C_{\phi, \tau}(\bfD_{\kinfty}^\dagger, \wtD_{L}^\dagger)  \simeq C_{\phi, \tau}(\bfD_{\rig, \kinfty}^\dagger, \wtD_{\rig, L}^\dagger);
\end{equation}
but they are (in general) \emph{not} quasi-isomorphic to $C_{\phi, \tau}(\wtD_{\rig, \kinfty}^\dagger, \wtD_{\rig, L}^\dagger)$ (equivalently, to $\rg(\gk, V)$): this is already so when $V=\qp$ is the trivial representation!  Thus, both   complexes in Eqn. \eqref{eqwrongintrocomp1} are \emph{wrong} ones.
\end{enumerate}
\end{thm}
\begin{proof}[Sketch of main ideas.]
    First we point out the complex $C_{\varphi, \tau}(\bfD_{\kinfty}, \wtd_L)$ in Item (1) is already  studied by the second named author in \cite{Zhao25} (using devissage and $\delta$-functors); we shall use a more conceptual strategy  which works for other complexes as well. For Item (2), to show $C_{\phi, \tau}(\bfD_{\rig, \kinfty}^\dagger, \wtD_{\rig, L}^\dagger)$ is \emph{not} quasi-isomorphic to $C_{\phi, \tau}(\wtD_{\rig, \kinfty}^\dagger, \wtD_{\rig, L}^\dagger)$; it is equivalent to show that the $\varphi$-cohomology
    \[ [\bfD_{\rig, \kinfty}^\dagger \xrightarrow{\varphi-1} \bfD_{\rig, \kinfty}^\dagger] \]
    is \emph{not} quasi-isomorphic to
     \[ [\wtD_{\rig, \kinfty}^\dagger \xrightarrow{\varphi-1} \wtD_{\rig, \kinfty}^\dagger]; \]
     this can be achieved by making use of \emph{other} $\varphi$-complexes (as well as some ``$\psi$-complexes").
\end{proof}

 \subsection{Cohomology III: $\phi+$Lie algebra} \label{subsec coho 3}
 So far, $\bfD_{\rig, \kinfty}^\dagger$ is still fugitive from any (correct) cohomology theory; in addition, the appearance of $(\wtD_{\rig, L}^\dagger)^{\delta-\gamma=0}$ in \eqref{eqintro3term} (similarly in other complexes in Thm. \ref{thm3complexwrong}(1)) is very unsatisfying because it is only a \emph{very} implicit   $\qp$-vector  space, and \emph{not} a module over $\bfb_{\rig, \kinfty}^\dagger$ or any other interesting rings.
 Note that it appears precisely because $\tau$-action on $\bfD_{\rig, \kinfty}^\dagger$ is \emph{not stable}.
 
 In a previous work \cite{Gao23} by the first named author, we construct a \emph{differential operator}, which is ``essentially" the Lie algebra operator associated to $\tau$, that is stable on  $\bfD_{\rig, \kinfty}^\dagger$!
Indeed, let
 $\nabla_\tau:= (\log \tau^{p^n})/{p^n}$ for $n\gg 0$ be the Lie-algebra operator with respect to the $\tau$-action, and define
 $$N_\nabla:=\frac{1}{p\mathfrak{t}}\cdot \nabla_\tau$$
 where $\mathfrak{t}$ is a certain  ``normalizing"  element (cf. Def. \ref{defnfkt}). (Note that there might be   some modifications  in certain cases when $p=2$). Then
\[N_\nabla(\bfD_{\rig, \kinfty}^\dagger) \subset  \bfD_{\rig, \kinfty}^\dagger.\]  
 This ``infinitesimal $\tau$-operator" brings the hope to construct a 3-term complex using $\bfD_{\rig, \kinfty}^\dagger$ \emph{only}. Define
\begin{equation}
C_{\phi, N_\nabla}(\bfD_{\rig, \kinfty}^\dagger):=   [ \bfD_{\rig, \kinfty}^\dagger \xrightarrow{\phi-1, N_\nabla} \bfD_{\rig, \kinfty}^\dagger\oplus \bfD_{\rig, \kinfty}^\dagger \xrightarrow{-N_\nabla, \frac{pE(u)}{E(0)}\phi-1} \bfD_{\rig, \kinfty}^\dagger ]
\end{equation} 
here the normalization by the invertible element $\frac{pE(u)}{E(0)}$ on the second arrow (cf. Notation \ref{notafields}) is needed to make the diagram a complex.
In fact, inspired by above, one can also form a similar complex for $(\phi, \Gamma)$-modules
\begin{equation}
 C_{\phi,  \nabla_\gamma}(\bbD_{\rig, \kpinfty}^\dagger):=   [ \bbD_{\rig, \kpinfty}^\dagger \xrightarrow{\phi-1,  \nabla_\gamma} \bbD_{\rig, \kpinfty}^\dagger\oplus \bbD_{\rig, \kpinfty}^\dagger \xrightarrow{-\nabla_\gamma, \phi-1} \bbD_{\rig, \kpinfty}^\dagger  ]
\end{equation} 
where $\nabla_\gamma$ is the Lie algebra operator for $\gammak$-action.
However, these two complexes are not  ``correct" complexes: even their $H^0$'s do not compute $H^0(\gk, V)$ (when they are associated to  $V \in \rep_\gk(\qp)$).
A key observation is that if we further take ``$\tau=1$ invariants" resp. ``$\gammak=1$ invariants" of these cohomology groups, then we again obtain the \emph{correct} cohomology.

\begin{theorem}[cf. Thm. \ref{thm: phi Lie algebra}] 
\label{thm:introliecoho}
 (Let $p$ be any prime). 
Let $H^i_{\phi, N_\nabla}$ resp. $H^i_{\phi, \nabla_\gamma}$  be the cohomology groups of $C_{\phi, N_\nabla}(\bfD_{\rig, \kinfty}^\dagger)$ resp. $C_{\phi,  \nabla_\gamma}(\bbD_{\rig, \kpinfty}^\dagger)$. One can consider their ``$\tau=1$ invariants" (in some sense) resp. $\gammak=1$ invariants, and we have 
\begin{equation}\label{eqintrocohoinv}
  (H^i_{\phi, N_\nabla}(\bfD_{\rig, \kinfty}^\dagger))^{\tau=1} \simeq  (H^i_{\phi, \nabla_\gamma}(\bbD_{\rig, \kpinfty}^\dagger))^{\gammak=1} \simeq H^i_{\phi, \gammak}(\bbD_{\rig, \kpinfty}^\dagger)  
\end{equation} 
Indeed, what we  actually prove are the following (stronger) results; for simplicity, assume all the modules are associated to a Galois representation $V\in \rep_\gk(\qp)$ (for general cases, one replaces $V$ below by a $B$-pair):
\begin{eqnarray}
     \label{eqintroliecoho}
H^i_{\phi, N_\nabla}  &  \simeq\injlim_n H^i(G_{K(\pi_n)}, V )  & =\bigcup_n H^i(G_{K(\pi_n)}, V )\\
  H^i_{\phi, \nabla_\gamma} &\simeq \injlim_n H^i(G_{K(\mu_n)}, V)  &=\bigcup_n H^i(G_{K(\mu_n)}, V)   
\end{eqnarray}
  \end{theorem}

 \begin{rem}
  Taking $\gammak=1$ invariants in \eqref{eqintrocohoinv}   is a legitimate and natural process. In contrast, the ``$\tau=1$ invariants" in \eqref{eqintrocohoinv}    is only meant to be   \emph{illustrative} here, and is indeed very \emph{artificial}; cf.  Rem. \ref{rem: artificial tau 1}.   We expect that the (stronger) comparisons in \eqref{eqintroliecoho} should be more useful.
 \end{rem}

\begin{remark} \label{rem: introend}
\begin{enumerate} 
\item When $i=0$ resp. $1$, Thm. \ref{thm:introliecoho} can be obtained relatively easily by hand, since these cohomology groups correspond to fixed points resp. extensions of modules.   However for $i=2$, there does not seem to be any naive method; indeed, our proof (which works uniformly for all $i$) uses (vanishing of) \emph{higher locally analytic vectors} recently studied by Porat \cite{Poratlav}, which in turn builds  on foundational works by Pan \cite{Pan22} and Rodrigues Jacinto--Rodr\'{\i}guez Camargo \cite{RJRC22}) etc.

\item  \label{itemremintro2}   
The ingredients of the complex $C_{\phi,  \nabla_\gamma}(\bbD_{\rig, \kpinfty}^\dagger)$ are long available since \cite{CC98, Ber02}; but as far as the authors are aware, the complex has not been studied in the literature; thus Thm. \ref{thm:introliecoho}   is  new even for this case.

\item Theorem \ref{thm:introliecoho} is strongly inspired by  \emph{Sen theory over the Kummer tower}, developed in \cite{GMWHT, GMWdR} by the first named author with Min and Wang; cf. \cite[\S 7]{GMWHT}. 
There, the   Lie algebra cohomology (without $\phi$) plays a crucial role in understanding prismatic cohomology of Hodge--Tate resp. $\bbdrplus$-crystals. 
It is thus natural to speculate that the complex $C_{\phi, N_\nabla}(\bfD_{\rig, \kinfty}^\dagger)$ (before or after taking   ``$\tau=1$" invariants of its cohomology) ---which is the \emph{only} useful complex capturing $\bfD_{\rig, \kinfty}^\dagger$ (and using it \emph{only}), and which works for any prime  $p$---to show up in other prismatic set-ups, and possibly to be a most useful complex in   future applications of $(\phi, \tau)$-modules.
\end{enumerate}
\end{remark}

\subsection{Structure of the paper} \label{subsec structure}
In \S \ref{sec: rings and gps}, we review notations on many rings (and pay special attention to \emph{integral} subrings). 
In \S \ref{sec:lav}, we first review the structure of some $p$-adic Lie groups; we then review  the notion of  \emph{higher} locally analytic vectors and their applications to locally analytic cohomologies. 
In \S \ref{sec:equiv cats}, we review equivalences of many module categories, particular those between $(\varphi, \Gamma)$-modules and  $(\varphi, \tau)$-modules.
We then carry out an axiomatic  study of cohomologies from \S \ref{sec: axiom gp coho} to \S\ref{sec: verify lie}.
In \S \ref{sec: axiom gp coho}, we axiomatically study group cohomology of $\hatg$; axioms  there are verified in \S \ref{sec: ts-1} using TS-1 descent.
In \S \ref{sec: axiom lie alg}, we axiomatically study Lie algebra cohomology of $\Lie \hatg$; axioms there are verified in \S \ref{sec: verify lie}.
In the final two sections \S \ref{sec: phi coho} and \S\ref{sec: coho phi tau}, we prove our main theorems on cohomology comparison. 
Indeed, we shall start with \S \ref{sec: phi coho} to deal with $\phi$-cohomology \emph{separately}; this makes many (although not all) cohomology comparisons in \S \ref{sec: coho phi tau} much more \emph{transparent}.

 \subsection*{Acknowledgement}
 We thank Olivier Brinon, Gal Porat and Juan Esteban Rodr\'{\i}guez Camargo for insightful discussions and correspondences. 
 H.G.~ thanks Yu Min and Yupeng Wang for useful collaborations in \cite{GMWHT, GMWdR}.
 L.Z.~ thanks Tong Liu and Ruochuan Liu for their generous hospitality during his postdoctoral stays at Purdue University and Peking University respectively. We thank the referees for carefully reading the paper and providing useful comments. 
   Hui Gao  is partially supported by the National Natural Science Foundation of China under agreements No. NSFC-12071201, NSFC-12471011.

\section{Notations: rings and integral subrings } \label{sec: rings and gps} 

 We first review the many period rings from $p$-adic Hodge theory; the rings in $(\varphi, \Gamma)$-module theory and in  $(\varphi, \tau)$-module theory have similar structures, and thus we introduce them in an \emph{axiomatic} fashion. We take special care about \emph{``integral overconvergent rings"}, cf. Rem \ref{rem int oc ring}.

 \subsection{Perfect rings} 
\begin{notation} \label{nota:ainf}
Let $C$ be the $p$-adic completion of $\overline{K}$, and let $\mathcal{O}_{C}$ be the ring of integers. Let $v_p$ be the valuation on $C$ such that $v_p(p)=1$.
Let \[\wte^+:=\ocpflat, \quad \wte :=C^\flat;\]
    $$\wt{\mathbf{A}}^+:= W(\wte^+), \quad  \wt{\mathbf{A}}:= W(\wte), \quad \wt{\mathbf{B}}^+:= \wt{\mathbf{A}}^+[1/p], \quad \wt{\mathbf{B}}:= \wt{\mathbf{A}}[1/p],$$
where $W(\cdot)$ means the ring of Witt vectors.
Using elements in Notation \ref{notafields}, define two elements in $\wte^+$ by $$\underline \varepsilon:= (\mu_{n}) _{n \geq 0}, \quad \upi:=\{\pi_n\}_{n \geq 0}$$
  let $[\underline \varepsilon], [\underline \pi ] \in \wt{\mathbf{A}}^+$ be the Teichm\"uller lifts.
  Also denote $\overline \pi =\underline{\varepsilon} -1 \in \wt{\mathbf E}^+$ (this is not $\underline \pi$), and let $[\overline \pi] \in \wt{\mathbf{A}}^+$ be its Teichm\"uller lift. 
\end{notation}

We introduce the following \emph{bounded overconvergent rings}. Here we use exactly the same notation as in \cite[\S 2]{GP21} (which we cite often), which in turn are modeled on those in \cite{Ber02} which in turn is slightly different from that in  \cite{CC98}.

\begin{notation}[Bounded overconvergent rings]
An element $x\in \wtb$ can be uniquely written as $x= \sum_{k \geq k_0} p^k[x_k]$. For each $r >0$, define a function
$$W^{[r,r]}(x) :=\inf_{k \geq k_0} \{k+\frac{p-1}{pr}\cdot v_{\wt{\mathbf E}}(x_k)\}.$$
Let $\wtb^{[r, +\infty]} \subset \wtb$ be the subset consisting of $x$ such that  $W^{[r, r]}(x) <+\infty$.  
Let 
$$\wta^{[r, +\infty]}:=\{ x\in \wta \cap \wtb^{[r, +\infty]} \mid W^{[r,r]}(x) \geq 0. \} $$
For $r \in \mathbb{Z}^{> 0}[1/p]$, one can show
$$\wta^{[r, +\infty]}  = \wt{\mathbf{A}}^+ [\frac{p}{[\overline \pi]^r}  ]^{\wedge_p}; \quad \wtb^{[r, +\infty]}=\wta^{[r, +\infty]}[1/p] $$
Further define
\[ \wtb^\dagger :=\cup_{r>0} \wtb^{[r, +\infty]} \]
\[  \wta^\dagger:=  \wtb^\dagger \cap \wta\]
\end{notation}

\begin{remark}[\emph{``Integral overconvergent rings"}] \label{rem int oc ring}

We discuss some very subtle (and confusing: at least to the authors) \emph
{integrality} issues. 
Note that we have
\[ \wta^{[r, +\infty]}[1/p] =(\wta^{[r, +\infty]}[{1}/{[\overline \pi]}])[1/p]=\wtb^{[r, +\infty]}  \]
However, we propose   that  $\wta^{[r, +\infty]}$   is \emph{not} the ``\emph{correct}" integral (bounded) overconvergent subring of $\wtb^{[r, +\infty]}$; rather, the ``correct" integral rings should be $
  \wta^{[r, +\infty]}[{1}/{[\overline \pi]}].$
Here are the reasons:
\begin{enumerate}
\item Firstly, we have the following relations:
\[ \wta^{[r, +\infty]} \subsetneqq \wta \cap  \wtb^{[r, +\infty]} =\wta^{[r, +\infty]}[{1}/{[\overline \pi]}] \]
\[ \cup_{r>0} \wta^{[r, +\infty]}  \subsetneqq  \wta^\dagger =\cup_{r>0}\wta^{[r, +\infty]}[{1}/{[\overline \pi]}] \]
Namely, the $\wta^{[r, +\infty]}$'s are not big enough to recover the entire integral (bounded) overconvergent ring $\wta^\dagger$.
 
\item \label{eq rem int oc ring} Secondly and more seriously, as we shall see in Prop \ref{prop: verify TS-1}, only the bigger ring $\wta^{[r, +\infty]}[{1}/{[\overline \pi]}]$ satisfies the TS-1 Axiom (due to \cite{Col08}). 

\item In ``retrospect", the possible ``confusion" here comes from the \emph{notation} $\wta^{[r, +\infty]}$, which \emph{looks like} the integral subring of $\wtb^{[r, +\infty]}$. 
This notation system is first introduced (as far as the authors are aware) in \cite[\S 2.1]{Ber02}, and is also used in the paper \cite{GP21}.
Note our $\wta^{[r,+\infty]}[{1}/{[\overline \pi]}]$ resp.~ $\wtb^{[r,+\infty]}$ is precisely the ring  $\wta^{(0, p/((p-1)r)]}$ resp.~ $\wtb^{(0, p/((p-1)r)]}$ in \cite[\S 5.2]{Col08}: for this reason, the notation of \cite{Col08} reflects better the integrality relations.
In this paper, we have continued to use the notation system of \cite{GP21}  as the paper will be cited often.

\end{enumerate} 
\end{remark}

We recall the following \emph{unbounded} overconvergent rings.

\begin{notation} \label{nota:wtai}
  For $I=[r,s]  \subset (0, +\infty)$ a   \emph{closed} interval, let
$W^{I}(x) := \inf_{\alpha \in I} \{W^{[\alpha, \alpha]}(x) \};$ 
$W^I$ defines a valuation on $\wtb^{[r, +\infty]}$. Let  $\wtb^{I}$ be the  completion of  $\wtb^{[r, +\infty]}$ with respect to $W^I$. 
Let $\wta^I \subset \wtb^I$ be the ring of integers with respect to $W^I$.
When $r \leq s$ are two elements in $ \mathbb{Z}^{> 0}[1/p]$, we have 
$$\wta^{[r,s]}  = \wt{\mathbf{A}}^+ [\frac{p}{[\overline \pi]^r}, \frac{[\overline \pi]^s}{p}  ]^{\wedge_p}; \quad \wtb^{[r, s]}=\wta^{[r, s]}[1/p]. $$ 
\end{notation}

\begin{construction}[Topologies]
Recall topologies on some rings (where $r>0$):
\begin{itemize} 
\item  $\wta$ is complete with respect to the weak topology, cf. \cite[\S 5.1]{Col08}.
\item Note $\wta^{[r,+\infty]}[{1}/{[\overline \pi]}]$ is precisely the ring $\wta^{(0, p/((p-1)r)]}$ in \cite[\S 5.2]{Col08}, thus is complete  with respect to the $W^{[r, r]}$  by \cite[Prop 5.6]{Col08}; the subring $\wta^{[r,+\infty]}$ is precisely the unit ball, hence is also $W^{[r, r]}$-complete. Finally, $\wta^{[r,r]}$ is $W^{[r, r]}$-complete  by definition.
\item All the ``$\mathbf{B}$"-rings we defined so far satisfy $\wtb^\ast=\wta^\ast[1/p]$ (with $\ast$ certain decorations) and hence are equipped with   inductive limit topologies.
\end{itemize}
Consider the following diagram where all arrows are inclusion maps.
\[
\begin{tikzcd}
\wta \arrow[d] & {\wta^{[r,+\infty]}} \arrow[l] \arrow[rr, "\mathrm{closed}"] \arrow[d] &                                                 & {\wta^{[r,r]}} \arrow[d] \\
\wtb           & {\wta^{[r,+\infty]}[\frac{1}{[\overline \pi]}]} \arrow[l] \arrow[r]   & {\wtb^{[r,+\infty]}} \arrow[r, "\mathrm{dense}"] & {\wtb^{[r,r]}}          
\end{tikzcd}
\]
Then all maps are continuous with respect to above topologies. (The only non-trivial fact to note is that $ \wta^{[r,+\infty]} \into \wta$ is continuous. 
As  $ \wta^{[r,+\infty]}$ is a metric space, it is equivalent to check  sequential continuity; this is easy and is mentioned in the proof of \cite[Prop 5.6]{Col08}.)
\end{construction}

\begin{remark} \label{rem more integral}
We give some more remarks about the above rings.
\begin{enumerate}
\item  For $\wtb^{[r, +\infty]}$, its completion with respect to $W^{[r,r]}$ is $\wtb^{[r, r]}$; its ring of integers is $\wtb^{[r, +\infty]} \cap \wta^{[r, r]}$ which strictly contains $\wta^{[r, +\infty]}$. However, $\wta^{[r,+\infty]}$ is   the ring of integers of $\wta^{[r,+\infty]}[{1}/{[\overline \pi]}]$.

 \item   $\wta^{[r,s]}$ is not the completion of $\wt{\mathbf{A}}^{[r, +\infty]}$  with respect to $W^{[r,s]}$ (already so when $r=s$).

\item \label{itemrssenintegral} In continuation with Rem \ref{rem int oc ring}\eqref{eq rem int oc ring}, the ring $\wta^{[r,s]}$ does not satisfy TS-1 descent; the ring $\wta^{[r,s]}[{1}/{[\overline \pi]}]$ does, but is a ring never used in the literature.

\item   We also refer to \cite[Rem. 2.1.11, Rem 2.1.15]{GP21} for some more integral subtleties.
 
\item We quickly mention that  we purposedly avoid defining $\wtb^I$'s  with $0 \in I$ as they are not used in this paper; cf. e.g. \cite[Rem 2.1.2]{GP21} for their  subtlies.  
\end{enumerate}
\end{remark}

\subsection{Imperfect rings}
\begin{notation}
\begin{enumerate}
   \item  Denote $K_{0, p^\infty}:=\cup_{n\geq 1}K_0(\mu_n)$.
Let $\mathbb{A}^+_{K_{0, p^\infty}} : =  W(k)[\![T]\!]$,  let $\mathbb{A}_{K_{0,p^\infty}}$ be the $p$-adic completion of $\mathbb{A}^+_{K_{0,p^\infty}}[1/T]$, and let   $\mathbb{B}_{K_{0,p^\infty}}:=\mathbb{A}_{K_{0,p^\infty}}[1/p]$.
We have an embedding $\mathbb{B}_{K_{0,p^\infty}} \inj \wt{\mathbf{B}}$ 
via  $T\mapsto [\underline \varepsilon]-1$.
 Let $\mathbb{B}$ be the $p$-adic completion of the maximal unramified extension of $\mathbb{B}_{K_{0,p^\infty}}$ inside $\wt{\mathbf{B}}$; and let $\mathbb{A}$ be the ring of integers. Let $\mathbb{A}_{\Kpinfty}:= \mathbb{A}^{G_{\Kpinfty}}$ and
 $\mathbb{B}_{\Kpinfty}:= \mathbb{B}^{G_{\Kpinfty}}$; in general, one cannot express $\mathbb{B}_{\Kpinfty}$ explicitly using $T$.

 \item Let $\mathbf{A}^+_{K_\infty} : = W(k)[\![u]\!]$, let $\mathbf{A}_{K_\infty}$ be the $p$-adic completion of $\mathbf{A}^+_{K_\infty}[1/u]$, and let $\mathbf{B}_{K_\infty}:=\mathbf{A}_{K_\infty}[1/p]$.
We have an embedding $\mathbf{B}_{K_\infty}\inj \wt{\mathbf{B}}$
via $u\mapsto [\underline \pi]$.
 Let $\mathbf{B}$ be the $p$-adic completion of the maximal unramified extension of $\mathbf{B}_{K_\infty}$ inside $\wt{\mathbf{B}}$, and let $\bfa$ be the ring of integers.
\end{enumerate}
\end{notation}

\begin{construction} \label{constr: many rings}
Let $(X, Y) \in \{ (\mathbb{A}, \bbb),  (\bfa, \bfb), (\wta, \wtb) \}$ be a pair of symbols (all rings considered as subrings of $\wtb$). Let $r>0$.
\begin{enumerate}
\item  (Galois invariants of rings). Let $\ast$ be either empty or a field in $\{ \kpinfty, \kinfty, L\}$; in the later case, let $X_\ast =X^{\gal(\barK/\ast)}$.
For example, $X_\ast$ could be $\mathbb{A}_\kpinfty$ or $\wta_\kinfty$. Define $Y_\ast$ similarly.

\item (Overconvergent rings). Let  
\[ X_\ast^{ [r, +\infty]} :=X_\ast \cap \wta^{[r, +\infty]}, \quad Y_\ast^{ [r, +\infty]} :=Y_\ast \cap \wtb^{[r, +\infty]} \] 
\[X_\ast^\dagger :=X_\ast \cap \wta^\dagger, \quad  Y_\ast^\dagger :=Y_\ast \cap \wtb^\dagger \] 
 
\item  \label{itemrigocring} (Rigid-overconvergent rings). 
Define $Y_\ast^{[r, s]} $ as the $W^{[r, s]}$-completion of $Y_\ast^{[r, +\infty]}$ inside $\wtb^{[r, s]}$.
Define
\[Y_\ast^{[r, +\infty)} :=\projlim_{s\to +\infty} Y_\ast^{[r, s]}  \]
and
\[Y_{\rig, \ast}^\dagger =\cup_{r \geq 0} Y_\ast^{[r, +\infty)} \]
These are the \emph{``analytic"} rings that we will use.
(One could also define $X_\ast^{[r, s]} $ as subring of integers of $Y_\ast^{[r, s]} $; but they are not needed; cf. Rem \ref{rem more integral}\eqref{itemrssenintegral}).

\item (Union of Frobenius inverses). 
For any of the $X_{\bullet, \ast}^I$ defined above where $I$ is an interval, define 
\[ X_{\bullet, \ast, \infty}^I =\cup_{m \geq 0} \varphi^{-m}(X_{\bullet, \ast}^{p^mI}) \]
where the Frobenius map $\varphi^{-1}$ is always defined in some bigger rings (which is always clear from the context) where $\varphi$ is bijective.
Define similar rings for any $Y_{\bullet, \ast}^I$ rings.
\end{enumerate}
\end{construction}

 \section{Notations: groups and  locally analytic vectors} \label{sec:lav} 

We   explain the structure of the $p$-adic Lie group $\hatg$ and its Lie algebra. We then review the notion of (higher) locally analytic vectors, and summarize some useful results in locally analytic cohomologies.
 
\subsection{Groups} \label{subsec: group notation}

\begin{Notation} \label{nota hatG}
Let $\hat{G}=\gal(L/K)$ be as in Notation \ref{notafields}, which is a $p$-adic Lie group of dimension 2. We recall the structure of this group in the following.
\begin{enumerate}
\item Recall that:
\begin{itemize}
\item if $K_{\infty} \cap K_{p^\infty}=K$ (always valid when $p>2$, cf. \cite[Lem. 5.1.2]{Liu08}), then $\gal(L/K_{p^\infty})$ and $\gal(L/K_{\infty})$ generate $\hat{G}$;
\item if $K_{\infty} \cap K_{p^\infty} \supsetneq K$, then necessarily $p=2$, and $K_{\infty} \cap K_{p^\infty}=K(\pi_1)$ (cf. \cite[Prop. 4.1.5]{Liu10}), and hence $\gal(L/K_{p^\infty})$ and $\gal(L/K_{\infty})$   generate an open subgroup  of $\hat{G}$ of index $2$.
\end{itemize}

\item Note that:
\begin{itemize}
\item $\gal(L/K_{p^\infty}) \simeq \Zp$, and we fix
$\tau \in \gal(L/K_{p^\infty})$ \emph{the} topological generator such that
\begin{equation} \label{eq1tau}
\begin{cases} 
\tau(\pi_i)=\pi_i\mu_i, \forall i \geq 1, &  \text{if }  \Kinfty \cap \Kpinfty=K; \\
\tau(\pi_i)=\pi_i\mu_{i-1}=\pi_i\mu_{i}^2, \forall i \geq 2, & \text{if }  \Kinfty \cap \Kpinfty=K(\pi_1).
\end{cases}
\end{equation}
 
\item $\gal(L/K_{\infty})$ and $\gal(K_{p^\infty}/K)$ are   not necessarily pro-cyclic when $p=2$. (This does cause some trouble, cf. Notation \ref{nota: gamma generator}).
\end{itemize}
\end{enumerate}
\end{Notation}

\begin{notation} \label{nota: gamma generator}
Let $p>2$. In this case  $\gal(L/\kinfty) \simeq \gal(\kpinfty/K)=\gammak \subset \z_p^\times$, and all these groups are pro-cyclic.
Fix a topological generator $\gamma$ of $\gal(L/\kinfty)$; this is indeed the \emph{only} reason we need to assume $p>2$. See   Rem. \ref{rem: intro p equal 2} to see why we need this for Lem. \ref{lemgroupiwasawa}.
 (As  side-note: when $p=2$, by \cite[Lem. 2.1]{Wangxiyuan}  we can always choose \emph{some} $\{\pi_n\}_{n \geq 0}$ so that $K_{\infty} \cap K_{p^\infty}=K$; this makes  $\gal(L/\kinfty) \simeq \gal(\kpinfty/K)$, but the problem of pro-cyclicity remains.)
\end{notation}

We collect some formulae that we use often; one can easily verify them.
 
\begin{lemma} \label{lem:delta elt}
Let $p>2$. 
\begin{enumerate}
\item $\gamma \tau \gamma^{-1} =\tau^{\chi(\gamma)}$.
\item Define $$\delta = \frac{\tau^{\chi(\gamma)} - 1}{\tau - 1} = \sum_{n \geq 1} \binom{\chi(\gamma)}{n} (\tau -1)^{n-1} \in \zp[[\tau-1]].$$
Then we have $$(\delta-\gamma)(\tau -1) =(1-\tau^{\chi(\gamma)})(\gamma-1). $$
\end{enumerate}
\end{lemma}

\begin{lemma} \label{lemgroupiwasawa}
Let $p>2$. 
Let $\zp[[\hat G]] $ be the Iwasawa algebra. We have the following exact sequence
\[ \zp[[\hat{G}]]  \xrightarrow{(\tau^{\chi(\gamma)}-1, \delta-\gamma)} \zp[[\hat{G}]]\oplus \zp[[\hat{G}]] \xrightarrow{( \gamma-1, \tau-1)}  \zp[[\hat{G}]] \xrightarrow{\epsilon} \zp \to 0 \]
Here $\epsilon$ is the augmentation map, and all other maps are defined using \emph{right} multiplication.
As a consequence, if $\wtm$ is an abelian group with continuous $\hatg$-action, then $\rg(\hatg, \wtm)$ is quasi-isomorphic to
\[C_{\gamma, \tau}(\wtm)  =   [  \wt{M}\xrightarrow{\gamma-1, \tau-1} \wt{M}\oplus \wt{M} \xrightarrow{\tau^{\chi(\gamma)}-1, \delta-\gamma} \wt{M} ] \]
\end{lemma}
\begin{proof}
Note indeed the first three terms form a Lazard--Serre resolution of the trivial representation $\zp$, cf. e.g. \cite[Thm. 5.7]{RJRC22}.
To check exactness, the only non-trivial part is exactness at the second term.
Suppose $(x, y) \in \zp[[\hat G]]\oplus \zp[[\hat G]]$ such that \begin{equation} \label{eqxygamma}
x(\gamma-1)=y(\tau-1).
\end{equation}
It suffices to show $(\tau-1) \mid x$. Note  $(\tau -1)\cdot  \zp[[\hat G]] \subset \zp[[\hat G]]$  is a  two-sided ideal. Thus we can consider \eqref{eqxygamma} in the quotient ring $\zp[[(\gamma-1)]]$. Then it is obvious $\bar x=0$.
\end{proof}

\subsection{Locally analytic vectors} \label{subsec: lie alg LAV}

We review locally analytic vectors in this subsection, and \emph{higher} locally analytic vectors in the next subsection. For simplicity (enough for our purpose), we always assume
\begin{itemize}
    \item $G$ is a \emph{compact} $p$-adic Lie group.
\end{itemize}

\begin{notation}[Locally analytic and pro-analytic vectors]
\begin{enumerate}
    \item Let $(W, \|\cdot \|)$ be a continuous $\Qp$-Banach representation of $G$.
Let $W^{G\dla} \subset W$ denote the subset of locally analytic vectors. This definition naturally extends to the case where $W$ is a LB representation.

\item cf. \cite[Def. 2.3]{Ber16}.
Let $W=\projlim_i W_i$ be a Fr\'echet representation of $G$. We say that $w \in W$ is \emph{pro-analytic} if its image in $W_i$ is locally analytic for each $i$. This definition naturally extends to the case where $W$ is a LF representation. We use $W^{G\dpa}$ to denote the \emph{pro-analytic} vectors.
\end{enumerate}
 \end{notation}

\begin{rem} \label{rem: caution pro ana}
We caution on the (subtle but serious) distinction between  \emph{locally analytic vectors} and \emph{pro-analytic vectors}. In summary, our slogan is: for \emph{arithmetic} purposes, we need LF representations and their \emph{pro-analytic vectors}; but for \emph{cohomological}  questions, we always only use the LB representations and their \emph{locally analytic vectors}. 
\begin{enumerate}
    \item  The notion of locally analytic vectors can be defined in a very broad context, e.g., for any barreled locally convex representation as in \cite{Tamme_loc_ana_rep_2015ANT} (thus including LF representations). However, in general we have a \emph{strict} inclusion:
    \[ W^{G\dla} \subsetneq W^{G\dpa} \]

      \item Most \emph{locally analytic cohomology} theory (e.g., \cite{Tamme_loc_ana_rep_2015ANT}, \cite{RJRC22}) are developed for general \emph{locally analytic} representations (even over a LF space). However, in general  they do \emph{not} work for the pro-analytic representations: intuitively, cohomology in general does not commute with limit.
      
    \item \label{item proana}   We   need the     pro-analytic vectors for \emph{arithmetic} purposes. For example, for $(\phi, \Gamma)$-modules resp. $(\phi, \tau)$-modules over the \emph{Robba ring}, the relevant group actions are only \emph{pro-analytic}.
    
  \item As discussed in the introduction, our main theorems in this paper concern cohomology of $(\phi, \tau)$-modules over the  Robba ring, which as mentioned in Item \eqref{item proana} are only pro-analytic vectors. An important step in the process is to use the $\phi$-operator to ``adjust" to the locally analytic case; cf. Lem. \ref{lemphidescent} which will be \emph{repeatedly} used in the final section \S \ref{sec: coho phi tau}.    
  
\item With above items in mind, (and to avoid possible confusions), we have chosen to review theorems on locally analytic cohomology only for \emph{LB (or even just Banach) representations}; cf.  Thms. \ref{thm: Tamme ana coho} and \ref{prop: no higher lav coho compa}.
\end{enumerate}
\end{rem}

\begin{lemma} \label{lemlamod}
Let $G$ be a  compact $p$-adic Lie group.
Let $B$ be a LB (resp. LF) ring equipped with locally analytic (resp. pro-analytic) action by $G$, i.e., $B^\la=B$ (resp. $B^\pa=B$). Let $W$ be a \emph{finite} free $B$-module  equipped with a continuous $G$-action. Then the $G$-action on $W$ is automatically locally analytic (resp. pro-analytic), i.e., $W^\la=W$ (resp. $W^\pa=W$).
\end{lemma}
\begin{proof}
This is an easy consequence of the local-analyticity criterion of \cite[Cor. 2.2.6]{RC22}.  \end{proof}

We recall a theorem on Lie algebra cohomology and  \emph{locally analytic} cohomology. (cf. Rem. \ref{rem: caution pro ana} for some caution).
We use $\rg_\la$ denotes locally analytic group cohomology (defined using locally analytic cochains). 

\begin{convention}
 In this paper, unless decorated with ``$\la$" in the subscripts, all group cohomologies are meant to be continuous cohomologies; sometimes we put ``$\cont$" in the subscript to emphasize the differences.   
\end{convention}

\begin{theorem}   \label{thm: Tamme ana coho}
    Let $W$ be a  locally analytic  LB $G$-representation, i.e., $W^\la=W$.
     \begin{enumerate}
     \item We have
\[  \rg_\cont(G, W) \simeq \rg_\la(G, W) \simeq   \rg(\Lie G, W)^G\]     
 
         \item We have
\[ \rg(\Lie G, W) \simeq \injlim_{G'\subset G \text{ open}} \rg_\la(G', W) \]
     \end{enumerate}
      \end{theorem}
     \begin{proof}
    The comparison 
    \[\rg_\cont(G, W) \simeq \rg_\la(G, W)\]
    is first proved by Lazard \cite{Lazard-65-IHES} when $W$ is a finite dimensional $\qp$-vector spaces; the proof works for general case, as observed in \cite[Thm. 3.8]{Kedlaya-Hochschild-Serre-analytic}. The remaining comparisons (also generalizing those of \cite{Lazard-65-IHES}) is proved in  \cite[page 938, Main Theorem]{Tamme_loc_ana_rep_2015ANT} and \cite[Cor. 21]{Tamme_loc_ana_rep_2015ANT}.
     \end{proof}

     We specialize the discussions to the group $\hatg$.
     
\begin{notation}
For $g\in \hat{G}$, let $\log g$ denote  the (formally written) series $(-1)\cdot \sum_{k \geq 1} (1-g)^k/k$. Given a $\hat{G}$-locally analytic representation $W$, the following two Lie-algebra operators (acting on $W$) are well defined:
\begin{itemize}
\item  for $g\in \gal(L/\kinfty)$ close enough to identity, one can define $\nabla_\gamma := \frac{\log g}{\log(\chi_p(g))}$; in the case when $\gal(L/\kpinfty)$-action on $W$ is trivial, one can (use the same formula and notation to) define $\nabla_\gamma := \frac{\log g}{\log(\chi_p(g))}$ for $g \in \gal(\kpinfty/K)$ close enough to identity;

\item for $n \gg 0$ hence $\tau^{p^n}$ close enough to identity, one can define $\nabla_\tau :=\frac{\log(\tau^{p^n})}{p^n}$.
\end{itemize}
Clearly, these two Lie-algebra operators form a $\qp$-basis of $\Lie(\hat{G}$).
(Caution: in Notation \ref{nota: gamma generator}, we can only fix a pro-generator $\gamma$ when $p>2$. Note however the notation $\nabla_\gamma$ here is well-defined for any $p$, and indeed is not related with ``$\gamma$" directly. 
We have chosen the notation $\nabla_\gamma$ for its simplicity and its resemblance with $\nabla_\tau$. We hope this does not cause confusion.)
\end{notation}

 \begin{lemma} \label{lemresolutionliealg} (Let $p$ be any prime). 
 Let $\hat{\fkg}=\Lie \hatg$, and let $U(\hat{\fkg})$ be the universal enveloping algebra (over $\qp$). The following sequence is exact.
 \[ U(\hat{\fkg})  \xrightarrow{\nabla_\tau, 1-\nabla_\gamma} 
  U(\hat{\fkg})\oplus U(\hat{\fkg}) 
  \xrightarrow{ \nabla_\gamma, \nabla_\tau}  U(\hat{\fkg})  \xrightarrow{\epsilon} \qp \to 0\]
As a consequence, if $\wtm$ is an abelian group with a continuous Lie algebra action by $\hat{\fkg}$, then the Lie algebra cohomology $\rg(\hat{\fkg}, \wtm)$ is quasi-isomorphic to
 \[C_{ \nabla_\gamma, \nabla_\tau} (\wtm):=[ \wtm \xrightarrow{ \nabla_\gamma, \nabla_\tau}  \wtm\oplus \wtm   \xrightarrow{\nabla_\tau, 1-\nabla_\gamma}    \wtm ] \]
 \end{lemma}
 \begin{proof}
 This is straightforward; e.g., one can use similar argument as in Lem. \ref{lemgroupiwasawa}.
Note the sequence on $U(\hat{\fkg})$ is a Chevalley--Eilenberg resolution of the trivial representation $\qp$, cf. e.g. \cite[Prop. 5.12]{RJRC22}.
 \end{proof}

\subsection{Higher locally analytic vectors}

We quickly review the notion of higher locally analytic vectors.
We refer the readers to \cite{Pan22}, \cite{RJRC22} and \cite{Poratlav} for details.
We start with the notion of higher \emph{analytic} vectors. Consider a  $p$-adic Lie group $N$ such that   there is a homeomorphism (of $p$-adic manifolds) $c: N \to \mathbb{Z}_p^d$. Let $W$ be a $N$-Banach representation over $\qp$ (not necessarily a   locally analytic representation).
 There is an isometry $$W\widehat{\otimes}_{\Q_{p}}\mathcal{C}^{\an}\left(N,\Q_{p}\right)\cong\mathcal{C}^{\an}\left(N,W\right),$$
where $\mathcal{C}^{\an}\left(N,W\right)$ is the space of $W$-valued
\emph{analytic} functions on $N$. (The analyticity condition here---not just locally analyticity---is the reason we require $N \simeq \mathbb{Z}_p^d$).  
 We then have 
 \[  \left(\mathcal{C}^{\an}\left(N,W\right)\right)^{N} \simeq W^{N\dan}, \quad \text{ via } f\mapsto f(1);\] 
this implies that the functor $W\mapsto W^{N\dan}$ is left exact. Thus we can define the right derived functors
for $i\geq 1$:
\[
\mathrm{R}_{G\dan}^{i}\left(W\right):= H^{i}\left(G,W\widehat{\otimes}_{\Q_{p}}\mathcal{C}^{\an}\left(G, \qp \right)\right).
\]

 Now let $G$ be a general compact $p$-adic Lie group; there exists an open normal subgroup $G_1$ such that   $G_{n}:=G_1^{p^{n-1}}$  are subgroups of $G_1$ and there exists a homeomorphism $c: G_1 \to  \mathbb{Z}_p^d $ such that $c(G_n)=(p^n \mathbb{Z}_p)^d$ for all $n$. See the paragraph below \cite[Prop 2.3]{BC16} for details.

\begin{defn}
 Let $W$ be a $G$-Banach representation over $\qp$.    Define 
\[
\mathrm{R}_{G\dla}^{i}\left(W\right):=\varinjlim_{n}\mathrm{R}_{G_{n}\dan}^{i}\left(W\right)=\varinjlim_{n} H^{i}\left(G_{n},W\widehat{\otimes}_{\Q_{p}}\mathcal{C}^{\an}\left(G_{n}, \qp \right)\right).
\]
 These cohomology groups are called the higher locally analytic vectors of $W$.
    \end{defn}

\begin{thm} \label{prop: no higher lav coho compa} Let $W$ be a $G$-Banach representation over $\qp$. 
 Suppose $W$ has no higher locally analytic vectors, that is
 \[ \mathrm{R}_{G\dla}^{i}(W)=0, \quad \forall i \geq 1.\]
 Then
 \[ \rg_\cont(G, W) \simeq \rg_\cont(G, W^\la) \simeq \rg_\la(G, W^\la) \simeq   \rg(\Lie G, W^\la)^G\]
 In particular, we have isomorphism of cohomology groups  \[ H^i(G, W) \simeq  H^i(G, W^\la) \simeq  H^i_\an(G, W^\la)  \simeq(H^i(\Lie G, W^\la))^G \]
\end{thm}
\begin{proof} With Thm. \ref{thm: Tamme ana coho} in mind, the only remaining comparison is between $\rg_\cont(G, W)$  and  $\rg_\cont(G, W^\la)$: see \cite[Cor. 1.6]{RJRC22}, or the all encompassing diagram \cite[Thm. 6.3.4]{RJRC23}.
 For the comparison of cohomology groups,  cf.  \cite[footnote of Thm. 1.7]{RJRC22} and   \cite[Rem. 5.6]{RJRC22}
\end{proof}

\begin{rem} 
Thm. \ref{prop: no higher lav coho compa}  is developed in \cite{RJRC22, RJRC23} in a much broader context, for non-compact groups $G$, and for solid representations. We mention that for a general  $W$, the differences between  $\rg_\cont(G, W)$  and  $\rg_\cont(G, W^\la)$ are accounted for by continuous group cohomology of \emph{higher locally analytic vectors}, cf. the spectral sequence in \cite[Thm. 1.5]{RJRC22}. 
\end{rem}

\begin{example}[\cite{Poratlav}] \label{example: porat}
    Let $r>\frac{p}{p-1}$. Let $M$ be a finite free module over  $\wtb^{[r,s]}_\kpinfty$  (resp. $\wtb^{[r,s]}_L$) with a continuous semi-linear $\gammak$-action (resp. $\hatg$-action). 
    Then $M$  has no higher locally analytic vectors.  
\end{example}
\begin{proof} 
    Via \cite[Example 5.5(2)]{Poratlav}, one can apply \cite[Prop. 5.3]{Poratlav} with respect to the ring $\wtb^{[r,s]}_\kpinfty$ (resp. $\wtb^{[r,s]}_L$).
\end{proof}

\section{Equivalence of module categories} \label{sec:equiv cats}
In this section, we define the \emph{many} categories of  $(\varphi, \Gamma)$-modules and  $(\varphi, \tau)$-modules, and prove the categorical equivalence theorems; cf. Theorem  \ref{thmequivetale} resp. \ref{thmequidagger} for the \'etale resp. non-\'etale case.
 In \S\ref{subsec: diff op phi tau}, we introduce an important differential operator on  $(\phi, \tau)$-modules over the Robba ring.

\subsection{Definition of categories} \label{subsec: def categories}
\begin{defn} \label{ring and mod}
Let $Q$ be a topological ring with a continuous $\hat G$-action.
Let $P \subset Q^{\tau=1}$ be a  subring stable under $\gammak$-action.
Let $R \subset Q^{\gammak=1}$ be a subring.
\begin{enumerate}
\item Let $\mod_\gammak(P)$ be the category where an object is a  finite free $P$-module  equipped with a continuous semi-linear  $\gammak$-action.

\item Let $\mod_\hatg(Q)$ be the category where an object is a finite free $Q$-module  equipped with a continuous semi-linear   $\hatg$-action.   

\item Let $\mod_{\tau}(R, Q)$ be the category where an object is a pair $(M, \wtm)$ where  $M$ is a finite free $R$-mod, $\wtm=M\otimes_R Q$ is equipped with a continuous semi-linear $\hat G$-action such that $M \subset \wtm^{\gammak=1}$.
\end{enumerate}
 There are two obvious functors
\begin{equation}
 \mod_\gammak(P)  \to   \mod_\hatg(Q) \leftarrow   \mod_{\tau}(R, Q)
\end{equation} 
where the first one is defined by $N \mapsto N \otimes_P Q$, and the second one is defined by $(M, \wtm) \mapsto \wtm$.
\end{defn}

\begin{defn} \label{def: phi mod cat}
Use notations in Def. \ref{ring and mod}.
Suppose there is a ring endomorphism $\varphi: Q \to Q$ (in practice: the Frobenius) which commutes with $\hat{G}$-action, and suppose $P, R$ are stable under $\varphi$.
\begin{enumerate}
\item Let $\mod_{\varphi,\gammak}(P)$   be the category where an object is an  $N \in \mod_\gammak(P)$  equipped with a continuous, $\varphi$-semi-linear, $\gammak$-commuting and \'etale map $\varphi: N \to N$; here the \'etale condition says that the induced map $1\otimes \varphi: P\otimes_{\varphi, P} N \to N$ is bijective. (cf. Rem. \ref{rem: etale condition} for a caution on the \'etaleness condition.)

\item Let $\mod_{\varphi, \hatg}(Q)$  be the analogously defined category.

 \item Let $\mod_{\varphi, \tau}(R, Q)$ be the category where an object is $(M, \wtm) \in \mod_{\tau}(R, Q)$ with $M$ equipped with a continuous $\varphi$-semilinear \'etale map $\varphi: M \to M$ such that the induced $\varphi: \wtm \to \wtm$ commutes with $\hatg$.
\end{enumerate}
\end{defn} 

\begin{rem}
    Caution: the inclusions $P \subset Q^{\tau=1}, R \subset Q^{\gammak=1}$ in Def. \ref{ring and mod} and \ref{def: phi mod cat} are in general \emph{strict}. For example, this is the case for the category $\mod_{\varphi, \tau}( \A_\kinfty       ,\wta_L   )$.  
\end{rem}


\subsection{\'Etale  modules and overconvergent modules}
  
\begin{theorem} \label{thmequivetale}
All categories  in the following diagrams (defined via   Def. \ref{def: phi mod cat}) are equivalent to $\rep_\gk(\zp)$.
\begin{enumerate}
\item The categories of   \'etale modules:
\begin{equation} \label{etalediag}
\begin{tikzcd}
{\mod_{\varphi, \gammak}(\mathbb{A}_\kpinfty)} \arrow[r]        & {\mod_{\varphi, \gammak}(   \wta_\kpinfty )} \arrow[d]          \\
                                                                & {\mod_{\varphi, \hat{G}}(    \wta_L )}                          \\
{ \mod_{\varphi, \tau}( \A_\kinfty       ,\wta_L   )} \arrow[r] & { \mod_{\varphi, \tau}(       \wta_\kinfty  ,\wta_L)} \arrow[u]
\end{tikzcd}
\end{equation} 

\item The  categories of overconvergent \'etale modules:
\begin{equation} \label{eqocmod}
\begin{tikzcd}
{\mod_{\varphi, \gammak}(    \mathbb{A}_\kpinfty^\dagger)} \arrow[r]            & {\mod_{\varphi, \gammak}(  \wta_\kpinfty^\dagger   )} \arrow[d]                 \\
                                                                                & {\mod_{\varphi, \hat{G}}(   \wta_L^\dagger  )}                                  \\
{ \mod_{\varphi, \tau}( \A_\kinfty^\dagger       ,\wta_L^\dagger   )} \arrow[r] & { \mod_{\varphi, \tau}(       \wta_\kinfty^\dagger  ,\wta_L^\dagger)} \arrow[u]
\end{tikzcd}
\end{equation} 
 \end{enumerate}
\end{theorem}
\begin{proof}
All results are known, we briefly review the proofs.
All categories in diagram \eqref{etalediag} are well-known. To prove their equivalences with $\rep_\gk(\zp)$, it suffices to prove the relevant ``mod $p$" categories are equivalent with $\rep_\gk(\fp)$. 
These follow quickly from the facts that $\kpinfty$ and $\kinfty$ are  APF extensions over $K$, and $C$ is a perfectoid field; cf. \cite{Fon90, Car13} for details.
These categories are further equivalent to those in  \eqref{eqocmod} by the overconvergence theorems of \cite{CC98} (for $(\phi, \Gamma)$-modules) and \cite{GL20, GP21} (for $(\phi, \tau)$-modules).
\end{proof}

\begin{rem} \label{rem: etale condition}
    \begin{enumerate}
        \item Obviously, the category $\rep_\gk(\qp)$ is equivalent to the isogeny category of any category in the diagrams in Thm. \ref{thmequivetale}.

        \item Indeed, we avoid introducing the notation ``$\mod_{\varphi, \gammak}(\mathbb{B}_\kpinfty)$", which would be a \emph{wrong} category if we define it by \emph{literally} following Def. \ref{def: phi mod cat}. (The correct definition is to use the isogeny category of $\mod_{\varphi, \gammak}(\mathbb{A}_\kpinfty)$, cf. e.g. \cite[Def I.4.2]{CC98}).
        This is a well-known issue concerning \'etaleness in the rational case. We hope this shall not cause confusion for the   readers.
    \end{enumerate}
\end{rem}

\subsection{Rigid-overconvergent modules}

We prove   equivalences of module categories over ``rigid-overconvergent rings". For simplicity, let  $\wtb_{\rig, L}^{\dagger, \pa} :=(\wtb_{\rig, L}^\dagger)^\hatgpa $.

\begin{theorem}\label{thmequidagger}  
We have a diagram of equivalences of categories (defined via   Def. \ref{def: phi mod cat}):
\begin{equation}
\label{eqocdagger}
\begin{tikzcd}
                                                                                  & {{\mod_{\varphi, \gammak}(\bB_{\rig, \kpinfty}^\dagger)} } \arrow[r] \arrow[d]                                        & {{\mod_{\varphi, \gammak}(\wtb_{\rig, \kpinfty}^\dagger)} } \arrow[d]                      \\
                                                                                  & { {\mod_{\varphi, \hatg}(\wtb_{\rig, L}^{\dagger, \pa})} } \arrow[r]                                                  & { {\mod_{\varphi, \hatg}(\wtb_{\rig, L}^\dagger)}           }                              \\
{ {\mod_{\varphi, \tau}(\bfB_{\rig, \kinfty}^\dagger, \wtb_{\rig, L}^\dagger )} } & { {\mod_{\varphi, \tau}(\bfB_{\rig, \kinfty}^\dagger, \wtb_{\rig, L}^{\dagger, \pa})} } \arrow[r] \arrow[u] \arrow[l] & { {\mod_{\varphi, \tau}(\wtb_{\rig, \kinfty}^\dagger, \wtb_{\rig, L}^\dagger)} } \arrow[u]
\end{tikzcd}
\end{equation}
 They are further equivalent to the category of $B$-pairs (as defined in \cite{Ber08ANT}).
\end{theorem} 
\begin{proof}
All results are (essentially) known; we briefly review the proofs.
The equivalence 
\[{\mod_{\varphi, \gammak}(\bB_{\rig, \kpinfty}^\dagger)} \simeq {\mod_{\varphi, \gammak}(\wtb_{\rig, \kpinfty}^\dagger)},  \]
as well as their equivalence with  the category of $B$-pairs,
 is well-known, cf. \cite[Thm. 1.2]{Berjussieu} and \cite[Thm. 2.2.7]{Ber08ANT}; arguments there also lead to equivalences between all categories in the right most column.

Consider the second row. To see
\[ {\mod_{\varphi, \hatg}(\wtb_{\rig, L}^{\dagger, \pa})}  \simeq {\mod_{\varphi, \hatg}(\wtb_{\rig, L}^\dagger)}  \]
one needs to show any $\wtd_{\rig, L}^\dagger \in {\mod_{\varphi, \hatg}(\wtb_{\rig, L}^\dagger)}$ \emph{descends  to a  pro-analytic module}.
But we already know it descends to an object $\bbD_{\rig, \kpinfty}^\dagger \in {\mod_{\varphi, \gammak}(\bB_{\rig, \kpinfty}^\dagger)}$, and the $\gammak$-action on $\bbD_{\rig, \kpinfty}^\dagger$ is always pro-analytic by Lem. \ref{lemlamod} (also essentially observed in \cite[Prop. III.1.1]{Ber08Ast}.)

Consider the third row. The equivalence
\[ {\mod_{\varphi, \tau}(\bfB_{\rig, \kinfty}^\dagger, \wtb_{\rig, L}^{\dagger, \pa})} \simeq {\mod_{\varphi, \hatg}(\wtb_{\rig, L}^{\dagger, \pa})}    \]
  follows from \cite[Prop. 6.1.6, Rem. 6.1.7]{GP21}. The chain of functors (induced by inclusion of rings)
\[ {\mod_{\varphi, \tau}(\bfB_{\rig, \kinfty}^\dagger, \wtb_{\rig, L}^{\dagger, \pa})}  \to   {\mod_{\varphi, \tau}(\bfB_{\rig, \kinfty}^\dagger, \wtb_{\rig, L}^\dagger )}  \to  {\mod_{\varphi, \tau}(\wtb_{\rig, \kinfty}^\dagger, \wtb_{\rig, L}^\dagger)}  \] 
implies equivalence of the middle category with others.
\end{proof}

\subsection{Notation of modules}
  We shall study cohomology theories from next section. We introduce notations for the modules.

  \begin{convention} \label{rem: big than alpha}
      In this paper, we will often construct modules defined on a ring ``over some interval", cf. below. From now on,  whenever we use an interval $I$  (of the form $[r, s]$ or $[r, s)$ with $0<r \leq s \leq +\infty$), we can always make $\mathrm{min}(I)  \gg 0$ with no harm.
In particular, we will always assume $\mathrm{min}(I) >\mathrm{max} \{\alpha, \frac{p}{p-1}\}$  where $\frac{p-1}{p}$ is needed in Example \ref{example: porat}, and $\alpha$ is a certain constant (depending only on $K$ and $\kinfty$) to make Prop. \ref{propverifyaxiommono} work.
  \end{convention}
 
 \begin{notation} \label{notaetalmod}
Let
   \[ T\in \rep_\zp(\gk), \quad \text{ resp. } V\in \rep_\qp(\gk) \] 
   \begin{enumerate}
       \item    Denote the various corresponding modules in categories (resp. isogeny categories) in  Thm. \ref{thmequivetale} by
  \begin{equation} \label{obj121}
\begin{tikzcd}
\bbD_\kpinfty \arrow[r] & \wtd_\kpinfty \arrow[d]   \\
                        & \wtd_L                    \\
\bfD_\kinfty \arrow[r]  & \wt\bfD_\kinfty \arrow[u]
\end{tikzcd}
\end{equation}
and
\begin{equation}\label{obj122}
\begin{tikzcd}
\bbD_\kpinfty^\dagger \arrow[r] & \wtd_\kpinfty^\dagger \arrow[d]   \\
                                & \wtd_L^\dagger                    \\
\bfD_\kinfty^\dagger \arrow[r]  & \wt\bfD_\kinfty^\dagger \arrow[u]
\end{tikzcd}
\end{equation}

\item Suppose $r \gg 0$ (cf. Convention  \ref{rem: big than alpha}) such that $\bbD_{\kpinfty}^{\dagger}$ can be descended to some $\bbD_{\kpinfty}^{[r, +\infty]}$, that is, 
\[ \bbD_{\kpinfty}^{[r, +\infty]}\otimes_{\bbb_{\kpinfty}^{[r, +\infty]}} 
 \bbb_{\kpinfty}^{\dagger} \simeq \bbD_{\kpinfty}^{\dagger}\]
Similarly for other modules. We thus can ``descend" all modules in diagram \eqref{obj122}  to the interval $[r, +\infty]$.
\begin{equation} \label{131roc}
    \begin{tikzcd}
{\bbD_{\kpinfty}^{[r, +\infty]}} \arrow[r] & {\wtd_{\kpinfty}^{[r, +\infty]}} \arrow[d] \\
                                        & {\wtd_{L}^{[r, +\infty]}}                  \\
{\bfD_{\kinfty}^{[r, +\infty]}} \arrow[r]  & {\wtd_{\kinfty}^{[r, +\infty]}} \arrow[u] 
\end{tikzcd}
\end{equation}
   \end{enumerate}
 \end{notation}

\begin{notation} \label{notarigmod}
Let $\bbD_{\rig,\kpinfty}^\dagger \in  \mod_{\varphi, \gammak}(\bB_{\rig, \kpinfty}^\dagger)$ (which might be of slope zero or not).
\begin{enumerate}
    \item 
Consider the corresponding objects via Thm. \ref{thmequidagger}.
\begin{equation}\label{131}
\begin{tikzcd}
{\bbD_{\rig,\kpinfty}^\dagger} \arrow[r] & {\wtd_{\rig,\kpinfty}^\dagger} \arrow[d] \\
                                         & {\wtd_{\rig,L}^\dagger}                  \\
{\bfD_{\rig,\kinfty}^\dagger} \arrow[r]  & {\wtd_{\rig,\kinfty}^\dagger} \arrow[u] 
\end{tikzcd}
\end{equation}
Here we explicitly \emph{avoid} considering the module $\wtd^{\dagger, \pa}_{\rig, L}$, and it will not show up in (any of) our cohomological considerations either; cf. Rem. \ref{rem: no lpa coho}.

\item Suppose $r\gg 0$ (cf. Convention \ref{rem: big than alpha}) such that $\bbD_{\rig,\kpinfty}^{\dagger}$ can be descended to some $\bbD_{\kpinfty}^{[r, +\infty)}$, that is, 
\[ \bbD_{\kpinfty}^{[r, +\infty)}\otimes_{\bbb_{\kpinfty}^{[r, +\infty)}} 
 \bbb_{\rig,\kpinfty}^{\dagger} \simeq \bbD_{\rig,\kpinfty}^{\dagger}\]
Similarly for other modules. We thus can ``descend" all modules in diagram \eqref{131} to the interval $[r, +\infty)$.
\begin{equation} \label{131r}
    \begin{tikzcd}
{\bbD_{\kpinfty}^{[r, +\infty)}} \arrow[r] & {\wtd_{\kpinfty}^{[r, +\infty)}} \arrow[d] \\
                                        & {\wtd_{L}^{[r, +\infty)}}                  \\
{\bfD_{\kinfty}^{[r, +\infty)}} \arrow[r]  & {\wtd_{\kinfty}^{[r, +\infty)}} \arrow[u] 
\end{tikzcd}
\end{equation}

\item 
For $r \leq s <+\infty$, let
\[ \bbD_{\kpinfty}^{[r, s]} = \bbD_{\kpinfty}^{[r, +\infty)} \otimes_{\bbb_{\kpinfty}^{[r, +\infty)} } \bbb_{\kpinfty}^{[r, s]} \]
and similarly for other modules. We thus can ``base change" all modules in diagram \eqref{131r} to the interval $[r, s]$.
\begin{equation} \label{131rs}
      \begin{tikzcd}
{\bbD_{\kpinfty}^{[r, s]}} \arrow[r] & {\wtd_{\kpinfty}^{[r, s]}} \arrow[d] \\
                                        & {\wtd_{L}^{[r, s]}}                  \\
{\bfD_{\kinfty}^{[r, s]}} \arrow[r]  & {\wtd_{\kinfty}^{[r, s]}} \arrow[u] 
\end{tikzcd}
\end{equation}
\end{enumerate}
\end{notation}

\begin{notation}\label{notation: phi inf twist mod}
\begin{enumerate}
\item Let $M^\dagger$ be an object in diagram \eqref{obj122}, and let $M^{[r, +\infty]}$ be its corresponding object in diagram \eqref{131roc}.
Further define $$M^{[r, +\infty]}_\infty : =\cup_{m \geq 0} \phi^{-m}(M^{[p^mr, +\infty]}).$$

\item Let $M^\dagger_\rig$ be an object in diagram \eqref{131}, and let $M^{[r, s]}$ be its corresponding object in \eqref{131rs}. Further define $$M^{[r, s]}_\infty : =\cup_{m \geq 0} \phi^{-m}(M^{[p^mr, p^ms]}).$$
\end{enumerate}
\end{notation}

\begin{prop} \label{prop: lav of mod}
    Use notations in diagram \eqref{131rs}. We have
\[(\wtd_L^{[r,s]})^{\hatg\dla, \gal(L/\kpinfty)=1} = \bbd^{[r,s]}_{\kpinfty, \infty}  =\bbd_\kpinfty^{[r,s]}\otimes_{\bbb_\kpinfty^{[r,s]}}  \bbb^{[r,s]}_{\kpinfty, \infty}  \]
\[(\wtd_L^{[r,s]})^{\hatg\dla, \gal(L/\kinfty)=1} = \bfd^{[r,s]}_{\kinfty, \infty}  =\bfd_\kinfty^{[r,s]}\otimes_{\bfb_\kinfty^{[r,s]}}  \bfb^{[r,s]}_{\kinfty, \infty}\]
(cf. Notation \ref{notation: phi inf twist mod} for $\bbd^{[r,s]}_{\kpinfty, \infty}$ and $\bfd^{[r,s]}_{\kinfty, \infty}$.)
\end{prop}
\begin{proof}
    For the first one, see \cite[Thm. 4.4, Thm. 9.1]{Ber16}.
    For the second, it follows from  \cite[Thm. 3.4.4]{GP21}.
\end{proof}
  
\subsection{A differential operator for $(\phi, \tau)$-modules} \label{subsec: diff op phi tau}

\begin{defn} \label{defnfkt}
(cf. \cite[\S 5.1]{GP21} for full details).
Recall we have an element $[\epsilon] \in \wta^+$.
Let $t=\log([\epsilon])  \in \bcrisplus$ be the usual element.
Recall  $E(u)=\mathrm{Irr}(\pi, K_0) \in W(k)[u]$ is the Eisenstein polynomial for $\pi$.
We embed $W(k)[u] \into \ainf$ by sending   $u$ to $[\underline \pi ]$ in Notation \ref{nota:ainf}; henceforth identify $u$ with $[\underline \pi ]$. 
Define the element
\[
\lambda :=\prod_{n \geq 0} (\varphi^n(\frac{E(u)}{E(0)}))  \in \bcrisplus.\]
Define
$$ \mathfrak{t} = \frac{t}{p\lambda},$$
then it turns out $\mathfrak{t} \in \wta^+$.
  \end{defn}

\begin{lemma} \label{lem b}
\cite[Lem. 5.1.1]{GP21}
For $r \gg 0$, $\mathfrak{t}, 1/\mathfrak{t} \in 
  \wt{\mathbf{B}}^{[r, +\infty)}$. In addition, $\mathfrak{t}, 1/\mathfrak{t} \in
 (\wt{\mathbf{B}}^{[r, +\infty)}_{ L})^{\hat{G}\dpa}$.
\end{lemma}

\begin{defn} \label{defndiffwtb}
 (cf. \cite[\S 4]{Gao23}). Define
$$N_\nabla: (\wt{\mathbf{B}}_{  \rig, L}^{\dagger})^{\hat{G}\dpa} \to (\wt{\mathbf{B}}_{  \rig, L}^{\dagger})^{\hat{G}\dpa}$$ 
by setting
\begin{equation}\label{eqnnring}
{N_\nabla:=}
\begin{cases} 
\frac{1}{p\mathfrak{t}}\cdot \nabla_\tau, &  \text{if }  \Kinfty \cap \Kpinfty=K; \\
& \\
\frac{1}{p^2\mathfrak{t}}\cdot \nabla_\tau=\frac{1}{4\mathfrak{t}}\cdot \nabla_\tau, & \text{if }  \Kinfty \cap \Kpinfty=K(\pi_1), \text{ cf. Notation \ref{nota hatG}. }
\end{cases}
\end{equation}
 Note that $1/\mathfrak t$ is in $ (\wt{\mathbf{B}}^\dagger_{\rig, L})^{\hat{G}\dpa}$ by Lem \ref{lem b}, hence  division by $\fkt$ is allowed.
 A convenient and useful fact is that $N_\nabla$ commutes with $\gal(L/\kinfty)$, i.e., $gN_\nabla=N_\nabla g, \forall g\in \gal(L/\kinfty)$, cf. \cite[Eqn. (4.2.5)]{Gao23}.
 (The   $p$ (resp. $p^2$) in the denominator of \eqref{eqnnring} makes  our monodromy operator compatible with the earlier theory of Kisin in \cite{Kis06}, but \emph{up to a minus sign}. See also \cite[1.4.6]{Gao23} for the general convention of minus signs in that paper.)
  \end{defn}

Use Notation \ref{notarigmod}. Since $\tau$-action on 
$\bfd_{\rig, \kinfty}^\dagger$ is pro-analytic (by discussions in Thm. \ref{thmequidagger}), we can define
\[ \nabla_\tau: \bfd_{\rig, \kinfty}^\dagger \to \wtd_{\rig, L}^{\dagger, \pa} \]
which induces (using that $\fkt$ is a unit):
\[ N_\nabla: \bfd_{\rig, \kinfty}^\dagger \to \wtd_{\rig, L}^{\dagger, \pa} \]
This operator satisfies Leibniz rule with respect to $N_\nabla$ on $(\wt{\mathbf{B}}_{  \rig, L}^{\dagger})^{\hat{G}\dpa}$ in Def. \ref{defndiffwtb}.

\begin{prop} \label{prop: N nabla operator}
Use notations in above paragraph. The image of $N_\nabla$ lands inside $\bfd_{\rig, \kinfty}^\dagger$ and hence induces an operator:
\[   N_\nabla: \bfd_{\rig, \kinfty}^\dagger \to  \bfd_{\rig, \kinfty}^\dagger \]
\end{prop}
\begin{proof}
Use exactly the same argument as \cite[Thm. 4.2.1]{Gao23} (removing all appearances of ``$V$" there).
\end{proof}

\section{Axiom: continuous group cohomology} \label{sec: axiom gp coho}

In this section, we consider continuous group cohomology for the group $\hat{G}$ (and its subgroups). We freely use notations from \S \ref{subsec: group notation} and \S \ref{subsec: def categories}.

\begin{defn}  \label{defn: C tau Mq}
Let $p>2$. Use module categories from Def. \ref{subsec: def categories}.
\begin{enumerate}
\item Let $N \in \mod_\gammak(P)$, define a complex
\[ C_\gamma(N): = [N \xrightarrow{\gamma-1} N] \]

\item Let $\wtm \in \mod_\hatg(Q)$, define a double complex $DC_{\gamma, \tau}(\wtm)$:
\[ \begin{tikzcd}
\wtm \arrow[r, "\tau-1"] \arrow[d, "\gamma-1"'] & \wtm \arrow[d, "\delta-\gamma"] \\
\wtm \arrow[r, "1-\tau^{\chi(\gamma)}"]         & \wtm                           
\end{tikzcd} \]
The associated total complex (as already appeared in Lem. \ref{lemgroupiwasawa}) is
\[C_{\gamma, \tau}(\wtm): =   [  \wt{M}\xrightarrow{\gamma-1, \tau-1} \wt{M}\oplus \wt{M} \xrightarrow{\tau^{\chi(\gamma)}-1, \delta-\gamma} \wt{M} ] \]

\item Let $(M, M_Q) \in \mod_{\tau}(R, Q)$.
Define a subset (which is in general only an abelian group):
\[ M_{Q, 0}= M_Q^{\delta-\gamma=0} \]
Then define a complex
\[ C_{\tau}(M, M_Q):= [ M \xrightarrow{\tau-1} M_{Q, 0}] \]
(To avoid confusions, we prefer not to use the notation ``$C_\tau(M)$", since there might be different ``$M_Q$"'s in different theories).

\end{enumerate}
\end{defn}

\begin{axiom} \label{axiomgroupQ}
Let $Q$ be a topological ring with a continuous $\hat G$-action. Assume
\begin{enumerate}
    \item $H^1(\gal(L/\kpinfty), Q)=0$, and
    \item $H^1(\gal(L/\kinfty), Q)=0$.
\end{enumerate}
 Note 
\begin{itemize}
    \item The first condition is equivalent to say $\tau-1$ is surjective on $Q$; equivalently, it says  $\rg(\gal(L/\kpinfty), Q)$ is concentrated in degree zero since  $\gal(L/\kpinfty)$ is pro-cyclic.
    \item \emph{When $p>2$}, the second condition is equivalent to say $\gamma-1$ is surjective on $Q$; equivalently, it says  $\rg(\gal(L/\kinfty), Q)$ is concentrated in degree zero since  $\gal(L/\kinfty)$ is pro-cyclic when $p>2$.
\end{itemize} 
\end{axiom}



\begin{theorem} \label{thmaxiomnophi}
Suppose $Q$ satisfies Axiom  \ref{axiomgroupQ}.
\begin{itemize}
    \item Let  \[P:= Q^{\gal(L/\kpinfty)=1}, \quad  R:=  Q^{\gal(L/\kinfty)=1}.\]
 (this fits with (and is stronger than) the general set-up in Def. \ref{ring and mod}; use notations there); 

 \item Suppose there exists $N \in \mod_\gammak(P)$ and $(M, \wtm) \in \mod_{\tau}(R, Q)$ that map to a common object $\wtm \in \mod_\hatg(Q)$ (that is,  $\wtm=N\otimes_P Q$). (This happens, e.g. when all these categories are equivalent).
\end{itemize}
  Then all the (five) complexes in the following are canonically quasi-isomorphic to each other:
 \begin{enumerate}
 \item $\rg(\hatg, \wtm)$, \quad $\rg(\gammak, N)$.
 \item   $C_\gamma(N),  \quad C_{\gamma, \tau}(\wtm), \quad C_\tau(M, \wtm)$ (which are defined only if $p>2$).
 \end{enumerate}
 \end{theorem}
\begin{proof} 
We have a $\hat{G}$-equivariant isomorphism $\wtm=N\otimes_P Q$; in particular, as a  $\gal(L/\kpinfty)$-representation, $\wtm$ is ``trivial". By Axiom  \ref{axiomgroupQ}, we have
\[ \rg(\gal(L/\kpinfty), \wtm) =(\wtm)^{\gal(L/\kpinfty)}=N[0]. \]
Apply Hochschild-Serre spectral sequence to conclude
\[ \rg(\hatg, \wtm) \simeq \rg(\gammak, N).\]

Consider complexes in Item (2) and so $p>2$.  Lem. \ref{lemgroupiwasawa} implies
\[C_{\gamma, \tau}(\wtm) \simeq \rg(\hatg, \wtm).\]
It is also obvious that
\[\rg(\gammak, N) \simeq C_\gamma(N)\]
To relate to $C_\tau(M)$, consider the double complex $DC_{\gamma, \tau}(\wtm)$
\[ 
\begin{tikzcd}
\wtm \arrow[r, "\tau-1"] \arrow[d, "\gamma-1"'] & \wtm \arrow[d, "\delta-\gamma"] \\
\wtm \arrow[r, "\tau^{\chi(\gamma)}-1"]         & \wtm                           
\end{tikzcd} 
\]
In above argument, we proved $H^1(\gal(L/\kpinfty), \wtm) =0$, which means that $\tau^a-1$ is surjective on $\wtm$ for any $a \in \z_p^\times$. A similar argument shows $\gamma-1$ is also surjective on $\wtm$.
Since $(\delta-\gamma)(\tau -1) =(1-\tau^{\chi(\gamma)})(\gamma-1)$, we see that $\delta-\gamma$ is also surjective on $\wtm$. Thus the total complex $C_{\gamma, \tau}(\wtm)$ is quasi-isomorphic to the  ``column kernel complex" of $DC_{\gamma, \tau}(\wtm)$, which is precisely  $C_\tau(M, \wtm)$.
\end{proof}

\section{Verification:  TS-1  descent} \label{sec: ts-1}
In this section, we verify Axiom \ref{axiomgroupQ} for several rings. Indeed, our proof of verification is also ``axiomatic", using TS-1 descent techniques developed in \cite{Col08} (also axiomatized in \cite{BC08}).
Indeed, our main results are axiomatization of two results in \cite{Col08}, see Cor \ref{cor: verify axiom q for rings}.
 Here we caution a potential source of confusion: in order to verify the TS-1 axiom in the \emph{Fr\'echet} setting, it will be \emph{necessary} to introduce the valuations used in \cite{Col08} which are \emph{different} from those in \S \ref{sec: rings and gps}: see the careful discussions in Constructions \ref{cons: W-val} and \ref{cons vrr}.

\subsection{Axioms:  TS-1 and TS-1-Fr\'echet}
\begin{notation} \label{notaahv}
Let $(A, \| \cdot \|)$ be a $\Zp$-Banach algebra, and let $v$ be the valuation associated to  $\| \cdot \|$.
  Suppose $v(x)=+\infty \Leftrightarrow x=0$ and suppose for any $x, y \in A$, we have
  \[v(xy) \geq v(x)+v(y)\]
\[v(x+y) \geq \min(v(x), v(y))\]
 Let $H$ be a profinite group which acts continuously on $A$  such that 
 \[v(gx)= v(x), \forall g \in H, x \in A.\]
\end{notation}

 \begin{axiom}[Axiom TS-1, following \cite{BC08}] \label{axiom:TS-1}   
 Let $A, H, v$ be as in Notation \ref{notaahv}. 
Say the pair $(A,  H)$ satisfies \emph{Axiom (TS-1)} (with respect to $v$), if 
\begin{itemize}
\item there exists some $c_1>0$, such that for any  $H_1 \subset H_2 \subset H$   two open subgroups (no normality condition on these subgroups assumed),  there exists $\alpha \in A^{H_1}$ such that $v(\alpha)>-c_1$ and \[\sum_{g \in H_2/H_1} g(\alpha)=1\]
  Here the summation index means $g$ runs through one (indeed, any) set of representatives of the coset.
  \end{itemize} 
 \end{axiom}

\begin{rem} \label{rem:HnoG}
    We point out some difference  with usual set-up of (Colmez--)Tate--Sen axioms. For example in the original \cite[Def 3.1.3]{BC08} and recent works such as \cite[\S 5]{Poratlav} or \cite{RC22}, they always assume $H$ in   Axiom \ref{axiom:TS-1} to be the \emph{kernel} of some map $G \to \bbz_p^d$ (usually, just some $G \to \bbz_p^\times \xrightarrow{\log} \zp$) for some bigger group $G$. For the relevant discussions on TS-1 \emph{only}, this bigger $G$ is never used; $G$ only plays relevant roles for further axioms such as TS-2, TS-3 in cited references.
\end{rem}

\begin{prop} \label{prop TS1 integral}
    Suppose $(A,  H)$ satisfies  Axiom \ref{axiom:TS-1} (with respect to $v$). Then
    \[ H^1(H, A)=0 \]
\end{prop}
\begin{proof}
    This seems to be well-known since \cite{BC08}, but we could not locate an earlier reference. Fortunately, this follows as special case of \cite[Prop 5.8]{Poratlav} (which even proves $H^i(H, A)=0$ for all $i \geq 1$). Again, as mentioned   above in Rem \ref{rem:HnoG}, the bigger group $G$ in \cite[\S 5]{Poratlav} is not needed for the proof of  \cite[Prop 5.8]{Poratlav} (which only concerns TS-1). We refer to the following Rem \ref{rem:TS-1integral vanish} for subtle discussions concerning ``integrality" issue.
\end{proof}

\begin{remark}\label{rem:TS-1integral vanish}
One needs to be careful with the choice of notations in \cite[\S 5B]{Poratlav}. In Porat's notation: $\wt{\Lambda}^+$ is the ring of integers of $\wt{\Lambda}$ (which then contains $p$, say we always assume $\wt{\Lambda}$  is a $\zp$-algebra); however, one might not necessarily have $p^{-1} \in \wt{\Lambda}$. Nonethless, Porat uses the (slightly confusing) notation ``$p^{-t}\wt{\Lambda}^+ \subset \wt{\Lambda}$" to denote the subset of elements with valuation $\geq -t$. 
In particular, the terminology ``\emph{rational cohomology}" in the statement of \cite[Prop 5.8]{Poratlav} is \emph{not} about base change from $\zp$ to $\qp$, but rather from $\wt{\Lambda}^+$ to $\wt{\Lambda}$.
\end{remark}

 We now generalize the TS-1 axiom to the Fr\'echet setting, which is an axiomatization of the argument in \cite[Prop. 10.2]{Col08}. We choose $c_1=1$ for convenience of writing (which is the case in \cite[Prop. 10.2]{Col08}): indeed, in all known applications of (Colmez--)Tate--Sen descent, $c_1$ could be chosen to be \emph{any} positive number.

 \begin{axiom}[Axiom  TS-1-Fr\'echet] \label{axiom ts-1-frechet}
Let $(A, v,  H)$ as in Notation \ref{notaahv}, and suppose:
\begin{equation} \label{tsfrechetc1}
\text{  $(A,  H)$ satisfies (TS-1) with $c_1=1$.}
\end{equation}
Let $(A_\infty, H)$ be a Fr\'echet $H$-ring containing $A$ as a subring. Say $(A, A_\infty, H)$ satisfies \emph{Axiom (TS-1-Fr\'echet)}, if $A_\infty$ can be written as an inverse limit \[A_\infty =\projlim_{i\geq 0} (A_i, v_i)\] where  $ (A_i, v_i)$ is a sequence of $\zp$-Banach algebras equipped with unitary $H$-actions, such that there are continuous and $H$-equivariant ring maps $A \xrightarrow{\alpha_i} A_i$ and $A_{i+1} \xrightarrow{\beta_i} A_i$ for each $i$, such that the valuations ``\emph{increase}" along these maps; namely,
\begin{eqnarray}
\label{eqviv} v(x) & \leq &  v_i(\alpha_i(x)) ,   \quad \forall x \in A; \\
\label{eqvivione}   v_{i+1}(y) & \leq &v_i(\beta_i(y)),    \quad \forall y \in A_{i+1}.
\end{eqnarray} 
(Note \eqref{eqviv} implies for each $i$, the triple $(A_i, v_i, H)$ satisfies TS-1). 
\end{axiom}

\begin{prop}\label{prop: ts1frechet vanish}
 Suppose $(A, A_\infty, H)$ satisfies  Axiom TS-1-Fr\'echet  (with $c_1$ set as 1). Then 
\[ H^1(H, A_\infty)=0 \]
\end{prop}

\begin{proof}
This is an axiomatization  of \cite[Prop. 10.2]{Col08}.

Let $h \mapsto c_h$ be a continuous cocycle of $H$ valued in $A_\infty$. 
We will construct a sequence of elements $b_n \in   A_\infty$ for $n \geq 0$, such that the following conditions are satisfied:
 \begin{enumerate}
 \item $v_i(b_n -b_{n-1}) \geq n , \forall i \leq n-2$ (here, let $b_{-1}=0$);
 \item $v_i(c_{n, h})  \geq n+2, \forall i \leq n-1,  \forall h \in H$; 
here $c_{n, h} :=c_h -(1-h)b_n$ is the modified \emph{cocycle}. 
 \end{enumerate}
Once constructed, it is clear $b_n$ converges (in the Fr\'echet topology) to some $b\in A_\infty$ that trivializes the cocycle.

When $n=0$, take $b_0=0$. In this case, both conditions are vacuous.
 Suppose now $b_n$ is constructed, so we have
\[ v_{n-2}(b_n -b_{n-1}) \geq n,\]
and
\begin{equation}\label{nallhnpulus2}
v_{n-1}(c_{n, h})  \geq n+2, \quad \forall h \in H.
\end{equation}   
   Let $H' \subset H$ be an open subgroup such that
 \begin{equation}\label{nplus4}
  v_{n }(c_{n, \delta}) \geq n+4,\quad \forall \delta \in H'.
\end{equation}   
Fix a set of representatives $Q=\{ \tau_1, \cdots, \tau_k \}$ for $H/H'$. 
 By Axiom (TS-1) for $A$, there exists some $\alpha \in A^{H'}$ such that $v(\alpha) >-1$, and $\sum_{i=1}^k \tau_i(\alpha)=1$. 
Define the weighted summation
\[b_Q =\sum_{i=1}^k \tau_i(\alpha)c_{n, \tau_i} \] 
Define
\[ b_{n+1} =b_n+b_Q \]
We now check the conditions, (it suffices to check for the maximal index $i$ since ``$v_j \geq v_{j+1}$" by  \eqref{eqvivione}), 
\begin{enumerate}
\item $v_{n-1}(b_{n+1}-b_n) =v_{n-1}(b_Q) > (-1)+(n+2)=n+1$, because for each $i=1, \cdots, k$,
\begin{itemize}
\item $v_{n-1}(\tau_i(\alpha))=v_{n-1}( \alpha) \geq v(\alpha) >-1$, which uses \eqref{eqviv};
\item $v_{n-1}(c_{n, \tau_i}) \geq n+2$ by \eqref{nallhnpulus2}.
\end{itemize}  

\item To check the second condition. Fix one $h\in H$, write $hQ=\{ \tau_1\gamma_1, \cdots, \tau_k \gamma_k \}$ where $\gamma_i \in H'$.
Note 
\begin{eqnarray*}
c_{n+1, h}&=&c_h-(1-h)(b_n+b_Q) \\
&=&c_{n,h}-(1-h)b_Q  \\
&=&  c_{n,h}+\sum_i h\tau_i(\alpha)hc_{n, \tau_i} -\sum_i \tau_i(\alpha) c_{n,  \tau_i }  \\    
&=&    \sum_i h\tau_i(\alpha)   c_{n,h}+\sum_i h\tau_i(\alpha)hc_{n, \tau_i} -\sum_i \tau_i(\alpha) c_{n,  \tau_i }, \quad \text{ using } \sum_i h\tau_i(\alpha)=1 \\    
&=& \sum_i (h\tau_i(\alpha) \cdot c_{n, h\tau_i}) -\sum_i \tau_i(\alpha) c_{n,  \tau_i }, \quad \text{ combine first two  using cocycle condition }  \\    
&=&  \sum_i \tau_i(\alpha)   c_{n,  \tau_i\gamma_i} - \sum_i \tau_i(\alpha) c_{n,  \tau_i }, \quad \text{using expression of $hQ$, note $\gamma_j(\alpha)=\alpha$, re-order first term }  \\    
&=& \sum_i \tau_i(\alpha) \tau_i(c_{n, \gamma_i}), \quad \text{ using cocycle condition}                
\end{eqnarray*}

By \eqref{nplus4},  $c_{n, \gamma_i}$ and hence $\tau_i(c_{n, \gamma_i})$ has $v_n$-valuation $\geq n+4$, thus we can conclude
\[ v_n(c_{n+1, h}) >n+3 \]
\end{enumerate}
\end{proof}

\subsection{Verification of axioms}
  In this subsection, we verify the TS-1-Fr\'echet Axiom for certain rings. However, we shall first make some observations on the valuation    $W^{[r,r]}$ resp. $W^{[r,s]}$ on $\wtb^{[r, +\infty]}$ resp. $\wtb^{[r,s]}$ defined   in \S \ref{sec: rings and gps}: unfortunately, they can \emph{not}  be used to verify Axiom \ref{axiom ts-1-frechet}. Recall in Eqn. \eqref{eqviv} and \eqref{eqvivione} in Axiom \ref{axiom ts-1-frechet}, we require the valuations to \emph{increase} (along ring maps).
  
 \begin{construction} \label{cons: W-val}
  We have the following observations on valuations $W^{[r,r]}$ and $W^{[r,s]}$:
\begin{enumerate}
\item The element $p$ has a fixed valuation; that is, one always have $W^{[r,r]}(p)=W^{[r,s]}(p)=1$.

\item $W^{[r,s]}(x)=\inf_{\alpha \in [r,s]} \{W^{[\alpha, \alpha]}(x)$\}, hence is in some sense a ``natural" definition. In addition, the ``floor valuation" $\lfloor W^{[r,s]}\rfloor$ is \emph{exactly} the ``$p$-adic valuation", cf. \cite[Lem. 2.1.10]{GP21}.

\item \label{item3wval} (Comparison of valuations.)
\begin{enumerate}
\item  For $x \in \wtb^{[r,s]} \subset \wtb^{[r',s']}$ where $[r',s'] \subset [r,s]$, the valuations increase:   $W^{[r,s]}(x) \leq W^{[r',s']}(x)$;
\item  For $y \in \wtb^{[r, +\infty]} \subset \wtb^{[r,s]}$, the valuations \emph{decrease}:  that is $W^{[r,r]}(y) \geq W^{[r,s]}(y)$ (and  $>$ could happen). 
\item \label{item3wvalc} Even if $y \in \wta^{[r, +\infty]}[1/[\overline{\pi}]] \subset \wtb^{[r, +\infty]} \subset \wtb^{[r,s]}$,  it is still possible $W^{[r,r]}(y) > W^{[r,s]}(y)$; cf. Example \ref{example: val}.
\end{enumerate} 
\end{enumerate}
\end{construction}

We introduce the following \emph{different} valuations, which are the same as those in \cite[\S 5.2, \S 5.4]{Col08} (after re-normalization of $r$). 
In some sense, these valuations are quite ``un-natural" (cf. Items (1)(2) in the following), but we need the Item \eqref{vval3c} to verify the TS-1-Fr\'echet axiom.

\begin{construction} \label{cons vrr}
For $x \in \wtb^{[r, +\infty]}$ with $x= \sum_{k \geq k_0} p^k[x_k]$, let 
 \[v^{[r,r]}(x): = \frac{pr}{p-1}W^{[r,r]}(x)=\inf_{k \geq k_0}  \{k\frac{pr}{p-1}+  v_{\wt{\mathbf E}}(x_k)\}. \]
and (forcibly) define
\[v^{[r,s]}(x): =  \mathrm{min} \{  v^{[r,r]}(x), v^{[s,s]}(x)\} \]
Clearly, these valuations induce the \emph{same} topologies on  $\wtb^{[r, +\infty]}$ resp.  $\wtb^{[r,s]}$. We have the following observations:
\begin{enumerate}
\item The valuation  $v^{[r,r]}(p)$ depends on $r$.
\item In general, $v^{[r,s]}(x) \neq \inf_{\alpha \in [r,s]} \{v^{[\alpha, \alpha]}(x)\}$.
\item \label{vval3}(Comparison of valuations.)
Let us be very careful here:
\begin{enumerate}
\item For $x \in \wtb^{[r,s]} \subset \wtb^{[r',s']}$ where $[r',s'] \subset [r,s]$, the valuations increase:   $v^{[r,s]}(x) \leq v^{[r',s']}(x)$.
\item For (a general) $y \in \wtb^{[r, +\infty]} \subset \wtb^{[r,s]}$,  (similar as the $W$-valuation case), the valuations \emph{decrease}:  
\[ v^{[r,r]}(y) \geq v^{[r,s]}(y) \] 

\item \label{vval3c}  However, if $y \in \wta^{[r, +\infty]}[1/[\overline{\pi}]] \subset \wtb^{[r, +\infty]} \subset \wtb^{[r,s]}$,  we have (\emph{unlike}  the $W$-valuation case in Cons \ref{cons: W-val}\eqref{item3wvalc}),
\[ v^{[r,r]}(y)= v^{[r,s]}(y). \]
This fact is remarked in the beginning of \cite[\S 5.4]{Col08}, and can be easily verified noticing now  $y= \sum_{k \geq 0} p^k[y_k]$ (i.e., the index $k$ stays \emph{non-negative}).
\end{enumerate}
  \end{enumerate}
 \end{construction}

\begin{example}\label{example: val}
Consider $u=[\underline \pi ] \in \wt{\mathbf{A}}^+$ in Notation \ref{nota:ainf}. Let $[r,s]=[1,2]$. Then
\[ W^{[1,1]}(u) =\frac{p-1}{pe} >W^{[1,2]}(u) =\frac{p-1}{2pe}; \]
\[ v^{[1,1]}(u)   =v^{[1,2]}(u) =\frac{1}{e}. \]
\end{example}

\begin{lemma} \label{prop: verify TS-1} 
 Let $K \subset Y \subsetneq \barK$ such that $\hat{Y}$ is a perfectoid field. Let $r >0$. 
Equip $\wta^{[r, +\infty]}[1/[\overline{\pi}]]$ resp. $\wtb^{[r, s]}$ with the valuation $v^{[r,r]}$ resp. $v^{[r,s]}$ in Construction \ref{cons vrr}.  
 We have
 \begin{enumerate}
\item The pairs 
    \[(\wta^{[r, +\infty]}[1/[\overline{\pi}]], \gal(\barK/Y)),   \quad (\wtb^{[r, s]}, \gal(\barK/Y))
    \]
    satisfy Axiom TS-1 (for any choice of $c_1>0$).

 \item Let $X\in \{ \kpinfty, \kinfty\}$, then the pairs 
    \[ (\wta_L^{[r, +\infty]}[1/[\overline{\pi}]], \gal(L/X)),  \quad (\wtb_L^{[r, s]}, \gal(L/X))      \]
        satisfy Axiom TS-1 (for any choice of $c_1>0$).
        
        \item 
    The triple $$(\wta^{[r, +\infty]}[1/[\overline{\pi}]], \quad  \wtb^{[r, +\infty)}, \quad \gal(\barK/Y))$$  satisfies Axiom TS-1-Fr\'echet.
 
    \item Let $X\in \{ \kpinfty, \kinfty\}$, then the  triple
   \[ (\wta_L^{[r, +\infty]}[1/[\overline{\pi}]], \quad \wtb^{[r, +\infty)}_L, \quad \gal(L/X))   \] satisfies Axiom TS-1-Fr\'echet.
 \end{enumerate} 
 \end{lemma} 
 \begin{proof}
 Consider the pair $(\wta^{[r, +\infty]}[1/[\overline{\pi}]], \gal(\barK/Y))$ in  Item (1).
When  $Y=\kpinfty$, this is proved in  \cite[Lem. 10.1]{Col08}; exactly the same argument works for general $Y$.
(Here, we caution again: as   mentioned in Rem \ref{rem int oc ring}, the ring $\wta^{[r, +\infty]}$ does not satisfy Axiom TS-1, and is not the correct ring to use here.)
The case  for   $\wtb^{[r, s]}$   follows because of the observation in Construction \ref{cons vrr}\eqref{vval3c}.
Item (3) also follows from  the observation in Construction \ref{cons vrr}\eqref{vval3c}, which is the \emph{key} reason we introduce these new valuations.
 Item (2) resp. (4) follows from Item (1) resp. (3).
 \end{proof}


\begin{cor} \label{cor: verify axiom q for rings}
    Axiom \ref{axiomgroupQ} is satisfied for the following rings (where $0 <r \leq s <+\infty$):
    \begin{enumerate}
    \item $\wta_L$;
    \item $\wta_L^{[r, +\infty]}[1/[\overline{\pi}]]$, $\wtb_L^{[r,s]}$;
    \item $\wtb_L^{[r,+\infty)}$;
    \item  $\wta_L^\dagger$, $\wtb^\dagger_{\rig, L}$.
    \end{enumerate} 
\end{cor}
\begin{proof}
Item (1) follows from the well-known TS-1 descent (also called almost purity) for the perfectoid field $C$. Items (2) and (3) are covered by Lem \ref{prop: verify TS-1}. Item (4) follows as they are direct limits of rings above.
 \end{proof}

 \begin{prop} \label{prop: verified TS-1 coho compa}
 We have quasi-isomorphisms of complexes in each item (but they do not compare over different items).
 In each  case, the $C_\tau$-complex is  defined when $p>2$.
     \begin{enumerate}
     \item Let $T\in \rep_\gk(\zp)$. Use Notation \ref{notaetalmod}. Then 
\[ \rg(\gammak, \wtd_\kpinfty) \simeq \rg(\hatg, \wtd_L) \simeq C_\tau(\wtd_\kinfty, \wtd_L) \]

   \item Let $T\in \rep_\gk(\zp)$. Use Notation \ref{notaetalmod}. Then
\[\rg(\gammak, \wtd^\dagger_\kpinfty) \simeq \rg(\hatg, \wtd^\dagger_L) \simeq C_\tau(\wtd^\dagger_\kinfty, \wtd^\dagger_L)\]

    \item \label{itemgpIver} Use diagram \eqref{131rs} in Notation \ref{notarigmod} (where the modules are not necessarily associated to Galois representations.) Then
\[\rg(\gammak,\wtd_{ \kpinfty}^I) \simeq   \rg(\hatg,\wtd_{ L}^I) \simeq C_\tau(\wtd_{ \kinfty}^I,\wtd_{ L}^I)\]

  \item \label{itemgprig}  Use diagram \eqref{131} in Notation \ref{notarigmod} (where the modules are not necessarily associated to Galois representations.) Then
\[\rg(\gammak,\wtd^\dagger_{\rig, \kpinfty}) \simeq   \rg(\hatg,\wtd^\dagger_{\rig,L}) \simeq C_\tau(\wtd^\dagger_{\rig,\kinfty},\wtd^\dagger_{\rig,L})\]

         \end{enumerate}
 \end{prop}
\begin{proof}
    Combine Cor. \ref{cor: verify axiom q for rings} with Thm. \ref{thmaxiomnophi}.
\end{proof}

\section{Axiom: Lie algebra cohomology} \label{sec: axiom lie alg}
In this section, we axiomatically study several complexes related with Lie algebra cohomology of $\Lie \hatg$. We freely use notations in \S \ref{subsec: lie alg LAV}.

\begin{axiom} \label{axiomelec}
Let $Q$ be a LB (resp. LF) ring with a locally analytic (resp. pro-analytic) $\hat G$-action.
Assume there exists $c\in Q^\times$ such that 
\[g(c)=\chi(g)c, \quad \forall g \in \gal(L/\kinfty)\]
Note this implies 
\[\nabla_\gamma(c)=c\]
Define an operator  
\[N_\nabla := \frac{1}{c} \cdot\nablatau \in Q\otimes_\qp \hat{\fkg}\]
We claim that $N_\nabla$ commutes with $\gal(L/\kinfty)$-action and hence commutes with $\nabla_\gamma$. Indeed, for  $g \in \gal(L/\kinfty)$, we have $g \tau g^{-1}=\tau^{\chi(g)}$ (to check this, one could compute actions on $\pi_i$ and use definition of $\tau$ in Notation \ref{nota hatG}); note this is true for any $p$ (the running assumption $p>2$ in Lemmas \ref{lem:delta elt} and \ref{lemgroupiwasawa} is to guarantee $\gamma, \tau$ generate $\hat{G}$ there).
Thus it is easy to see that for  $g \in \gal(L/\kinfty)$, $g \nabla_\tau g^{-1} =\chi(g) \nabla_\tau$, and thus $gN_\nabla g^{-1}=N_\nabla$. (Indeed, this axiomatizes the argument for \cite[Eqn (4.2.5)]{Gao23}).
\end{axiom}

\begin{defn}
    Suppose Axiom \ref{axiomelec} is satisfied. 
       Let $\wtm \in \mod_{\hatg}(Q)$, and suppose $X \subset \wtm$ is a subspace that is stable under $N_\nabla$ action. Define
    \[ C_{N_\nabla}(X):=\quad [X\xrightarrow{\nnabla} X]\]
\end{defn}
 

\begin{prop}\label{propliecohonnabla}
Let $Q$ be a LB (resp. LF) ring with a locally analytic (resp. pro-analytic) $\hat G$-action. Let $\wtm \in \mod_\hatg(Q)$.
\begin{enumerate}
    \item The Lie algebra cohomology $\rg(\hat{\fkg}, \wtm)$ is quasi-isomorphic to
 \[C_{ \nabla_\gamma, \nabla_\tau} (\wtm):=[ \wtm \xrightarrow{ \nabla_\gamma, \nabla_\tau}  \wtm\oplus \wtm   \xrightarrow{\nabla_\tau, 1-\nabla_\gamma}    \wtm ] \]
 
\item Suppose there exists $c \in Q^\times$ as in Axiom \ref{axiomelec}, then $\rg(\hat{\fkg}, \wtm)$ is quasi-isomorphic to
 \[ C_{\nabla_\gamma, N_\nabla}(\wtm):= [ \wtm \xrightarrow{ \nabla_\gamma, N_\nabla}  \wtm\oplus \wtm   \xrightarrow{N_\nabla, -\nabla_\gamma}    \wtm ] \]
\end{enumerate}
     \end{prop}
\begin{proof}
Item (1) is a direct consequence of Lem. \ref{lemresolutionliealg}.
For Item (2), it suffices to note the following diagram induces a quasi-isomorphism (note $c$ is a unit).
\[
\begin{tikzcd}
\wtm \arrow[d, "\mathrm{id}"] \arrow[rr, "{\nabla_\gamma, N_\nabla}"] &  & \wtm \oplus \wtm \arrow[d, "{(\mathrm{id}, c)}"] \arrow[rr, "{N_\nabla, -\nabla_\gamma}"] &  & \wtm \arrow[d, "c"] \\
\wtm \arrow[rr, "{ \nabla_\gamma, \nabla_\tau}"]                      &  & \wtm\oplus \wtm \arrow[rr, "{\nabla_\tau, 1-\nabla_\gamma}"]                             &  & \wtm               
\end{tikzcd}\]
Here, $\mathrm{id}$ stands for identity map, and $c$ stands for multiplication by $c$ map.
\end{proof}

\begin{axiom} \label{axiomLieQ}
Let $Q$ be a LB   ring with a locally analytic  $\hat G$-action.
 Assume:
\begin{enumerate}
    \item $H^1(\mathrm{Lie}~\gal(L/\kpinfty), Q)=0$, and
    \item $H^1(\mathrm{Lie}~\gal(L/\kinfty), Q)=0$.
\end{enumerate}
Equivalently, it says that the Lie algebra operators $\nabla_\gamma, \nabla_\tau$ are both surjective on $Q$.  

Caution: for our application, we should \emph{not} allow $Q$ to be a LF ring with \emph{pro-analytic} action; the general Lie algebra cohomology theory (recalled in Thm. \ref{thm: Tamme ana coho}) only works well (and correct) for \emph{locally analytic}  representations; cf. Rem. \ref{rem: caution pro ana}.
\end{axiom}
 
\begin{theorem}\label{thm:liealgq}
Suppose $Q$ satisfies Axiom  \ref{axiomLieQ} (in particular, $Q$ is a LB ring). 
Let
 \[P:= Q^{\gal(L/\kpinfty)=1}, \quad  R:=  Q^{\gal(L/\kinfty)=1}.\]
 Suppose there exists $N \in \mod_\gammak(P)$ and $(M, \wtm) \in \mod_{\tau}(R, Q)$ that maps to a common object  $\wtm \in \mod_\hatg(Q)$ (that is,  $\wtm=N\otimes_P Q$).
 Then 
 \begin{enumerate}
\item \label{96itme1} the following Lie algebra cohomologies are quasi-isomorphic:
\[ \rg(\hat{\fkg}, \wtm), \quad \rg(\Lie~\gal(L/\kinfty), \wtm^{\nabla_\tau=0}), \]

\item Suppose   there exists $c\in Q^\times$ as in Axiom \ref{axiomelec}. Then
\[ (\rg(\hat{\fkg}, \wtm))^{\gal(L/\kpinfty)=1} \simeq C_\nablagamma(N)\]
\[ (\rg(\hat{\fkg}, \wtm))^{\gal(L/\kinfty)=1} \simeq C_{N_\nabla}(M)\]
 
 \end{enumerate}
\end{theorem}
 \begin{proof}
 Consider Item (1). By Hochschild--Serre spectral sequence, it suffices to show $\nabla_\tau$ is surjective on $\wtm$.  Since  $\wtm=N\otimes_P Q$ and thus $\wtm$ is a ``trivial" $\tau$-representation, we can use Axiom \ref{axiomLieQ} to conclude.
 We also note a similar argument shows $\nabla_\gamma$ is surjective on $\wtm$.
 
 Consider Item (2). The comparison $ (\rg(\hat{\fkg}, \wtm))^{\gal(L/\kinfty)=1} \simeq C_\nablagamma(N)$ is an easy consequence of Thm. \ref{thm: Tamme ana coho} and Hochschild--Serre spectral sequence. Consider the other one.
 By Prop. \ref{propliecohonnabla}, we have 
  \[ \rg(\hat{\fkg}, \wtm) \simeq C_{\nabla_\gamma, N_\nabla}(\wtm)\]
  The right hand side is the total complex of the double complex
  \[\begin{tikzcd}
\wtm \arrow[d, "N_\nabla"] \arrow[r, "\nabla_\gamma"] & \wtm \arrow[d, "N_\nabla"] \\
\wtm \arrow[r, "\nabla_\gamma"]                      & \wtm                      
\end{tikzcd}\]
Since $\nabla_\gamma$ is surjective, thus we obtain a quasi-isomorphism
\[ \rg(\hat{\fkg}, \wtm) \simeq [\wtm^{\nabla_\gamma=0} \xrightarrow{N_\nabla} \wtm^{\nabla_\gamma=0}]\]
Take $\gal(L/\kinfty)$-invariants on both sides. Note $\gal(L/\kinfty)$-action on $\wtm^{\nabla_\gamma=0}$ is smooth and commutes with $N_\nabla$; thus we can take $\gal(L/\kinfty)$-invariants inside the complex. 
Since 
\[(\wtm^{\nabla_\gamma=0})^{\gal(L/\kinfty)}=M,\]
we can conclude.
 \end{proof}
 

\section{Verification: monodromy descent} \label{sec: verify lie}

In this section, we prove Prop. \ref{propverifyaxiommono}, verifying Axiom \ref{axiomelec} and Axiom \ref{axiomLieQ} in this context. 
The techniques in this section are informed by the monodromy descent results in overconvergent $(\phi,\tau)$-modules, cf. \cite{GP21}.

\begin{prop} \label{proppargammasurj}
Suppose $\mathfrak t^{\pm 1} \in \wtb_L^{[r,s]}$ (which holds for $r \gg 0$; cf. Lem. \ref{lem b}).
Then $\nabla_\gamma$ is surjective on $Q=\wtb_L^{[r, s], \hatgla}$.
\end{prop}
\begin{proof}
    This essentially follows from \cite[Thm. 5.3.5]{GP21} (and is also noted in \cite[Cor. 4.2.67]{Poy19}). 
Define
\[\partial_\gamma:=\frac{1}{\fkt}\nabla_\gamma \]
as in \cite[5.3.4]{GP21}. It suffices to prove $\partial_\gamma$ is surjective on $Q$ since $\fkt$ is a unit. Note
\begin{equation}\label{eqpargammat}
    \partial_\gamma(\fkt)=1. 
\end{equation} 
By \cite[Thm. 5.3.5]{GP21}, given an element $x\in Q$, there exists some $\fkt_n \in Q^{\nabla_\gamma=0}$ (that approximates $\fkt$) such that 
\[ x=\sum_{i \geq 0} x_i(\fkt-\fkt_n)^i \]
where $x_i \in Q^{\nabla_\gamma=0}$.
Clearly, a pre-image of $x$ under $\partial_\gamma$ is
\[ \hat{x} =\sum_{i \geq 0} \frac{x_i}{i+1}(\fkt-\fkt_n)^{i+1} \]
Here  the convience of $\partial_\gamma$ (other than $\nabla_\gamma$) is that $\partial_\gamma(\fkt-\fkt_n)^i=i(\fkt-\fkt_n)^{i-1}$.
\end{proof}

\begin{lemma}[\cite{TR12_unpub, Poy19}]
There is an element $b \in \bbb_L^\dagger$ such that $(\tau-1)(b)=1$. 
In addition,  $b \in  \wtb_{\rig,L}^{\dagger, \pa}$.
\end{lemma}
\begin{proof}
The element $b$ is first constructed in \cite[Lem. 3.5]{TR12_unpub}; its analyticity is proved  in Poyeton's thesis \cite[Lem. 4.2.33]{Poy19}. Since these two references are not published papers, we give a summary; in addition, we give a more conceptual reproof of the analyticity (compare with the ``direct"  proof of \cite[Lem. 4.2.33]{Poy19}).

Consider  the 2-dimensional $\Qp$-representation $V$ of $G_K$ (associated to our choice of $\{\pi_n\}_{n \geq 0}$) with a basis $(e_1, e_2)$
 such that \[ g(e_1, e_2)=(e_1, e_2) \smat{1& \frac{c(g)}{\chi(g) } \\ 0 & \frac{1}{\chi(g) }}\]
 where $\chi$ is the $p$-adic cyclotomic character, and $c$ is the cocycle such that $c(g)=0$ for $g \in \gal(\barK/\kinfty)$ and $c(\tau)=1$. Note that the representation factors through $\hatg$-action.

Consider the overconvergent $(\phi, \Gamma)$-module associated to $V$, which is
\[ D= D^\dagger_{\kpinfty}(V) =(V\otimes \bbb^\dagger)^{\gkpinfty}  =(V\otimes \bbb_L^\dagger)^{\gal(L/\kpinfty)}  \]
since all representations are overconvergent by \cite{CC98}, $D$ has dimension two. The fixed point $e_1$ obviously belongs to $D$; another basis element, after scaling coefficient of $e_2$ (note $\bbb^\dagger_\kpinfty$ is a field), is of the form $-be_1+ e_2$ for some \emph{unique} $b \in \bbb_L^\dagger$. Since $-be_1+ e_2$ is fixed by $\tau$, we easily deduce 
\[(\tau-1)(b)=1.\]

We now prove the analyticity of $b$.  We always have
\[ V\otimes_\qp \wtb^\dagger_{\rig} \simeq D\otimes_{\bbb^\dagger_\kpinfty} \wtb^\dagger_{\rig} \]
A key point here is that $V$ is a finite dimensional representation of $\hatg$, thus all elements of $V$ are automatically locally analytic vectors. Thus, take $G_L$-invariant of the above isomorphism and take $\hatgpa$ elements, we have

\[ V\otimes_\qp \wtb_{\rig,L}^{\dagger, \pa} \simeq D\otimes_{\bbb^\dagger_\kpinfty} \wtb_{\rig,L}^{\dagger, \pa}\]
Using unicity of $b$, it is clear $b \in \wtb_{\rig,L}^{\dagger, \pa}$.
\end{proof}

\begin{prop}[\cite{Poy19}] \label{propsurjnablagatau}
    Suppose $b \in  \wtb_L^{[r,s]}$  (that is, $r \gg 0$).
Then $\nabla_\tau$ is surjective on $Q=\wtb_L^{[r, s], \hatgla}$.
\end{prop}
\begin{proof}
This is proved in \cite[Cor. 4.2.65]{Poy19}; we content ourselves by pointing out its analogy with Prop. \ref{proppargammasurj}.
    Clearly 
    \begin{equation}
        \nabla_\tau(b)=1.
    \end{equation}
    This is the analogue of \eqref{eqpargammat}.
 Along similar lines as the proof of Prop. \ref{proppargammasurj},
it suffices to construct a sequence $b_n\in Q^{\nabla_\tau=0}$ (that approximates $b$) such that $x\in Q$ can be written as
\[ x=\sum_{i \geq 0} y_i(b-b_n)^i \]
with $y_i \in Q^{\nabla_\tau=0}$.
Then a pre-image  of $x$ under $\nabla_\tau$ is
\[ \hat{y} =\sum_{i \geq 0} \frac{y_i}{i+1}(b-b_n)^{i+1} \]
The construction of $b_n$ (as well as $y_i$) is quite similar to \cite[Thm. 5.3.5]{GP21}, and is carried out in detail in \cite[Lem. 4.2.62, Thm. 4.2.64]{Poy19}.
\end{proof}

\begin{prop} \label{propverifyaxiommono}
There exists some $\alpha>0$ such that for any $\alpha<r\leq s <+\infty$, the ring
$Q=\wtb_L^{[r, s], \hatgla}$ satisfies  Axiom \ref{axiomelec} and Axiom \ref{axiomLieQ}.
 \end{prop}
\begin{proof} 
    For Axiom \ref{axiomelec}, one can use $\fkt$ (indeed $p\fkt$ or $p^2\fkt$ for normalization purposes, cf. Def.  \ref{defndiffwtb}).
     Axiom \ref{axiomLieQ} is verified in Props. \ref{proppargammasurj} and \ref{propsurjnablagatau}.
\end{proof}



\begin{rem} \label{rem: diffi lpa ring}
    We do not know if $\nabla_\gamma$ or $\nabla_\tau$ is surjective on $ (\wtb_{\rig, L}^\dagger)^\hatgpa $. Indeed, it is not clear if one can prove analogues of  \cite[Thm. 5.3.5]{GP21} for this ring.
\end{rem}

\begin{cor}\label{cor: lie alg coho verified applied}
   Use Notation \ref{notarigmod}. We have
\[ (\rg(\Lie \hatg,  \wtD_{L}^{[r, s], \la})^{\gal(L/\kpinfty)=1} \simeq C_{\nabla_\gamma}(\bbD_{\kpinfty, \infty}^{[r, s]}) \]
\[ (\rg(\Lie \hatg,  \wtD_{L}^{[r, s], \la})^{\gal(L/\kinfty)=1} \simeq C_{\nnabla}(\bfD_{\kinfty, \infty}^{[r, s]}) \] 
\end{cor}
\begin{proof}
    Apply Thm. \ref{thm:liealgq}  in conjunction with Prop. \ref{propverifyaxiommono}. Note the invariant spaces of $\wtD_{L}^{[r, s], \la}$ are computed in Prop. \ref{prop: lav of mod}.
\end{proof}

\section{$\varphi$-cohomologies}
\label{sec: phi coho}

  This section can be regarded as a preparation for the main theorems in the next section \S \ref{sec: coho phi tau}. Here, we separately study $\phi$-cohomologies and their comparisons; this makes many (although not all) comparisons in \S\ref{sec: coho phi tau} much more transparent.
  We also construct some natural yet ``\emph{wrong}" $\phi$-complexes; their discussions rely on the $\psi$-operators.

  Recall the usual notation (for some $Y$ equipped with $\phi$-action):
  \[ C_\phi(Y): = [Y \xrightarrow{\phi-1} Y] \]

 \subsection{$\varphi$-cohomologies}

We first recall a basic lemma.
\begin{lemma} \label{lemphicoho}
Let $R \subset S$ be two rings. Suppose there is a bijective Frobenius map $\phi: S\to S$, such that $\phi(R) \subset R$. Let 
\[\phi^{-\infty}(R) =\cup_{i \geq 0} \phi^{-i}(R) \subset S \]
Let $\mod_\phi(R)$ be the category consisting of finite free $R$-modules $M$ equipped with $\phi_R$ semi-linear $\phi: M \to M$ such that $\det \phi$ is invertible. Define  $\mod_{\varphi}(\varphi^{-m}(R))$ and $\mod_{\varphi}(\varphi^{-\infty}(R))$ similarly.
\begin{enumerate}
    \item The base change functors induce   equivalences
\[\mod_\phi(R) \simeq \mod_{\varphi}(\varphi^{-m}(R)) \simeq  \mod_{\varphi}(\varphi^{-\infty}(R))\]
\item Let $N \in \mod_\phi(R)$ with corresponding $N_m \in \mod_{\varphi}(\varphi^{-m}(R))$ and $N_\infty   \in \mod_{\varphi}(\varphi^{-\infty}(R))$. Then the natural morphisms
\[  C_\varphi(N) \to  C_\varphi(N_m) \to C_\varphi(N_\infty)\]
are quasi-isomorphisms.
\end{enumerate} 
\end{lemma}
\begin{proof} This is well-known and easy to prove. \end{proof}

 



\begin{prop} \label{propphicohoall} 
Let $T\in \rep_\gk(\zp)$. Use Notation \ref{notaetalmod}. 
\begin{enumerate}
    \item  
    $ \rg(\gkpinfty, T) \simeq C_\phi(N) $ for
    \[ N \in \{ \bbD_\kpinfty, 
  \wtd_\kpinfty,  
\wtd_\kpinfty^\dagger  \}, \]
(but not for $\bbD_\kpinfty^\dagger$, cf. Prop. \ref{propwrongphi})

    \item   $ \rg(G_L, T)  \simeq  C_\phi( \wtd_L )  \simeq C_\phi( \wtd_L^\dagger).$

    \item    $ \rg(\gkinfty, T) \simeq C_\phi(M) $ for
    \[ M \in \{ \bfd_\kinfty,  
 \wtd_\kinfty,  
\wtd_\kinfty^\dagger  \}, \]
(but not for $\bfd_\kinfty^\dagger$, cf. Prop. \ref{propwrongphi})
\end{enumerate}
 \end{prop} 
\begin{proof}
We only prove Item (3), the other items are similar.
The quasi-isomorphisms
\[ C_\phi(\wtd_\kinfty^\dagger) \simeq C_\phi(\wtd_\kinfty) \simeq  \rg(\gkinfty, T) \]
follow  from \cite[Thm. 8.6.2]{KL15}, since $\hatkinfty$ is a perfectoid field.
This in particular implies $\rg(\gkinfty, T)$ is concentrated in $[0,1]$. It is easy to check
\[C_\varphi( \bfd_\kinfty ) \simeq   \rg(\gkinfty, T) \]
since they have matching cohomology groups.
 \end{proof}

\begin{prop} \label{propphicohorig}
Let $V\in \rep_\gk(\qp)$. Use Notation \ref{notaetalmod}. 
    Let $\bbD_{\rig,\kpinfty}^\dagger \in  \mod_{\varphi, \gammak}(\bB_{\rig, \kpinfty}^\dagger)$ be the module associated to $V$, and use Notation \ref{notarigmod}.
    \begin{enumerate}
    \item  
    $ \rg(\gkpinfty, V) \simeq C_\phi(N) $ for
    \[ N \in \{ \bbD_\kpinfty, 
  \wtd_\kpinfty,  
\wtd_\kpinfty^\dagger, \wtd_{\rig,\kpinfty}^\dagger \}, \]
(but not for $\bbD_\kpinfty^\dagger, \bbD_{\rig,\kpinfty}^\dagger$, cf. Prop. \ref{propwrongphi})

    \item   $ \rg(G_L, T)  \simeq  C_\phi( \wtd_L )  \simeq C_\phi( \wtd_L^\dagger) \simeq C_\phi( \wtd_{\rig,L}^\dagger).$

    \item    $ \rg(\gkinfty, T) \simeq C_\phi(M) $ for
    \[ M \in \{ \bfd_\kinfty,  
 \wtd_\kinfty,  
\wtd_\kinfty^\dagger, \wtd_{\rig,\kinfty}^\dagger  \}, \]
(but not for $\bfd_\kinfty^\dagger, \bfd_{\rig,\kinfty}^\dagger$,  cf. Prop. \ref{propwrongphi})
\end{enumerate}
\end{prop}
\begin{proof}
    The only new additions from Prop. \ref{propphicohoall}  are those concerning $\varphi$-cohomologies of $\wtd^\dagger_{\rig, X}$ for $X \in \{ \kpinfty, L, \kinfty \}$. The ring    $\wtb^\dagger_{\rig, X}$ satisfies \cite[Hypothesis 1.4.1]{Ked08Ast_relative_Frob} (cf. the final paragraph of \cite[Rem. 2.2.9]{Ked08Ast_relative_Frob}), and thus we can apply \cite[Prop. 1.5.4]{Ked08Ast_relative_Frob} to see
    $ C_\phi(\wtd^\dagger_{X}) \simeq C_\phi(\wtd^\dagger_{\rig, X})$.
  \end{proof}  

\begin{lemma}[``$\varphi$-descent" to closed intervals] \label{lemphidescent}
Use Notation \ref{notation: phi inf twist mod}.
 We have 
\begin{equation}\label{phicoho1}
   C_\phi(M^\dagger) \simeq [M^{[r, +\infty]} \xrightarrow{\phi-1} M^{[pr, +\infty]}] \simeq [M_\infty^{[r, +\infty]} \xrightarrow{\phi-1} M_\infty^{[pr, +\infty]}] 
\end{equation}
and
\begin{equation}\label{phicoho2}
C_\varphi(M^\dagger_\rig) \simeq [ M^{[r, pr]} \xrightarrow{\varphi-1}M^{[pr, pr]} ] \simeq [ M_\infty^{[r, pr]} \xrightarrow{\varphi-1} M_\infty^{[pr, pr]} ] \end{equation}
Caution: for both Items, we are only comparing $\phi$-cohomologies for (minor) modifications of a \emph{fixed} $M^\dagger$ resp. a \emph{fixed}  $M^\dagger_\rig$. Namely, for example we are \emph{not} comparing  $C_\phi(\bbD_\kpinfty^\dagger)$ and $C_\phi(\bfD_\kinfty^\dagger)$ which are (obviously) completely different.
\end{lemma}
\begin{proof} 
Consider \eqref{phicoho1} and \eqref{phicoho2} simultaneously. 
Consider the first quasi-isomorphisms, i.e., comparing $\phi$-cohomology of $M^\dagger$ resp. $M^\dagger_\rig$ with that of $M^{[r, +\infty]}$ resp. $M^{[r, pr]}$.
One can apply \cite[Prop. 6.3.19]{KL15} to modules with tilde, and apply \cite[Prop. 5.4.12]{KL2} to modules without tilde.
We now consider $\phi$-cohomology of 
$ M_\infty^{[r, +\infty]}$ and $ M_\infty^{[r, pr]}$. Since the argument is similar, we only consider the later one.
Indeed, since cohomology commutes with colimits, it suffices to prove that
\[  C_\varphi(M^\dagger_\rig) \simeq  [\phi^{-m}(M^{[p^mr, p^{m+1}r]})   \xrightarrow{\varphi-1}  \phi^{-m}(M^{[p^{m+1}r, p^{m+1}r]}) ]    \]
Lem. \ref{lemphicoho} implies the right hand side is quasi-isomorphic to  $[  M^{[p^mr, p^{m+1}r]}   \xrightarrow{\varphi-1}  M^{[p^{m+1}r, p^{m+1}r]}  ] $, and thus we can conclude using the known first quasi-isomorphism in \eqref{phicoho2}.
\end{proof}

\subsection{$\psi$-complexes and some ``wrong" $\varphi$-complexes} \label{subsecwrongphi}
In this subsection, we discuss some ``natural" $\varphi$-complexes that turn out to be ``wrong" complexes. The key idea is to use some  \emph{$\psi$-complexes.} We will be very brief here, since the main result in this subsection, Prop. \ref{propwrongphi}, is only needed to supply some ``wrong" complexes.

For brevity, we assume \emph{$K/\qp$ is a finite extension}, and stick with (rational) \'etale case: namely, we   let  $V\in \rep_\gk(\qp)$ and use   Notations  \ref{notaetalmod} and \ref{notarigmod}.

\begin{construction} \label{construct: psi op}
We quickly review the ``$\psi$"-operator  in $(\phi, \Gamma)$-module theory and define a similar operator on $(\phi, \tau)$-modules. (For simplicity, we stick with rational case).
\begin{enumerate}
    \item Recall constructions in  \cite[\S 3.1]{Her98}.  Since $\bbb/\phi(\bbb)$ is of degree $p$, one can define  
$\psi_\kpinfty: \bbb \to \bbb $ by
\[ \psi_\kpinfty(x) =\frac{1}{p}\phi^{-1}(\mathrm{Tr}_{\bbb/\phi(\bbb)}(x)) \]
It is (only) additive, and satisfies
\[ \psi_\kpinfty(a\phi(b))=b\psi_\kpinfty(a) \]
It is stable on $\bbb_\kpinfty$ and $ \bbb^\dagger_\kpinfty$.

\item By \cite[Prop. 3.1]{Her98}, one can define 
\[ \psi_\kpinfty: \bbd_\kpinfty \to   \bbd_\kpinfty \]
which is semi-linear with respect to $\psi_\kpinfty$ on  $\bbb_\kpinfty$.
It restricts to
\[ \psi_\kpinfty: \bbd^\dagger_\kpinfty \to   \bbd^\dagger_\kpinfty \]

\item Since $\bfb/\phi(\bfb)$ is of degree $p$, one can similarly define 
$\psi_\kinfty: \bfb \to \bfb$ by
\[ \psi_\kinfty(x) =\frac{1}{p}\phi^{-1}(\mathrm{Tr}_{\bfb/\phi(\bfb)}(x)) \]
It   satisfies exactly the same properties as $ \psi_\kpinfty$, and is   stable on $\bfb_\kinfty$ and $ \bfb^\dagger_\kinfty$.  

\item Using exactly the same recipe as \cite[Prop. 3.1]{Her98}, one can define 
\[ \psi_\kinfty: \bfd_\kinfty \to   \bfd_\kinfty \]
satisfying similar properties as \emph{loc. cit.} 
It restricts to
\[ \psi_\kinfty: \bfd^\dagger_\kinfty \to   \bfd^\dagger_\kinfty \]
\end{enumerate}
Caution: throughout the paper, we use $\phi$ to denote Frobenius on all rings, because all of them are  \emph{induced} (or naturally extended) from a same Frobenius on $\ainf$. However, there is no ``universal $\psi$" that would \emph{induce} both  $\psi_\kpinfty$ and $ \psi_\kinfty$; whence their distinct notations.
\end{construction}

For $X \in \{ \kpinfty, \kinfty\}$, denote
\[ C_{\psi_X}(Y): = [ Y \xrightarrow{\psi_X -1 } Y ] \]

\begin{lemma} \label{lempsicohomatch} 
(Recall  $K/\qp$ is a finite extension in this subsection.) We have quasi-isomorphisms:
   \begin{enumerate}
 \item $C_{\psi_\kpinfty}(\bbd_\kpinfty^\dagger ) \simeq C_{\psi_\kpinfty}(\bbd_\kpinfty)$
       \item    $C_{\psi_\kinfty}(\bfd_\kinfty^\dagger ) \simeq C_{\psi_\kinfty}(\bfd_\kinfty)$
   \end{enumerate}
 \end{lemma}
 \begin{proof}
     Consider Item (1). They have matching $H^0$ by \cite[Prop. III.3.2(2)]{CC99}; they have matching $H^1$ by \cite[Lem. 2.6]{LiuIMRN08}. In fact in \cite[Lem. 2.6]{LiuIMRN08}, he proves the same statement when the modules are associated to a $\zp$-representation; the condition $[K:\qp]<\infty$ guarantees that $\bbd_\kpinfty/(\psi_\kpinfty-1)$ is a \emph{finite} $\zp$-module, which is crucial in his argument. We also point out that the arguments \cite[Prop. III.3.2(2)]{CC99} and \cite[Lem. 2.6]{LiuIMRN08} does not make use of $\gamma$-operators. (Although we have to start with a $(\phi, \Gamma)$-module in order to have the overconvergent module $\bbd_\kpinfty^\dagger$.)

     Item (2) follows from exactly similar argument as Item (1); we leave details to the interested readers. In fact, since the Frobenius ring $\bfb_\kinfty$ is much more explicit and simpler than $\bbb_\kpinfty$, the computation is even easier. Note that the $\tau$-operator is not stable on $\bfb_\kinfty$ (unlike the fact that $\gammak$ even commutes with $\psi_\kpinfty$), and hence cannot  even ``co-exist" with $\psi_\kinfty$ here! Fortunately,  similar to Item (1), the $\tau$-operator   will never show up in the arguments.
 \end{proof}

\begin{prop} \label{propwrongphi}
\hfill 
\begin{enumerate}
\item   We have $C_\varphi(\bbd_\kpinfty^\dagger ) \simeq  C_\varphi(\bbd_{\rig,\kpinfty}^\dagger ) $.  They are (in general) \emph{not} quasi-isomorphic to $C_\varphi(\bbd_\kpinfty)$, which is already the case when  $V=\qp$. 
    \item  We have $C_\varphi(\bfd_\kinfty^\dagger ) \simeq  C_\varphi(\bfd_{\rig,\kinfty}^\dagger ) $.  They are (in general) \emph{not} quasi-isomorphic to $C_\varphi(\bfd_\kinfty)$, which is already the case when  $V=\qp$. 
\end{enumerate}
 \end{prop}
\begin{proof} The two items are similar, we only prove Item (2). 
Since the   pair of rings $\bfb_\kinfty^\dagger \subset \bfb_{\rig,\kinfty}^\dagger$ satisfies \cite[Hypothesis 1.4.1]{Ked08Ast_relative_Frob}, we can apply \cite[Prop. 1.5.4]{Ked08Ast_relative_Frob} to deduce $C_\varphi(\bfd_\kinfty^\dagger) \simeq  C_\varphi(\bfd_{\rig,\kinfty}^\dagger)$.
Now let $V=\qp$.
Let $E=\bfb_\kinfty/\bfb_\kinfty^\dagger$.
To see $C_\varphi(\bfd_\kinfty^\dagger)$ is \emph{not} quasi-isomorphic to $C_\varphi(\bfd_\kinfty)$ in this case, it is equivalent to show 
\[ E \xrightarrow{\varphi-1} E\] is \emph{not} bijective.
Consider the commutative diagram where both columns are short exact:
\[
    \begin{tikzcd}
0 \arrow[d] \arrow[r]                                         & E^{\psi=0} \arrow[d]            \\
E \arrow[d, "\mathrm{id}"', two heads] \arrow[r, "\varphi-1"] & E \arrow[d, "-\psi", two heads] \\
E \arrow[r, "\psi-1"]                                         & E                              
\end{tikzcd} \]
 The bottom row is  bijective by Lem. \ref{lempsicohomatch}. If the second row is also bijective, then $E^{\psi=0}=0$. This is not the case; for example, the following element induces a non-zero element in $E^{\psi=0}=(\bfb_\kinfty)^{\psi=0}/(\bfb_\kinfty^\dagger)^{\psi=0}$:
 \[  x= \sum_{i \leq 0, p \nmid i} p^{\lfloor\log(-i)\rfloor} u^i. \]
 Indeed, the coefficients converges to zero as $i \to -\infty$, but not in any ``linear rate".
\end{proof}

\section{Cohomology of $(\varphi,\tau)$-modules and comparisons} \label{sec: coho phi tau}

In this section, we prove all the main theorems on cohomology comparisons. Some comparisons simply reduce  to comparison of $\phi$-cohomologies, which are fully studied in the previous section \S \ref{sec: phi coho}. However, to obtain other results (particularly the most interesting ones), we need to make crucial use of \emph{group cohomology} and \emph{Lie algebra cohomology}, which were axiomatically studied from \S \ref{sec: axiom gp coho} to \S\ref{sec: verify lie}. 
Indeed, in our approach, we shall first ``break down" the roles of the many operators such as $\phi, \gamma, \tau, \nabla_\gamma, N_\nabla$, and then cohesively bring them together; this makes many cohomology comparisons very \emph{transparent}: many times, we just use \emph{one} operator in each step.

\begin{remark}[Routes of comparisons] \label{rem: route compa} 
As there are \emph{many} complexes, there are then many ways to pair and compare them.
\begin{enumerate}
\item In some (easiest) cases, one can use $\phi$-operator \emph{only} to obtain comparisons: see Step 1 in proof of Thm \ref{thmetalcompa}.

\item In some other cases, one can use ``group-operator" \emph{only} to  obtain comparisons: see Step 2 in proof of Thm \ref{thmetalcompa} and Step 1 in proof of Thm \ref{thmcohononetale}.

\item For the remaining cases, one need to make use of $\phi$-operator \emph{first}, and then apply ``group-operator"  techniques. See Step 2 in proof of Thm \ref{thmcohononetale}; see also Thm \ref{thm: phi Lie algebra}.
\end{enumerate}
In many cases, the roles of $\phi$-operator and group-operator cannot be  ``interchanged" or with their order ``switched" (in Item (3) above); see e.g. Rem \ref{rem: oc phi gamma coho tricky} and the (rather confusing) Rem \ref{rem: rig version phi not enough}.
\end{remark}

\begin{rem}[No $\delta$-functors or $\psi$-operators]
     Many cohomology comparisons on  $(\varphi, \Gamma)$-modules (and $B$-pairs, cf. \S\ref{subsec: B pair}) are well-known in the literature; however, we shall never \emph{use} any of them. Note the comparison theorems in \cite{Her98, TR12_unpub, NakBpair} use $\delta$-functors (and makes rather involved computations), and the comparison theorem in \cite{LiuIMRN08} use $\psi$-operators. 
     We shall \emph{completely avoid $\delta$-functors and $\psi$-operators} (although we shall use $\psi$ to construct some \emph{wrong} complexes in \S \ref{subsec: wrong phi tau complex}). Further comments are given in \S \ref{subsec: final historical remarks}.
\end{rem}

\subsection{Definition of complexes}
\begin{defn} We define complexes used in this section; cf. \S \ref{sec:equiv cats} for the module categories. 
    \begin{enumerate}
        \item 
For $N \in  \mod_{\varphi,\gammak}(P)$, let 
\[ C_{\phi, \gammak}(N): = \rg(\gammak, N)^{\phi=1} \]
\[ C_{\phi, \nabla_\gamma}(N): = \rg(\Lie\gammak, N)^{\phi=1} \]
here the second complex is defined when $\Lie \gammak$ acts on $N$ (e.g., when the relevant $N$ has locally analytic $\gammak$-action); similar for other complexes in the following.

 \item For $\wtm \in \mod_{\varphi, \hatg}(Q)$, let
\[ C_{\phi, \hatg}(\wtm): = \rg(\hatg, \wtm)^{\phi=1} \]
\[ C_{\phi, \nabla_\gamma, \nabla_\tau}(\wtm)   :=\rg(\Lie\hatg, \wtm)^{\phi=1} \]

 \item Let $(M, M_Q) \in \mod_{\varphi, \tau}(R, Q)$. 
 \begin{itemize}
     \item  When $p>2$, define
\[ C_{\phi, \tau}(M, M_Q) =(C_\tau(M, M_Q))^{\phi=1} \]
where $C_\tau(M, M_Q)$ is defined in Def. \ref{defn: C tau Mq}.
\item (For any $p$). When Axiom \ref{axiomelec} is satisfied for $Q$, and let $N_\nabla := \frac{1}{c} \cdot\nablatau$; suppose further $M$ is stable under $N_\nabla$-action, then define 
  \[ C_{\phi, N_\nabla}(M):= \quad [M\xrightarrow{\phi-1, \nnabla} M\oplus M \xrightarrow{-N_\nabla,      \frac{\phi(c)}{c}\phi-1} M] \]     note the appearance of $\frac{\phi(c)}{c}$ makes above sequence a complex. (In this section, $N_\nabla$ is always the operator in Prop. \ref{prop: N nabla operator}).
 \end{itemize}
    \end{enumerate}
\end{defn}

\begin{rem} \label{rem: c phi nnabla homotopy fiber}
    The complex $ C_{\phi, N_\nabla}(M)$ is precisely the $\phi=1$ homotopy fiber of the complex
    \[ C_\nnabla(M)=[ M\xrightarrow{\nnabla} M]\]
    Indeed, $C_\nnabla(M)$ is quasi-isomorphic to the (``$\phi$-equivariant") complex $M \xrightarrow{\nabla_\tau} cM$. We can then form the following diagram which induces a quasi-isomorphism between the two rows.
    \[
    \begin{tikzcd}
M \arrow[d, "\mathrm{id}"] \arrow[rr, "{\phi-1, N_\nabla}"] &  & M\oplus M \arrow[d, "\mathrm{id}\oplus c"] \arrow[rr, "{-N_\nabla,  \frac{\phi(c)}{c}\phi-1}"] &  & M \arrow[d, "c"] \\
M \arrow[rr, "{\phi-1, \nabla_\tau}"]                       &  & M\oplus cM \arrow[rr, "{-\nabla_\tau, \phi-1}"]                                                &  & cM              
\end{tikzcd}
\]
\end{rem}

\begin{convention}\label{convention: phi coho closed interval}
    In this section, we will many times need to ``descend" $\phi$-cohomology to closed intervals as in Lem. \ref{lemphidescent}. For notation simplicity, we shall simply use $C_\phi(M^{[r, pr]})$ to denote 
\[  [ M^{[r, pr]} \xrightarrow{\varphi-1}M^{[pr, pr]} ]  \]
\end{convention}

\subsection{Cohomology comparison:  \'etale and \'etale overconvergent case} 
 
\begin{theorem} \label{thmetalcompa}
Let $T\in \rep_\gk(\zp)$. Use Notation \ref{notaetalmod}. Then $\rg(\gk, T)$ is quasi-isomorphic to all complexes in the following two diagrams (where the two bottom rows are defined only when $p>2$):
    \begin{equation} \label{diagetale}
      \begin{tikzcd}
{C_\phigamma(\bbD_\kpinfty) } \arrow[r] & {C_\phigamma(\wtd_\kpinfty) } \arrow[d] \\
                                                 & {C_{\phi, \hatg}(\wtd_L) }                    \\
{C_\phitau(\bfD_\kinfty,\wtd_L)} \arrow[r]       & {C_\phitau(\wt\bfD_\kinfty,\wtd_L)} \arrow[u]  
\end{tikzcd}  
    \end{equation}
 and
\begin{equation} \label{diagetaleoc}
 \begin{tikzcd}
{C_\phigamma(\bbD_\kpinfty^\dagger) } \arrow[r] & {C_\phigamma(\wtd_\kpinfty^\dagger) } \arrow[d] \\
                                                         & {C_{\phi, \hatg}(\wtd_L^\dagger)}                    \\
                                                         & {C_\phitau(\wt\bfD_\kinfty^\dagger,\wtd_L^\dagger)} \arrow[u]  
\end{tikzcd}   
\end{equation}
(We caution that they are \emph{not} quasi-isomorphic to $C_\phitau(\bfD_\kinfty^\dagger,\wtd_L^\dagger)$, which is ``supposed" to sit in bottom left corner of diagram \eqref{diagetaleoc}; cf.  Prop. \ref{propwrongphitau}.)
\end{theorem}
  
\begin{proof}
\textbf{Step 1.} (arguments using $\phi$-cohomology only).
Prop. \ref{propphicohoall} tells us that  $ \rg(\gkpinfty, T) \simeq C_\phi(N) $ for
    \[ N \in \{ \bbD_\kpinfty, 
  \wtd_\kpinfty,  \wtd_\kpinfty^\dagger  \}, \]
thus by Hochschild--Serre spectral sequence,
\[ \rg(\gk, T) \simeq \rg(\gammak,  \rg(\gkpinfty, T)) \simeq \rg(\gammak, C_\phi(N)) \simeq C_\phigamma(N) \]

Prop. \ref{propphicohoall} also says $C_\phi(\bfD_\kinfty) \simeq C_\phi(\wtD_\kinfty)$; thus (by staring at the associated double complex), we have $$C_\phitau(\bfD_\kinfty,\wtd_L) \simeq C_\phitau(\wtD_\kinfty,\wtd_L).$$

\textbf{Step 2.} (arguments using group cohomology only).
Consider right most column  of diagram \eqref{diagetale} resp. \eqref{diagetaleoc}. These complexes (along each column) are quasi-isomorphic to each other, because they already have the same group cohomologies by Prop. \ref{prop: verified TS-1 coho compa}.


\textbf{Step 3.}  (the (integral) imperfect overconvergent  case.)  
A quick examination shows that we have so far connected all complexes in both diagrams  \eqref{diagetale} and \eqref{diagetaleoc}, except $C_\phigamma(\bbD_\kpinfty^\dagger)$. This is the most tricky complex, cf. Rem. \ref{rem: oc phi gamma coho tricky} for some discussions. 
To prove
\begin{equation} \label{eqn: coho oc not invert p}
    C_\phigamma(\bbD_\kpinfty^\dagger) \simeq \rg(\gk, T),
\end{equation}
we   need to study its rational version first, that is:
\begin{equation} \label{eqn: coho oc invert p}
    C_\phigamma(\bbD_\kpinfty^\dagger[1/p]) \simeq \rg(\gk, T[1/p]).
\end{equation}
Eqn. \eqref{eqn: coho oc invert p} will be proved in the final step in the proof of Thm. \ref{thmcohononetale}. Here we  use Eqn. \eqref{eqn: coho oc invert p}   to prove Eqn. \eqref{eqn: coho oc not invert p}. The readers can check that there is no circular reasoning here.
 Note that $\bbD_\kpinfty^\dagger/p^n \simeq \bbD_\kpinfty/p^n$ for any $n \geq 1$; this implies that we have a short exact sequence (where all maps are induced by obvious inclusions)
\begin{equation*}
    0 \to \bbD_\kpinfty^\dagger \to \bbD_\kpinfty^\dagger[1/p] \oplus \bbD_\kpinfty \to \bbD_\kpinfty[1/p] \to 0
\end{equation*}
Apply $C_\phigamma$ to all the terms above, we obtain a diagram 
\[  0 \to  C_\phigamma(\bbD_\kpinfty^\dagger) \to  C_\phigamma(\bbD_\kpinfty^\dagger[1/p]) \oplus  C_\phigamma(\bbD_\kpinfty) \to  C_\phigamma(\bbD_\kpinfty[1/p]) \to 0;\]
its totalization is acyclic. Consider cohomology; use the fact that
\[ C_\phigamma(\bbD_\kpinfty[1/p]) \simeq \rg(\gk, T[1/p]) \]
which is the rational version of a proven comparison, and use \eqref{eqn: coho oc invert p}, we see that we must have
\[  C_\phigamma(\bbD_\kpinfty^\dagger) \simeq   C_\phigamma(\bbD_\kpinfty)\]
\end{proof}

\begin{rem}\label{rem: oc phi gamma coho tricky}
    We point out that the  argument in Step 1 of proof of Thm. \ref{thmetalcompa} can \emph{not} be used to treat $C_\phigamma(\bbD_\kpinfty^\dagger)$ because $C_\phi( \bbD_\kpinfty^\dagger)$  is a ``wrong" $\phi$-complex by Prop. \ref{propwrongphi}.
We shall treat  $C_\phigamma( \bbD_\kpinfty^\dagger)$ after we deal with  $C_\phigamma( \bbD_{\rig,\kpinfty}^\dagger)$ in Thm. \ref{thmcohononetale} (there will not be any circular reasoning). 
See also Rem. \ref{remhistphigamma}   for some history of these two complexes.
\end{rem}

\begin{remark} \label{remqpver}
Similar statements as Thm. \ref{thmetalcompa} hold if one considers  $V \in \rep_\gk(\qp)$ and related rational modules. Except of course Eqn. \eqref{eqn: coho oc invert p}, which will only be proved in   the final step in the proof of Thm. \ref{thmcohononetale}.
\end{remark}

\begin{rem}
  The complex $C_{\phi, \hatg}(\wtd_L)$ is exactly the 4-term complex in \cite[Thm. 0.2]{TR11} constructed by  Tavares Ribeiro; the complex $C_\phitau(\bfD_\kinfty,\wtd_L)$ is also considered in \cite{Zhao25} by the second named author. In both references, these complexes are compared with $\rg(\gk, T)$ using devissage and $\delta$-functors. We regard the proof in Thm. \ref{thmetalcompa} as a much more conceptual (and complete) one.
\end{rem}

\subsection{Cohomology comparison:  rigid-overconvergent  case}

 \begin{theorem}\label{thmcohononetale} 
 Use Notation \ref{notarigmod}. All the complexes in the following are quasi-isomorphic.
\begin{enumerate}
   \item ($\phi$+group)
     \[
  \begin{tikzcd}
{C_\phigamma(\bbD_{\rig, \kpinfty}^\dagger) } \arrow[r] & {C_\phigamma(\wtd_{\rig, \kpinfty}^\dagger) } \arrow[d] \\
                                                         & {C_{\phi,\hatg}(\wtd_{\rig, L}^\dagger) }                    \\
                                                         & {C_\phitau(\wtd_{\rig, \kinfty}^\dagger,\wtd_{\rig, L}^\dagger)} \arrow[u]  
\end{tikzcd}
\]
where the bottom row is defined only when $p>2$.
(We caution that they are \emph{not} quasi-isomorphic to $C_\phitau(\bfD_{\rig, \kinfty}^\dagger,\wtd_{\rig, L}^\dagger)$, which is ``supposed" to sit in bottom left corner of the diagram; cf. Prop. \ref{propwrongphitau}.)

 \item ($\phi+$Lie algebra, then group invariant). $(C_{\varphi, \nablagamma}(\bbD_{\rig,\kpinfty}^\dagger))^{\gammak} $;    (this works for any  $p$).
    \item (Galois cohomology). $\rg(\gk, V)$ and other (rational) cohomology theories from Thm. \ref{thmetalcompa} (cf. Rem. \ref{remqpver}), if  the modules in \eqref{131} come from a Galois representation $V \in \rep_\gk(\qp)$.
\end{enumerate}
  \end{theorem}
 
\begin{proof} 
\textbf{Step 1.} (Cohomologies of modules with tilde.)  We have
\begin{equation} \label{eqstep1rig}
C_\phigamma(\wtd_{\rig, \kpinfty}^\dagger)  \simeq  C_{\phi,\hatg}(\wtd_{\rig, L}^\dagger) \simeq  C_\phitau(\wtd_{\rig, \kinfty}^\dagger,\wtd_{\rig, L}^\dagger)
\end{equation} 
(where the third complex is defined when $p>2$), by incorporating $\phi$-action to Prop \ref{prop: verified TS-1 coho compa}\eqref{itemgprig}.


\textbf{Step 2.} (Cohomologies of modules over $\kpinfty$-tower.) We now prove: 
\begin{equation}
    C_\phigamma(\bbD^\dagger_{\rig, \kpinfty})     \simeq  C_\phigamma( \wtd^\dagger_{\rig, \kpinfty})  \simeq (C_{\varphi, \nablagamma}(\bbD_{\rig,\kpinfty}^\dagger))^{\gammak}
\end{equation}
Using Lem \ref{lemphidescent} and   Convention \ref{convention: phi coho closed interval}, we have
\[ C_\phi(\bbD^\dagger_{\rig, \kpinfty}) \simeq    C_\phi(\bbD_{\kpinfty, \infty}^{[r, pr]}), \quad \text{ and } C_\phi( \wtd^\dagger_{\rig, \kpinfty}) \simeq  C_\phi(\wtd^{[r, pr]}_{ \kpinfty})\]
Thus it suffices to prove for $I=[r, pr]$ and $[pr, pr]$, we have
\[   \rg(\gammak, \wtd^I_{\kpinfty})  \simeq \rg(\gammak, \bbD_{\kpinfty, \infty}^I) \simeq \rg(\Lie \gammak, \bbD_{\kpinfty, \infty}^I)^{\gammak=1} \]
(Note for later two terms, we have $\infty$ on the subscripts.)
Example \ref{example: porat} informs us that $\wtd^I_{\kpinfty}$ has no higher locally analytic vectors; its locally analytic vectors are precisely $\bbD_{\kpinfty, \infty}^I$ by Prop. \ref{prop: lav of mod}. Thus we can use Thm. \ref{prop: no higher lav coho compa} to conclude.


Note this completes all cohomology comparisons in Item (1) and (2).


\textbf{Step 3.} (Galois cohomology).
To finish the proof of   this entire theorem as well as Step 3  of the proof of Thm. \ref{thmetalcompa}, it remains to prove that when the modules are associated to some $V\in \rep_\gk(\qp)$, we have
\[ C_\phigamma(\bbD_\kpinfty^\dagger) \simeq C_\phigamma(\bbD_{\rig, \kpinfty}^\dagger) \]
and
\[C_\phigamma(\wtd_\kpinfty^\dagger)  \simeq C_\phigamma(\wtd_{\rig, \kpinfty}^\dagger) \simeq \rg(\gk, V)\]
For the first one, it suffices to note that via  Prop. \ref{propwrongphi}, we have 
\[ C_\phi(\bbD_\kpinfty^\dagger) \simeq C_\phi(\bbD_{\rig, \kpinfty}^\dagger)\]
For the second one, it suffices to note  that via  Prop. \ref{propphicohorig}, we have 
\[  C_\phi(\wtd_\kpinfty^\dagger) \simeq C_\phi(\wtd_{\rig, \kpinfty}^\dagger) \simeq \rg(\gkpinfty, V) \] 
 \end{proof}
 
 \begin{remark}\label{rem: rig version phi not enough}  
In this remark, we (very briefly) point out the \emph{necessity} of the TS-1-Fr\'echet Axiom (for $\wtb^{[r, \infty)}$) introduced in Axiom \ref{axiom ts-1-frechet}; namely, the (easier) TS-1 Axiom for the Banach ring $\wtb^{[r,s]}$ is \emph{not} enough for  proof of \eqref{eqstep1rig}. 
(To save space, we leave all the double complex expansions to interested readers). 
 This is prompted by the following observation on proof of \eqref{eqstep1rig}.
By Lem. \ref{lemphidescent} and use Convention \ref{convention: phi coho closed interval}, we have
\[ C_\phigamma(\wtd_{\rig, \kpinfty}^\dagger)  \simeq  C_\phigamma(\wtd_{\kpinfty}^{[r, pr]})\]
\[ C_{\phi,\hatg}(\wtd_{\rig, L}^\dagger) \simeq  C_{\phi,\hatg}(\wtd_{L}^{[r, pr]}) \]
\emph{If} we could prove (which looks \emph{similar} to above formulae):
\begin{equation} \label{eqphitaudes} 
 C_\phitau(\wtd_{\rig, \kinfty}^\dagger,\wtd_{\rig, L}^\dagger) \simeq C_\phitau(\wtd_{ \kinfty}^{[r, pr]}, \wtd_{ L}^{[r, pr]}) 
\end{equation}
we could then use Item \eqref{itemgpIver} of  Prop \ref{prop: verified TS-1 coho compa} to conclude \eqref{eqstep1rig}; this will then only make use of the TS-1 Axiom. However, \emph{a priori}, unlike the above two comparisons, Eqn \eqref{eqphitaudes} does \emph{not}  follow from Lem. \ref{lemphidescent}.
Indeed, by staring at the double complexes (using $\phi$ and $\tau-1$) associated to two sides of \eqref{eqphitaudes}, it reduces to prove
\[ C_\phi(\wtd_{\rig, L}^{\dagger, \delta-\gamma=0}) \simeq C_\phi(\wtd_{ L}^{[r, pr], \delta-\gamma=0}); \]
write out double complexes (using $\phi$ and $\delta-\gamma$) again, one quickly observes that in order to apply  Lem. \ref{lemphidescent}, one needs to know that $\delta-\gamma$ is \emph{surjective} on $\wtd_{\rig, L}^{\dagger}$ (it is so on $\wtd_{ L}^{[r, pr]}$ by Prop \ref{prop: verified TS-1 coho compa}\eqref{itemgpIver}); but then this requires the knowledge of Prop \ref{prop: verified TS-1 coho compa}\eqref{itemgprig}, which builds on the TS-1-Fr\'echet Axiom.
 \end{remark}
 
 \begin{rem} \label{rem: no lpa coho}
   One can also define  $(C_{\varphi, \nablagamma, \nablatau}(\wtd_{\rig,L}^{\dagger,\pa}))^{\hat{G}}$, but it is  not clear if it is quasi-isomorphic to cohomologies in Thm. \ref{thmcohononetale}; an essential difficulty is noted in Rem. \ref{rem: diffi lpa ring}.
  \end{rem}

  \subsection{Cohomology of $B$-pairs} \label{subsec: B pair}
We (very) briefly discuss cohomology of $B$-pairs, using our methods.
Berger \cite{Ber08ANT} proves   that the category of $B$-pairs is equivalent to ${{\mod_{\varphi, \gammak}(\bB_{\rig, \kpinfty}^\dagger)} } $, and thus equivalent to the many categories in Thm. \ref{thmequidagger}. One advantage of $B$-pairs is that its definition only involves $\gk$-actions and no $\phi$-operators, making it a more ``natural" candidate as a generalization of the category $\rep_\gk(\qp)$.  
We refer to \cite{Ber08ANT} for definition of $B$-pairs. Its cohomology theory is defined and studied in  \cite[\S 5]{NakBpair}.
Let $W=(W_e, W_\dr^+)$ be a  $B$-pair. 
Define its cohomology by
\[  \rg(\gk, W): = \rg(\gk,   [W_e \oplus W_\dR^+ \xrightarrow{x, -y} W_\dR] ) \]
Let $\bbD_{\rig,\kpinfty}^\dagger$ be the corresponding $(\phi, \Gamma)$-module.

\begin{theorem} \label{thm: coho B pair}
\cite[Thm. 5.11]{NakBpair}
    Use notations in above paragraph. We have 
 \[ \rg(\gk, W) \simeq  C_{\phigamma}( \bbD^\dagger_{\rig, \kpinfty}) \]
\end{theorem}
\begin{proof}
Nakamura's proof uses $\delta$-functors and makes rather involved computations; we shall give a more conceptual proof along similar ideas in this paper.
    
Indeed, by Thm. \ref{thmcohononetale}, it  is equivalent to prove 
 \[ \rg(\gk, W) \simeq  C_{\phi, \gammak}( \wtD^\dagger_{\rig, \kpinfty}) \simeq \rg(\gk, C_\phi(\wtd^\dagger_{\rig}) ) . \]  
Thus it suffices to prove
\begin{equation}\label{eqphibpair}
    [W_e \to W_\dR/W_\dR^+] \simeq  [\wtd^\dagger_\rig \xrightarrow{\varphi-1} \wtd^\dagger_\rig ]
\end{equation} 
An (algebraic) proof of this is given in \cite[Prop. 3.2, Prop. 3.3]{Berjussieu}; the proof is rather technical and uses slope filtration.
There is another conceptual \emph{geometric} proof using the fact that both sides are quasi-isomorphic to \[\rg(X_{\mathrm{FF}}, \cale) \]
where $X_{\mathrm{FF}}$ is the Fargues--Fontaine curve, and $\cale$ is the vector bundle  on $X_{\mathrm{FF}}$ corresponding to $W$; cf.  \cite[Prop. 5.3.3]{FF18}.

Here we provide another conceptual (and     more elementary) \emph{algebraic}  proof, along similar ideas in this paper.
Consider the diagram 
\[
\begin{tikzcd}
\qp \arrow[d] \arrow[r] & \wtbrig \arrow[r, "\varphi-1"] \arrow[d]        & \wtbrig \arrow[d]        \\
\be \arrow[d] \arrow[r] & {\wtbrig[1/t]} \arrow[d] \arrow[r, "\varphi-1"] & {\wtbrig[1/t]} \arrow[d] \\
\bdr/\bdrplus \arrow[r] & \bdr/\bdrplus \arrow[r]                         & 0                       
\end{tikzcd}\]
here all rows are short exact; cf. for example \cite[Lem. 4.5.3]{KL2} for the top row; cf. \cite[Cor. 1.1.6]{Ber08ANT} for the middle row.
Thus the totalization of the entire diagram is acyclic.
Now the left vertical column is the fundamental short exact sequence. Thus the totalization of the ``right half" of the diagram is also acyclic. 
Suppose $\wtd^\dagger_\rig$ can be descended to $\wtd^{[r, +\infty)}$ for some $r\gg 0$. Note all rings on  the ``right half" are $\wtb^{[r, +\infty)}$-algebras, and thus we can tensor them with $\wtd^{[r, +\infty)}$ (over $\wtb^{[r, +\infty)}$). The totalization of the tensor complex is still acyclic; unravel the totalization and use the formula
\[ W_e= (\wtd^{[r, +\infty)}\otimes_{\wtb^{[r, +\infty)}} \wtbrig[1/t])^{\phi=1},\]
we obtain   the desired quasi-isomorphism \eqref{eqphibpair}.
\end{proof}

\subsection{$\phi$-equivariant Lie algebra cohomologies}
We finally treat the complex  $C_{\varphi, N_\nabla}(\bfD_{\rig,\kinfty}^\dagger)$.

\begin{lemma}
    $C_{\varphi, N_\nabla}(\bfD_{\rig,\kinfty}^\dagger)$ is quasi-isomorphic to the complex 
    \[   \bfD_{ \kinfty}^{[r, pr]} \xrightarrow{\phi-1, N_\nabla} \bfD_{ \kinfty}^{[pr, pr]}\oplus \bfD_{ \kinfty}^{[r, pr]} \xrightarrow{N_\nabla, \frac{pE(u)}{E(0)}\phi-1} \bfD_{ \kinfty}^{[pr, pr]}   \]
\end{lemma}
\begin{proof} By Rem. \ref{rem: c phi nnabla homotopy fiber},
   $C_{\varphi, N_\nabla}(\bfD_{\rig,\kinfty}^\dagger)$ is quasi-isomorphic to the totalization of  
   \[
\begin{tikzcd}
{\bfD_{\rig, \kinfty}^\dagger} \arrow[d, "\nabla_\tau"] \arrow[r, "\phi-1"] & {\bfD_{\rig, \kinfty}^\dagger} \arrow[d, "\nabla_\tau"] \\
{\fkt\bfD_{\rig, \kinfty}^\dagger} \arrow[r, "\phi-1"]                      & {\fkt\bfD_{\rig, \kinfty}^\dagger}                     
\end{tikzcd}
\]
Note in above diagram, since we are working $\qp$-rationally, using $\fkt$ is   equivalent to  using $p\fkt$ or $p^2\fkt$.
The top row is quasi-isomorphic to $C_\phi( \bfD_{ \kinfty}^{[r, pr]})$ by Lem. \ref{lemphidescent}. It remains to prove
\[C_\phi(\fkt\bfD_{\rig,\kinfty}^\dagger ) \simeq C_\phi(\fkt \bfD_{ \kinfty}^{[r, pr]}) \]
The point   is that \cite[Prop. 5.4.12]{KL2} (as applied in our  Lem. \ref{lemphidescent})  is still applicable here, because  $\fkt\bfD_{\rig,\kinfty}^\dagger$ is still a $\phi$-module over $\bfb_{\rig,\kinfty}^\dagger$! Indeed, $\phi(\fkt)=\frac{pE(u)}{E(0)}\fkt$ and $\frac{pE(u)}{E(0)}$ is a unit in  $\bfb_{\rig,\kinfty}^\dagger$.
\end{proof}

\begin{theorem} \label{thm: phi Lie algebra}
Let $\bbD_{\rig, \kpinfty}^\dagger$ be a $(\phi, \Gamma)$-module, and use Notation \ref{notarigmod}.
Let $W$ be the associated $B$-pair.
\begin{enumerate}
    \item  We have a quasi-isomorphism
\[ C_{\phi, \nabla_\gamma}(\bbD_{\rig, \kpinfty}^\dagger)  \simeq  \injlim_n \rg(G_{K(\mu_n)},W)\]
Taking cohomology, we obtain
\begin{equation} \label{eqn: thm gamma lie}
    H^i_{\phi, \nabla_\gamma}(\bbD_{\rig, \kpinfty}^\dagger)   \simeq\injlim_n H^i(G_{K(\mu_n)}, W ) =\bigcup_n H^i(G_{K(\mu_n)}, W )
\end{equation}

 \item  We have a quasi-isomorphism
\[ C_{\phi, \nnabla}(\bfD_{\rig, \kinfty}^\dagger)   \simeq  \injlim_n \rg(G_{K(\pi_n)},W)\]
Taking cohomology, we obtain
\begin{equation}\label{eqn: thm tau lie}
  H^i_{\phi, N_\nabla}(\bfD_{\rig, \kinfty}^\dagger)  \simeq \injlim_n H^i(G_{K(\pi_n)}, W)  =\bigcup_n H^i(G_{K(\pi_n)}, W )  
\end{equation}
\end{enumerate}
 \end{theorem}
 
\begin{proof}
Consider Item (1). We have
   \begin{align*}
  & \    C_{\phi, \nabla_\gamma}(\bbD_{\rig, \kpinfty}^\dagger)    \quad &&  \\
    \simeq &    C_{\phi, \nabla_\gamma}(\bbD_{\kpinfty}^{[r, pr]})    \quad && \text{by Lem. \ref{lemphidescent} and Convention \ref{convention: phi coho closed interval}} \\
      \simeq  & C_\phi(\rg(\Lie \gammak,  \bbD_{ \kpinfty}^{[r, pr]}) ), \quad && \text{re-writing above}   \\
      \simeq  & C_\phi( \injlim_n \rg_\la( \gal(\kpinfty/K(\mu_n), \bbd_{\kpinfty}^{[r, pr]}) )   \quad && \text{by Thm. \ref{thm: Tamme ana coho} }   \\
      \simeq  & C_\phi( \injlim_n \rg_\cont( \gal(\kpinfty/K(\mu_n), \bbd_{\kpinfty}^{[r, pr]}) )     \quad && \text{by Thm. \ref{thm: Tamme ana coho} }\\
       \simeq  &  \injlim_n C_\phi(\rg_\cont( \gal(\kpinfty/K(\mu_n), \bbd_{\kpinfty}^{[r, pr]}) )  \quad && \text{cohomology commutes with colimits} \\
        \simeq  &  \injlim_n C_{\phi, \Gamma_{K(\mu_n)}}( \bbd_{\rig, \kpinfty}^\dagger ) , \quad && \text{by Lem. \ref{lemphidescent}  }  \\
         \simeq  &  \injlim_n \rg(G_{K(\mu_n)},W), \quad && \text{by Thm. \ref{thm: coho B pair} }
   \end{align*}
Taking cohomology, we obtain \eqref{eqn: thm gamma lie}: indeed, the inductive limit of cohomologies of $B$ pairs above is an increasing union because
\begin{equation}
    \rg(G_{K(\mu_n)},W) \simeq  \rg(\gal( K(\mu_{n
+1})/K(\mu_{n})), \rg(G_{K(\mu_{n
+1})},W) 
\end{equation}
and finite group cohomology  on $\qp$-vector spaces is concentrated in degree zero (cf. e.g. \cite[Cor. 16.5]{Hilton_Stammbach_homological_GTM_v2}).
 
Consider Item (2). We have
       \begin{align*}
  & \    C_{\phi, \nnabla}(\bfD_{\rig, \kinfty}^\dagger)    \quad &&  \\
    \simeq &    C_{\phi, \nnabla}(\bfD_{\kinfty}^{[r, pr]})    \quad && \text{by Lem. \ref{lemphidescent}} \\
    \simeq &    C_{\phi, \nnabla}(\bfD_{\kinfty, \infty}^{[r, pr]})    \quad && \text{ by Lem. \ref{lemphicoho} } \\
      \simeq  & C_\phi( (\rg(\Lie \hatg,  \wtD_{L}^{[r, pr], \la})^{\gal(L/\kinfty)=1} ), \quad &&  \text{  by Cor. \ref{cor: lie alg coho verified applied}  }   
  \end{align*}
 For notation simplicity, we first compute above cohomology without taking $\gal(L/\kinfty)=1$. 
Very similar to computations in Item (1), we have
   \begin{align*}
  & \ C_\phi  (\rg(\Lie \hatg,  \wtD_{L}^{[r, pr], \la})) \quad &&  \\
      \simeq  & C_\phi( \injlim_n \rg_\la( \gal(L/K(\pi_n,\mu_n), \wtd_{L}^{[r, pr], \la}) )   \quad && \text{by Thm. \ref{thm: Tamme ana coho} }   \\
      \simeq  & C_\phi( \injlim_n \rg_\cont( \gal(L/K(\pi_n,\mu_n), \wtd_{L}^{[r, pr]}) )     \quad && \text{ by Thm. \ref{thm: Tamme ana coho} }\\
        \simeq  &  \injlim_n C_{\phi, \gal(L/K(\pi_n,\mu_n)) }(  \wtd_{\rig,L}^\dagger ) , \quad && \text{ by Lem. \ref{lemphidescent} }   \\
         \simeq  &  \injlim_n \rg(G_{K(\pi_n,\mu_n)},W), \quad && \text{by Thms. \ref{thmcohononetale}  and \ref{thm: coho B pair} }
   \end{align*}
We can conclude by taking $\gal(L/\kinfty)$-invariants. (Note here, the $\gal(L/\kinfty)$-action on $\rg(G_{K(\pi_n,\mu_n)},W)$ factors through a finite quotient.)
\end{proof}

\begin{rem} \label{rem: artificial tau 1}
    \begin{enumerate}
        \item 
Item (1) of Thm. \ref{thm: phi Lie algebra} is strictly stronger than the comparison in Item (3) of Thm. \ref{thmcohononetale}; since the later follows by taking $\gammak$-invariant of the former.

 \item From the comparison 
\[  H^i_{\phi, N_\nabla}(\bfD_{\rig, \kinfty}^\dagger)  \simeq \bigcup_n H^i(G_{K(\pi_n)}, W )  \]
One can  define
\[ (H^i_{\phi, N_\nabla}(\bfD_{\rig, \kinfty}^\dagger) )^{\tau=1} := H^i_{\phi, N_\nabla}(\bfD_{\rig, \kinfty}^\dagger)  \cap ( \bigcup_n H^i(G_{K(\mu_n, \pi_n)}, W )   )^{\tau=1} \]
And then obviously
\[ (H^i_{\phi, N_\nabla}(\bfD_{\rig, \kinfty}^\dagger) )^{\tau=1} \simeq H^i(\gk, W) \]
We do not expect this very \emph{artificial} construction to be very useful.
    \end{enumerate}
\end{rem}

\subsection{Some ``wrong" $(\varphi, \tau)$-complexes} \label{subsec: wrong phi tau complex}
This subsection is a continuation of \S \ref{subsecwrongphi}.
Recall there we studied some ``natural" $\varphi$-complexes that turn out to be ``wrong" ones; we now show they induce some ``natural" $(\varphi, \tau)$-complexes that also turn out to be ``wrong" complexes (as mentioned in the statements of Thms. \ref{thmetalcompa} and \ref{thmcohononetale}).
Let $K/\qp$ be a finite extension, and consider  $V \in \rep_\qp(\gk)$. Use Notations  \ref{notaetalmod} and \ref{notarigmod}.

\begin{prop} \label{propwrongphitau}
We have  a quasi-isomorphism 
\begin{equation}\label{eqwrongintrocomp}
      C_{\phi, \tau}(\bfD_{\kinfty}^\dagger, \wtD_{L}^\dagger) \simeq C_{\phi, \tau}(\bfD_{\rig, \kinfty}^\dagger, \wtD_{\rig, L}^\dagger);
\end{equation}
but they are (in general) \emph{not} quasi-isomorphic to $C_{\phi, \tau}(\wtD_{\rig, \kinfty}^\dagger, \wtD_{\rig, L}^\dagger)$ (equivalently, to $\rg(\gk, V)$): this is already so when $V=\qp$ is the trivial representation.  (Thus, both   complexes in Eqn. \eqref{eqwrongintrocomp} are \emph{wrong} ones.)
\end{prop}
  \begin{proof}
Note $C_{\varphi, \tau}(M, \wtm)$ is the totalization of the double complex
\begin{equation} \label{eqdoublephitau}
    \begin{tikzcd}
M \arrow[d, "\varphi-1"] \arrow[r, "\tau-1"] & \wtm^{\delta-\gamma=0} \arrow[d, "\varphi-1"] \\
M \arrow[r, "\tau-1"]                        & \wtm^{\delta-\gamma=0}                       
\end{tikzcd}
\end{equation}
Thus to prove
\[ C_{\phi, \tau}(\bfD_{\kinfty}^\dagger, \wtD_{L}^\dagger) \simeq C_{\phi, \tau}(\bfD_{\rig, \kinfty}^\dagger, \wtD_{\rig, L}^\dagger)\]
it suffices to prove
\[ C_\phi(\bfD_{\kinfty}^\dagger) \simeq  C_\phi(\bfD_{\rig, \kinfty}^\dagger) \text{ and } C_\phi(\wtD_{L}^\dagger) \simeq  C_\phi(\wtD_{\rig, L}^\dagger)  \]
these follow from \cite[Prop. 1.5.4]{Ked08Ast_relative_Frob}.

We now prove $C_{\phi, \tau}(\bfD_{\rig, \kinfty}^\dagger, \wtD_{\rig, L}^\dagger)$ is \emph{not} quasi-isomorphic to $C_{\phi, \tau}(\wtD_{\rig, \kinfty}^\dagger, \wtD_{\rig, L}^\dagger)$.
Staring at their associated double complexes   \eqref{eqdoublephitau}, it suffices to note that $C_\varphi(\bfD_{\rig, \kinfty}^\dagger)$ and $C_\varphi(\wtD_{\rig, \kinfty}^\dagger)$ are \emph{not} quasi-isomorphic, this follows from Prop \ref{propphicohorig}(3) and   Prop. \ref{propwrongphi}.
 \end{proof}

 \subsection{Final remarks} \label{subsec: final historical remarks}
In this subsection, we make some extensive   historical comparisons on the study of cohomology of $(\varphi, \Gamma)$-modules; most interestingly, we comment how we avoided using $\psi$-operator in the current paper (which leads to slightly stronger new results). Nonetheless, in the future, we  very much hope to gain some understanding of the very mysterious (indeed, very baffling) $\psi$-operator for the  $(\varphi, \tau)$-modules.

\begin{rem} \label{remhistphigamma}
We discuss the (rather involved)  history of Herr complexes of (various) $(\varphi, \Gamma)$-modules when they are associated to Galois representations. For simplicity,  we only discuss results concerning $V \in \rep_\gk(\qp)$.
    \begin{enumerate}
 \item In \cite{Her98}, Herr proves that
 \[C_\phigamma(\bbd_\kpinfty) \simeq \rg(\gk, V)\]
His results indeed hold for $V$ a torsion representation or integral $\zp$-representation; in fact, his proof goes via devissage to the torsion case, where he proves the construction of the complex $C_\phigamma(\bbd_\kpinfty)$ is a $\delta$-functor.

 \item The complexes $C_\phigamma(\bbD_\kpinfty^\dagger)$ and $C_\phigamma(\bbD_{\rig, \kpinfty}^\dagger)$ are \emph{much} trickier. 
In \cite{LiuIMRN08}, under the assumption $K/\qp$ is a \emph{finite extension}, Liu proves their comparison with $ \rg(\gk, V)$. His argument goes in the following route (slightly modified and simplified here): 
\begin{itemize}
    \item Liu first deals with $C_\phigamma(\bbD_\kpinfty^\dagger)$. 
    One has comparison
    \[ C_\phigamma(\bbd_\kpinfty^\dagger) \simeq C_{\psi, \gammak}(\bbd_\kpinfty^\dagger); \quad C_\phigamma(\bbd_\kpinfty) \simeq C_{\psi, \gammak}(\bbd_\kpinfty);\]
    these are proved by \cite{CC98}; cf. also \cite[Lem. 2.5]{LiuIMRN08} for a review and summary. Liu then proves \[C_{\psi}(\bbd_\kpinfty^\dagger) \simeq C_{\psi}(\bbd_\kpinfty)\]
    cf.  \cite[Lem. 2.6]{LiuIMRN08}, which makes essential use that $K/\qp$ is a finite extension (the authors  do not know if this still holds when $K/\qp$ is not a finite extension). 
    This implies \[C_{\psi, \gammak}(\bbd_\kpinfty^\dagger) \simeq C_{\psi, \gammak}(\bbd_\kpinfty).\]

\item It remains to compare $C_\phigamma(\bbD_\kpinfty^\dagger)$ and $C_\phigamma(\bbD_{\rig, \kpinfty}^\dagger)$; this is implied by   $C_\phi(\bbD_\kpinfty^\dagger) \simeq C_\phi(\bbD_{\rig, \kpinfty}^\dagger)$, as proved in \cite[Prop. 1.5.4]{Ked08Ast_relative_Frob}.
\end{itemize}
      \end{enumerate}
\end{rem}

\begin{rem} \label{remfinalourmethodphigamma} We compare our methods with those in Rem. \ref{remhistphigamma}.
\begin{enumerate}
  \item The complexes \[\text{ $C_\phigamma(\wtd_\kpinfty)$, $C_\phigamma(\wtd_\kpinfty^\dagger)$,  $C_\phigamma(\wtd_{\rig, \kpinfty}^\dagger)$ }\]are not often considered in the literature (to the knowledge of the authors). As we see in Thms \ref{thmetalcompa} and \ref{thmcohononetale}, it is easy to show that they are quasi-isomorphic to $ \rg(\gk, V)$; in addition, this leads to a rather easy (and conceptual) reproof of comparison with $C_\phigamma(\bbd_\kpinfty)$. The proofs of all these results essentially only make use of $\varphi$-operators.

  \item  As recalled  in  Rem. \ref{remhistphigamma}(2) above,  Liu  first treated $C_\phigamma(\bbD_\kpinfty^\dagger)$ using $\psi$-operator (where the proof requires  $[K:\qp]<\infty$), then he compares $C_\phigamma(\bbD_{\rig, \kpinfty}^\dagger)$ using $\phi$-operator. 
  We go in a \emph{different} route. We first treat $C_\phigamma(\bbD_{\rig, \kpinfty}^\dagger)$ in Step 2 of proof of Thm. \ref{thmcohononetale},  using \emph{locally analytic vectors} (which works for all $K$); we then compare with $C_\phigamma(\bbD_\kpinfty^\dagger)$ in Step 4 there,  using $\phi$-operator. 
\end{enumerate}
\end{rem}

\begin{rem} We comment on the $\psi$-operators in Construction \ref{construct: psi op}.
\begin{enumerate}
    \item In the study of $(\phi, \Gamma)$-modules, the $\psi$-operator ($\psi_\kpinfty$ in Construction \ref{construct: psi op}) has played a \emph{very} significant role.
    \begin{itemize}
            \item A first application of $\psi$-operator is in \cite{Her98, Her01}, where he proves finiteness and duality for cohomology of  $(\phi, \Gamma)$-modules, \emph{without} using its relation with Galois cohomology. For example, a key player in the theory is the \emph{heart}, cf. \cite[\S 3.4]{Her98} 
            \begin{equation} \label{eqcoeur}
                (\phi-1)(\bbD_\kpinfty) \cap (\bbD_\kpinfty)^{\psi=0}
            \end{equation}  
            
        \item In the first proof of overconvergence of $(\phi, \Gamma)$-modules, \cite{CC98} makes use of $\psi$-operators; however, as shown in \cite{BC08}, it is indeed unnecessary. Rather, overconvergence follows from a general formalism of \emph{Sen theory}, as axiomatized in \cite{BC08} and later further in \cite{BC16}. 

        \item The $\psi$-operator is a key player in Iwasawa theory. By \cite[Thm. II.1.3, Rem. II.3.2]{CC99}, we have
\[ (\bbd_\kpinfty^\dagger)^{\psi=1} = (\bbd_\kpinfty)^{\psi=1} \simeq H^1_{\mathrm{Iw}}(K, V) \]
and 
\[  \bbd_\kpinfty^\dagger/(\psi-1) \simeq  \bbd_\kpinfty/(\psi-1) \simeq H^2_{\mathrm{Iw}}(K, V) \]
where $H^i_{\mathrm{Iw}}$ is the Iwasawa cohomology.
    \end{itemize}
    \item Moving to the world of  $(\phi, \tau)$-modules, our understanding of the $\psi$-operator ($\psi_\kinfty$ in Construction \ref{construct: psi op}) is very limited. We   use this opportunity to pose some questions, hoping to come back  in future investigations.
    \begin{itemize}
        \item Can we develop some $(\psi, \tau)$-cohomology theory? If one uses the category $\mod_{\phi, \hatg}(\bfa_\kinfty, \wta_L)$, one quickly finds the problem that $\psi$ is only well-defined on \emph{imperfect} rings and hence is problematic on the ring $\wta_L$. Furthermore, say let us first only consider $\bfd_\kinfty$ over the imperfect ring $\bfa_\kinfty$, and thus one can form  the heart as in \eqref{eqcoeur}; however practically all methods in \cite{Her98} cannot be applied here, essentially because $\kinfty/K$ is not a Galois tower.
        \item  For the proof of overconvergence of $(\phi, \tau)$-modules in \cite{GL20, GP21},    the  $\psi$-operator is never used. In particular, the proof of \cite{GP21} makes use of the axiomatic (locally analytic) Sen theory developed by \cite{BC08, BC16}.
        \item Does $\psi$ on $\bfd_\kinfty$ lead  to any ``Iwasawa theory"? Note since $\kinfty/K$ is not Galois, it is not even clear what should be the correct ``Iwasawa algebra" here. Put in another way, we do not know if $\bfd_\kinfty^{\psi=1}$  has any interesting module structure.
    \end{itemize}
\end{enumerate}
    
\end{rem}


 \end{document}